\documentclass{aims_hal}

\usepackage{amsmath}
\usepackage{amssymb,amsthm}
\usepackage{amscd,amsbsy}
\usepackage{bbm}
\usepackage{mathrsfs}
\usepackage{mathtools}
\usepackage{textgreek}
\usepackage{upgreek}
\usepackage{amsbsy}
\usepackage{cancel}
\newcommand{\dst}{\displaystyle}

\usepackage{mathbbol}
\DeclareSymbolFontAlphabet{\amsmathbb}{AMSb}

\usepackage{empheq}
\newcommand*\shadedbox[1]{%
\colorbox{shadebrushcolor}{\hspace{1em}#1\hspace{1em}}}
{\endEmphEqMainEnv}

\usepackage{enumitem}

 \usepackage[colorlinks=true]{hyperref}
\hypersetup{urlcolor=blue, citecolor=red}

\def\currentvolume{}
  \def\currentyear{}
   \def\currentmonth{}
    \def\ppages{}

\newcommand{\solEps}				{u^\varepsilon}
\newcommand{\solShift}				{\tilde{u}}
\newcommand{\modelOpShift}			{\tilde{A}} 
\newcommand{\modelFormShift}			{\tilde{a}} 

\newtheorem{theorem}{Theorem}[section]

\newtheorem{lemma}[theorem]{Lemma}
\newtheorem{proposition}{Proposition}[section]

\theoremstyle{definition}

\newtheorem{remark}{Remark}

\newcommand{\ep}{\varepsilon}
\newcommand{\f}{\frac}
\usepackage{standalone}
\RequirePackage{tikz}
\usetikzlibrary{shapes,arrows,shadows,patterns,pgfplots.groupplots,decorations.markings,chains,calc}
\usepackage{pgfplots}

\definecolor{mygreen}{rgb}{0,0,0}
\definecolor{myorange}{rgb}{0,0,0}
\definecolor{myred}{rgb}{0,0,0}
\definecolor{myblue}{rgb}{0,0,0}

\usepackage[normalem]{ulem}

\usepackage{marvosym}
\usepackage{manfnt}

\newcommand{\inserted}[1]{{\color{black}#1}} 

\def\t{\tilde}
\newcommand{\TM}{T}
\newcommand{\VI}{\mathcal{V}_\infty}
\newcommand{\tT}{\t\TM}
\newcommand{\LM}{L}
\newcommand{\uin}{\sol^0}
\newcommand{\tO}{\t\Omega}
\newcommand{\norme}[1]{ {\big\lVert  #1\big\rVert}}
\newcommand{\initNoise}				{\zeta}

\DeclareMathOperator*{\argmin}{argmin}

\newcommand{\dsp}					{\displaystyle}
\newcommand{\R}						{\mathbbm{R}}

\newcommand{\N}						{\mathbbm{N}}
\newcommand{\Mtot}{M_{\mathrm{tot}}}

\newcommand{\eps}					{\varepsilon}
\newcommand{\vareps}				{\varepsilon}
\newcommand{\abs}[1]				{\left| #1 \right|}
\newcommand{\norm}[1]				{\left\|#1\right\|}  
\newcommand{\diff}					{\mathrm{d}}     
 
\newcommand{\deriv}[1]				{\frac{\mathrm{d}}{\mathrm{d}#1}}
\newcommand{\Cst}					{C^{\text{st}}}
\newcommand{\p}						{\partial}
\newcommand{\sol}					{u}
\newcommand{\solF}					{u}
\newcommand{\solFinf}				{u_{\infty}}
\newcommand{\solS}					{u^{\varepsilon}}
\newcommand{\solSinf}				{u^{\varepsilon}_{\infty}}
\newcommand{\Ltwo}[1][0,L]{\mathrm{L}^2(#1)}
\newcommand{\Hone}[1][0,L]{\mathrm{H}^1(#1)}
\newcommand{\HoneR}[1][0,L]{\mathrm{H}^1_R(#1)}

\newcommand{\Htwo}[1][0,L]{\mathrm{H}^2(#1)}
\newcommand{\Hthree}[1][0,L]{\mathrm{H}^3(#1)}

\newcommand{\Czero}[1][0,T]{\mathrm{C}^0(#1)}
\newcommand{\Cone}[1][0,T]{\mathrm{C}^1(#1)}

\newcommand{\state}					{z}
\newcommand{\stateDof}					{\mathrm{\state}}

\newcommand{\stateI}				{z_\infty}
\newcommand{\observ}				{y}
\newcommand{\adjoint}				{q}

\newcommand{\stateSpaceI}				{\mathcal{Z_\infty}}                                	
\newcommand{\stateSpace}				{\mathcal{Z}}

\newcommand{\Id}					{\mathrm{Id}}                              	
\newcommand{\modelOp}				{A} 
                         	
\newcommand{\modelOpI}				{A_\infty} 
\newcommand{\obsOp}					{C}
\newcommand{\obsOpDof}				{\mathrm{\obsOp}}

\newcommand{\initState}				{\zeta}                                	
                                	
\newcommand{\criter}				{\mathscr{J}}

\newcommand{\T}{T}
\newcommand{\cov}{\varPi}
\newcommand{\covDof}{\Pi}

\title[Asymptotics in inverse problems for depolymerization] 
      {Asymptotic approaches in inverse problems for depolymerization estimation}

\author[Marie Doumic and Philippe Moireau]{}


 \email{marie.doumic@inria.fr}
 \email{philippe.moireau@inria.fr}
\date{\today}

\thanks{$^*$ Corresponding author: Philippe Moireau}

\begin{document}
\begin{abstract}

Depolymerization reactions constitute frequent experiments, for instance in biochemistry for the study of amyloid fibrils. The quantities experimentally observed are related to the time dynamics of a quantity averaged over all polymer sizes, such as the total polymerised mass or the mean size of particles. The question analysed here is to link this measurement to the initial size distribution. To do so, we first derive, from the initial reaction system
\[ \mathcal{C}_i \stackrel{b}{\longrightarrow} \mathcal{C}_{i-1} + \mathcal{C}_1,\qquad i\geq i_0\geq 2,\]
two asymptotic models: at first order, a backward transport equation, and at second order, an advection-diffusion/Fokker-Planck equation complemented with a mixed boundary condition at $x=0$. We estimate their distance to the original system solution. We then turn to the inverse problem, {\it i.e.}, how to estimate the initial size distribution from the time measurement of an average quantity, given by a moment of the solution. This question has been  studied in~\cite{armiento:2016} for the first order asymptotic model, and we analyse here the second order asymptotic. Thanks to Carleman inequalities and to log-convexity estimates, we prove observability results and error estimates for a Tikhonov regularization. We then develop a Kalman-based observer approach, and implement it on simulated observations. Despite its severely ill-posed character, the second-order approach appears numerically  more accurate than the first-order one.

\end{abstract}

\maketitle

\medskip
\centerline{\scshape Marie Doumic}
\medskip
{\footnotesize
\centerline{Inria \& Sorbonne Universit\'e}
\centerline{Team MAMBA, Inria Paris, 2 rue Simone Iff}
\centerline{Paris, 75012, France}
} 

\medskip

\centerline{\scshape Philippe Moireau$^*$}
\medskip
{\footnotesize
\centerline{Inria -- LMS,Ecole Polytechnique, CNRS, Institut Polytechnique de Paris}
\centerline{Team M$\sf{\Xi}$DISIM, Inria Saclay - Ile-de-France, Turing Building}
\centerline{Palaiseau, 91128, France}
}

\bigskip


\centerline{\today}


\tableofcontents


\section{Introduction}

The problem under study has been first inspired by experiments carried out for the PrP protein, responsible of Creutzfeldt-Jakob's disease, where depolymerisation has been isolated from other possible reactions in order to study the (in)stability of the polymers~\cite{armiento:2016}. This kind of experiment may be encountered in many other application fields, as for instance actin filaments~\cite{Bendes_2022}, lyzosyme amyloids~\cite{Bellova_2010}, thermoplastics~\cite{Van_2007}, polyethylene-terephthalate (PET) depolymerisation  for its recycling~\cite{Karpati_2019}, etc. The experimental protocol considered here consists in subjecting polymers of different sizes  to depolymerisation conditions - for instance by stirring - and observing their size decay through the temporal measurement of averaged quantities, such as the total polymerised mass (related to the first-order moment of the solution) or the average size of polymers (related to the second-order moment of the solution). In this article, we address the issue of estimating the initial size distribution from such measurements.

\subsection{Departure point: the depolymerisation model}

Let us denote  $\mathcal{C}_i$ the  polymers containing $i$ monomers and $\mathcal C$ the monomers. Depolymerisation is a chain reaction modelled by  the following reaction system
\begin{equation}\label{eq:depol}
\mathcal{C}_i \stackrel{b}{\longrightarrow} \mathcal{C}_{i-1} + \mathcal{C},\qquad i\geq i_0+1\geq 2,
\end{equation}
where the depolymerisation rate $b>0$ is assumed to be constant, and $i_0$ the smallest polymer size. We denote $C_i(\tau)$ the concentration of polymers containing $i$ monomers at time $\tau$. The system is closed: no other reaction, no creation or degradation of polymers, so that assuming that  the initial concentrations satisfy $C(0)+\sum\limits_{i=i_0}^\infty i C_i(0)=\Mtot >0,$ this quantity is conserved over time.

Applying the mass balance equation,  we obtain the following infinite dimensional system 
\begin{equation}\label{eq:depol-ode}
	\deriv{\tau} C_i = b(C_{i+1} - C_i),\qquad i\geq i_0+1, \qquad \deriv{\tau} C_{i_0}=bC_{i_0+1},\qquad \deriv{\tau} C = \sum\limits_{i=i_0+1}^\infty b C_i.
\end{equation}
The experimental protocol is then modeled as the (noisy) measurement of
\[\tau \mapsto \sum\limits_{i=i_0}^\infty i^k C_i(\tau)
\]
on the time window $\tau \in [0,{\mathcal T}],$ for a given $k\in \N.$ Observing the polymerised mass corresponds to $k=1$~\cite{batzli2015agitation}, whereas light scattering methods are modeled by $k=2$~\cite{some2013light}.

{\inserted{Our depolymerisation model may be seen as a specific case of the Becker-D\"oring system, which describes polymerisation-depolymerisation chain reactions~\cite{BeckerDoring_1935}. A specificity of many polymeric systems is that, in average, polymers contain a very high number of monomers, leading to formally define 
\[\ep\coloneqq\f{1}{i_M}\coloneqq\f{\sum\limits_{i=i_0}^\infty C_i(0)}{\sum\limits_{i=i_0}^\infty i C_i(0)},
\]
where $i_M\gg 1$ denotes the initial average polymer size. It is thus natural to look for continuous approximations of~\eqref{eq:depol-ode}, despite the fact that in our case the average polymer size decays with time.

\subsection{Continuous approximations of the depolymerisation model}

Many studies have thus been carried out to derive approximate size-continuous models for the Becker-D\"oring system in the limit $\ep\to 0$, deriving various boundary conditions according to the physics considered. At first order, the so-called {\it Lifshitz-Slyozov system}~\cite{LifshitzSlyozov_1961} is obtained, which consists in the nonlinear coupling of a continuous transport equation with the mass conservation equation. This first-order system does not require any boundary condition for well-posedness if the flux around $x=0$ is backward, {\it i.e.} if depolymerisation dominates polymerisation, as in our case. This is the most studied case, since it is used to model gelation phenomena such as Ostwald ripening or Ouzo effect. Weak convergence results in measure spaces have been obtained in this case, see~\cite{ColletGoudonPoupaudVasseur_2002,LaurencotMischler_2002,Schlichting_2019}. To take into account nucleation, {\it i.e.} the spontaneous creation of polymers out of monomers, the Lifshitz-Slyozov system complemented by  a non-homogeneous Dirichlet boundary condition has been studied more recently in~\cite{Calvo_2018,deschamps2017quasi,CalvoHingantYvinec_2021}.}

At second order, the system~\eqref{eq:depol-ode} may be approximated by an advection-diffusion / Fokker-Planck equation, which needs to be complemented by a boundary condition at $x=0.$ Several types of boundary conditions have been proposed, depending on the problem considered. In~\cite{Conlon_2010,Velazquez_1998}, a homogeneous Dirichlet boundary condition is considered, and its long-time asymptotics is analysed in~\cite{ConlonDabkowski_2022}. In~\cite{ConlonSchlichting_2019}, a non-homogeneous Dirichlet boundary condition is imposed, 
derived from the equilibrium solution of the Fokker-Planck equation. Another non-homogeneous boundary condition is proposed 
in~\cite{ColletGoudonPoupaudVasseur_2002}, linking $\solS$ to the concentration of monomers. 
This condition is obtained in order to have mass conservation of the full system, when polymerisation is also considered, see also~\cite{Goudon2019}.  The asymptotics is no more valid in our case, {\it i.e.} when polymerisation vanishes; we thus propose another boundary condition, that we call {\it absorbing}  since it is designed in order to remain as close as possible to the original system~\eqref{eq:depol-ode}, and in reference to the field of artificial boundary conditions~\cite{halpern1986artificial}.

\subsection{Outline of the article and main results}

Our case is simpler than the full Becker-D\"oring system for two reasons: we only consider constant depolymerisation rates $b_i\equiv b$ for all $i\geq i_0,$ and the absence of polymerisation makes our system linear with a constant outflux. We are thus able to derive two error estimates for the approximate systems. First, we derive an $L^2$ error estimate between the discrete system~\eqref{eq:depol-ode} and the first-order approximate system~\eqref{eq:first-order} (Prop.~\ref{prop:first-order:estim}). This estimate reveals to be in the order of $\ep$. Second, we derive an $L^2$ error estimate between~\eqref{eq:depol-ode} and a second-order approximate system~\eqref{eq:second-order} complemented with an {\it absorbing} boundary condition (Prop.~\ref{prop:second-order:estim}), which reveals to be in the order of $\ep^{3/2}.$ This improved error estimate shows that
the absorbing boundary condition in~\eqref{eq:second-order} is well-adapted to our purely depolymerising system. Up to our knowledge, it has not been proposed before, and no quantitative error estimates have been obtained yet for systems including polymerisation.

Once equipped with these two approximate models, which are more easily tractable numerically and theoretically than~\eqref{eq:depol-ode}, we turn in Section~\ref{sec:inverse} to settle the main purpose of our paper: 
How to estimate the initial size distribution from moments measurements? This inverse problem has already been solved when using the first-order approximate system~\eqref{eq:first-order}:  we only recall the main results of the article~\cite{armiento:2016} in Section~\ref{subsec:inverse:advection}. In Section~\ref{subsec:inverse:advection-diffusion}, we explain in detail the inverse problem setting for the second-order approximate system~\eqref{eq:second-order}, that we set in a finite domain $[0,L]$ in~\eqref{eq:second-order-L} and, by using a similar set of differential equations as for the first-order system, reduce to the following inverse problem: How to estimate the initial condition not from the time-dynamics of a moment of the solution, but from the time-dynamics of the boundary value $t\mapsto \sol(t,0)$? Finally, Subsection~\ref{subsec:expostab} is dedicated to the preliminary analysis and two useful inequalities on the direct problem~\eqref{eq:second-order-L}. 

Section~\ref{sec:inverse:anal} provides the inverse problem solution. Our reformulation in terms of  observation at the boundary makes our problem close to controllability issues, studied in similar systems in~\cite{Coron:2005vd,Cornilleau:2012dba}. To obtain an observability result, we thus
use two main inequalities: a Carleman-type inequality (Theorems~\ref{coroll:Carleman} and~\ref{th:carleman}) on the one-hand, inspired by~\cite{Cornilleau:2012dba}, and a log-convexity argument (Proposition~\ref{prop:logconv} in Section~\ref{subsec:expostab}), inspired by~\cite{Bardos:1973aa}. These two inequalities lead us to our main observability result, stated in Theorem~\ref{th:initobserv}. In turn, this result leads us to propose a generalised Tikhonov regularisation, for which we prove an error estimate in Theorem~\ref{th:tikhonov}.

\inserted{\subsection{Some notation considerations}
In what follows, we  first study the direct problems over the whole space (Section~\ref{sec:direct}). When doing so we use the index $\cdot_\infty,$ for instance $\solFinf$ and $\solSinf$ respectively for the first-order and second-order solution, $\stateSpaceI$ for the state space with $x\in (0,\infty),$ etc. We then turn to the problem over an interval $[0,\LM],$ more convenient both to obtain certain estimates and to be closer to the numerical solution. In such cases, we simply drop the indices $\infty,$ considering for instance $\solEps,$ $\modelOp$, $\stateSpace$ etc. The link between the problem in a finite and in an infinite state space, and our modelling choices, are explained in Section~\ref{subsec:inverse:advection-diffusion}. In Section~\ref{sec:inverse:anal}, we drop the indices $\cdot^{\ep,1}$ for simplicity.}

\section{Asymptotic models and error estimates}
\label{sec:direct}

As a first step in our inverse problem setting, let us study the direct problems linked to two successive approximations of~\eqref{eq:depol-ode} in the limit $\ep\to 0$. The first order approximation is a backward transport equation, studied in Section~\ref{subsec:direct:advection}, and the second order is an advection-diffusion or Fokker-Planck equation, addressed in Section~\ref{subsec:direct:advection-diffusion}. \inserted{We illustrate these estimates by numerical simulations in Section~\ref{sec:num}.}\subsection{First order asymptotics: a pure advection model}
\label{subsec:direct:advection}
We aim at formulating an asymptotic model of \eqref{eq:depol-ode} in the large size limit, {\it i.e.} when $\ep \to 0$. Introducing the change of variables $t= \tau \vareps$ with $t \in [0,T]$ and $T=\ep \mathcal{T}$ a finite time, we have for $c_i^\vareps: t \mapsto C_i(t/\vareps)$,
\begin{equation}
\label{eq:BD}
\deriv{t}c_i^\vareps=\frac{b}{\vareps}(c_{i+1}^\vareps-c_i^\vareps), \qquad i\geq i_0+1,\qquad \deriv{t}c_{i_0}^\eps=\f{b}{\eps}c_{i_0+1}^\eps, \qquad c_i^\vareps(0)=c_i^{0,\eps}.
\end{equation}
We also assume that 
\[
	\sum_{i \geq i_0} \vareps  |c_i^\vareps(0)|^2 < \infty,
\]
namely $(c_i^\ep(0))_{i\geq i_0} \in \ell_2(\N)$. We easily deduce that this is still the case for all $t > 0$,  {\it i.e.} $(c_i^\ep(t))_{i\geq i_0} \in \ell_2(\N)$.

Let us now define the scaled size variable $x_i^\vareps \coloneqq \vareps (i-i_0)$, and the piecewise constant interpolant $c^\vareps \in \Czero[\R^+;\Ltwo[\R^+]]$ defined by \begin{equation}\label{def:ueps}
	c^\vareps(t,x) \coloneqq c_{i+1}^\ep, \quad t \geq 0,\,   x \in [x_i^\ep,x_{i+1}^\ep).
\end{equation}

Assuming $\vareps$ small, we prove in the next proposition that the limit of $c^\vareps$ is $\solFinf$ solution of 
\begin{equation}
\label{eq:first-order}
\left\{
\begin{aligned}
	\partial_t \solFinf - b \partial_x \solFinf&=0,&  (t,x) \in  [0,T]\times \R^+ ,\\
	\solFinf(0,x) &= \sol^0(x), & x \in \R^+,
\end{aligned}\right.
\end{equation}
for a convenient initial condition $\sol^0$ close to $c^\ep (0,x)=(c_i^{0,\ep})_{i\geq i_0+1}$.
Assuming that $\sol^0(x) \in \mathrm{H}^k(\R_+^*),\, k \geq 1$, the method of characteristics or a simple application of semi-group theory gives 
\[
	\solFinf \in \Czero[(0,T);\mathrm{H}^k(\R_+^*)] \cap \Cone[(0,T);\Ltwo[\R_+]].
\]
In order to compare the solution of \eqref{eq:BD} and the solution of \eqref{eq:first-order}, we introduce the discrete norm
\[
 \norm{u}_{2,\vareps}^2 \coloneqq \sum_{i \geq i_0} \vareps  |u(x_i^\vareps)|^2,
\]
which is well-defined for functions in $\Hone[\R_+^*]$, consistent with the $\Ltwo[\R_+]$-norm as $\varepsilon$ tends to 0, and which coincides with the $\Ltwo[\R^+]-$norm for piecewise constant functions defined on the grid $(x_i^\ep),$ as $c^\ep(t,\cdot)$ is. We have the following error estimate between $\sol^\ep$ representing $c_i^\ep$ and $\sol_\infty$.

\begin{proposition}\label{prop:first-order:estim}
Let  $\sol^0 \in \mathrm{H}^2(\R^+)$ and $\solFinf\in  \Czero[(0,T);\mathrm{H}^2(\R_+^*)] \cap \Cone[(0,T);\Ltwo[\R_+]]$ solution to~\eqref{eq:first-order}. Let $(c_i^{0,\ep})\in \ell_2 (\N)$, $(c_i^\ep)$ solution to~\eqref{eq:BD} and $c^\ep$ defined by~\eqref{def:ueps}  such that 
\[\|c^\vareps(0,\cdot)- \sol^0 \|_{2,\vareps} \leq \varepsilon.\] 
There exists a constant $\Cst >0$  such that for all $t > 0$, we have
	 \[
	 \norm{ c^\vareps(t,\cdot)- \solFinf(t,\cdot)}_{2,\vareps} \leq  \vareps \inserted{(1+\Cst \Vert\sol^0\Vert_{\Htwo}t)}.
	 \]
\end{proposition}
{\inserted{\begin{remark}
If we simply assume $\sol^0(x_i^\ep)=c^\ep(0,x_i)=c_{i+1}^{0,\ep},$  the right-hand side becomes $\Cst \ep  \Vert\sol^0\Vert_{\Htwo} t.$ Our assumption means that when $\ep\to 0$ the distribution $(c_i^{0,\ep})$ is the discretisation of a smooth function $\sol^0.$ In a numerical analysis perspective, our estimate is an error estimate for a semi-discretisation of the transport equation~\cite{Hundsdorfer_2007}. Following~\cite{ColletGoudonPoupaudVasseur_2002,LaurencotMischler_2002}, we could also look for equivalent assumptions easier to justify from a modelling perspective, of the kind
\[\vert c_{i+1}^{0} - c_i^0 \vert \leq K_i/i_M,\qquad  \vert c_{i+1}^{0} - 2 c_i^0 +c_{i-1}^0 \vert \leq K_i/i_M^2,
\]
with $\norm{K_i}_{2,\ep} <\infty$ and $i_M=1/\ep.$ 
    \end{remark}}}
\begin{proof}

Let \inserted{$c^\ep$ defined by~\eqref{def:ueps}} and introduce $\tilde{\sol}^{\vareps}(t,\cdot) \coloneqq c^\vareps(t,\cdot) - \solFinf(t,\cdot)$. We have for all $i\geq i_0$
\begin{align}\label{eq:eq-diff-asympt-0}
	\frac{\diff}{\diff t} \tilde{\sol}^{\vareps}(t,x^\vareps_i) &= 	\frac{b}{\vareps}(c^\vareps(t,x^\vareps_{i+1}) - c^\vareps(t,x^\vareps_{i})) - b \partial_x \solFinf(t,x^\vareps_i) \notag \\
	&=  \frac{b}{\vareps}(\tilde{\sol}^{\vareps}(t,x^\vareps_{i+1}) - \tilde{\sol}^{\vareps}(t,x^\vareps_{i})) \notag \\
	&\hspace{1cm}+ \left[ \frac{b}{\vareps}(\solFinf(t,x^\vareps_{i+1}) - \solFinf(t,x^\vareps_{i})) - b \partial_x \solFinf (t,x_i) \right],
\end{align}
By Taylor's expansion with remainder in integral form applied to $\solFinf(t,\cdot) \in \Htwo[x_i^\vareps,x_{i+1}^\vareps]$ we have 
\[
	\solFinf(t,x^\vareps_{i+1}) = \solFinf(t,x^\vareps_{i}) + \vareps \partial_x \solFinf(t,x^\vareps_{i}) +  \int_{x_i}^{x_{i+1}} \partial^2_{x} \solFinf(t,x) ( x_{i+1} - x )\, \diff x,
\]
we obtain
\begin{multline}\label{eq:eq-diff-asympt-0b}
	\frac{\diff}{\diff t} \tilde{\sol}^{\vareps}(t,x^\vareps_i) = \frac{b}{\vareps}\bigl(\tilde{\sol}^{\vareps}(t,x^\vareps_{i+1}) - \tilde{\sol}^{\vareps}(t,x^\vareps_{i})\bigr) +\frac{b}{\vareps}\int_{x_i}^{x_{i+1}} \partial^2_{x} \solFinf(t,x) (x_{i+1} - x)\, \diff x.
\end{multline}
Multiplying this equation by $\tilde{\sol}_i(t) \coloneqq \tilde{\sol}^{\vareps}(t,x^\vareps_i)$, we obtain
\begin{align}\label{eq:eq-diff-asympt-0c}
	\f{1}{2}\frac{\diff}{\diff t} | \tilde{u}_i|^2 &= \frac{b}{\vareps}\left(\tilde{u}_i \tilde{u}_{i+1}-(\tilde{u}_i)^2\right) + \frac{b}{\vareps}\int_{x_i}^{x_{i+1}} \tilde{u}_i \partial^2_{x} \solFinf(t,x) (x_{i+1} - x)\, \diff x \notag \\ 
&\inserted{\leq \frac{b}{2\vareps}\left(|\tilde{u}_{i+1}|^2-|\tilde{u}_i|^2\right)} + \frac{b}{\vareps}\int_{x_i}^{x_{i+1}} \tilde{u}_i \partial^2_{x} \solFinf(t,x) (x_{i+1} - x)\, \diff x.
\end{align}
By summing these quantities --~with $|\tilde{u}_i(t)|^2 \xrightarrow[]{i\to \infty} 0$~--  we find 
\[
	{\f{1}{2}}\frac{\diff}{\diff t}  \sum_{i \geq i_0} \vareps|\tilde{u}_i(t)|^2 \leq -\frac{b}{2} |\tilde{u}_{i_0}|^2 +  \sum_{i \geq i_0} {b \tilde{u}_i(t)} \int_{x_i}^{x_{i+1}} \partial^2_{x} \solFinf(t,x) (x_{i+1} - x)\, \diff x, \\
\]	
which implies
\begin{multline}\label{eq:estimate-l2-1}
	\sum_{i \geq i_0} \vareps|\tilde{u}_i(t)|^2 \leq   \sum_{i \geq i_0} \vareps|\tilde{u}_i(0)|^2  \\ 
	+ 2 \int_0^t \left(\sum_{i \geq i_0} \vareps|\tilde{u}_i(\inserted{s})|^2\right)^{\frac{1}2} \left( \sum_{i \geq i_0} \frac{b^2}{\vareps} \left(\int_{x_i}^{x_{i+1}} \partial^2_{x} \solFinf(s,x) (x_{i+1} - x) \diff x\right)^2\right)^{\frac{1}2} \diff s.
\end{multline}
Using Cauchy-Schwarz inequality, we find
\begin{align*}
&\sum_{i \geq i_0}\frac{1}{\vareps} \left(\int_{x_i}^{x_{i+1}} \partial^2_{x} \solFinf(s,x) (x_{i+1} - x) \diff x\right)^2 \\
&\hspace{2cm}\leq \sum_{i \geq i_0} \frac{1}{\vareps} \left( \int_{x_i}^{x_{i+1}} \abs{\partial^2_{x} \solFinf(s,x)}^2 \diff x \right) \left(  \int_{x_i}^{x_{i+1}} (x_{i+1} - x)^2 \diff x \right)\\
&\hspace{2cm}\leq \frac{\vareps^2}{{3}}   \int_{x_{i_0}}^{\infty} \abs{\partial^2_{x} \solFinf(s,x)}^2 \diff x ,
\end{align*}
leading to
\begin{multline}\label{eq:estimate-l2}
	\sum_{i \geq i_0} \vareps|\tilde{u}_i(s)|^2 \leq   \sum_{i \geq i_0} \vareps|\tilde{u}_i(0)|^2  \\ 
	+ \frac{2\vareps{b}}{\sqrt{3}} \int_0^t \left(\sum_{i \geq i_0} \vareps|\tilde{u}_i(s)|^2\right)^{\frac{1}2} \left( \int_{0}^{\infty} |\partial_{x}^2 \solFinf  (s,x)|^2 \diff x\right)^{\frac{1}2} \diff s.
\end{multline}
We then recall the following Gronwall's lemma \cite[Theorem 5]{dragomir2003some}.
\begin{lemma}[Gronwall Lemma]\label{lemma:Gronwall}
Let $f : [a, b] \to \R_+$   a continuous function which satisfies the following relation
\[
	\frac{1}{2} f(t)^2\leq \frac{1}{2} f_0^2 +\int_a^t g(s)f(s)\diff s, \quad t \in [a,b],
\]
where $f_0\in \R$  and $g$ is a nonnegative continuous function in $[a, b].$ Then the following estimate holds
\[
	\vert f(t)\vert \leq \vert f_0\vert + \int_a^t g(s) \diff s,\quad t\in [a,b].
\]	
\end{lemma}
Hence applying Lemma~\ref{lemma:Gronwall} to \eqref{eq:estimate-l2}, \inserted{by taking
\[g(s)=\f{\ep b}{\sqrt{3}}\left( \int_{0}^{\infty} |\partial_{x}^2 \solFinf  (s,x)|^2 \diff x\right)^{\frac{1}2} \leq \f{\ep b}{\sqrt{3}} \Vert \sol^0\Vert_{H^2(\R^+)}
\]}
we  obtain
\[
	\left(\sum_{i \geq i_0} \vareps|\tilde{u}_i(t)|^2\right)^{\frac{1}2} \leq \inserted{\eps+}\frac{\varepsilon \inserted{b} t}{\sqrt{3}}\inserted{\Vert u^0\Vert_{H^2(\R^+)}},
\]
which finally leads to
\[
		\norm{ c^\vareps(t,\cdot) - \solFinf(t,\cdot)}_{2,\varepsilon} \lesssim  \vareps \inserted{(1+t\Vert u^0\Vert_{H^2(\R^+)})},
\]
and concludes the proof.
\end{proof}

\subsection{Second order asymptotics: an advection-diffusion model}\label{sec:second-order}
\label{subsec:direct:advection-diffusion}
We now go to the second order approximation, and consider the following system
\begin{equation}\label{eq:second-order}
\left\{
\begin{aligned}
        \partial_t \solSinf-b \partial_x\solSinf- \f{b\ep}{2} \partial_{x}^2 \solSinf &= 0, & (t,x)\in  [0,T]\times\R^+, \\
        \partial_t\solSinf(t,0) - b \partial_x \solSinf(t,0) &=0, & t\in[0,T],  \\
        \solSinf(0,x) &= \sol^\vareps(0,x),&  x\in \R^+.
\end{aligned}
\right.
\end{equation}
\inserted{The advection-diffusion equation satisfied by $\solSinf$ comes naturally from a Taylor expansion of the equation, see {\it e.g.}~\cite{Velazquez_1998,ColletGoudonPoupaudVasseur_2002}. The idea of the transport equation as a boundary condition is to lead to a solution as close as possible to the solution in the whole space, hence without any boundary condition: the transport equation is then obtained as the first order approximation of the exact transparent boundary condition, see~e.g.~\cite{halpern1986artificial}. This led us to state the first-order equation, without the corrective diffusion term, and to guess that since it only concerns a boundary layer the error is not made worse. 

We can also justify the boundary condition as follows. An important specificity of our model, differently from the full Becker-D\"oring system where $i_0=1$, is the conservation of the number of polymers, defined by
\[P_0\coloneqq \ep \sum\limits_{i=i_0}^\infty c_i^\ep.\]
We look for $u^\ep_\infty (t,x_i)=c_{i+1}+o(\ep),$ and having already obtained the transport-diffusion equation for $u^\ep_\infty,$  we want to preserve the conservation of polymers for the continuous model. We first have
\[ \deriv{t}c_{i_0}^\ep = \f{b}{\ep}c_{i_0+1}^\ep = \f{b}{\ep} u^\ep_\infty(t,0) +o(1),\quad \deriv{t}c_{i_0+1}^\ep =b\f{c^\ep_{i_0+2}-c^\ep_{i_0+1}}{\ep}=b\partial_x u^\ep_\infty(t,0) +o(1), \]
and by a Taylor expansion, for $i\geq i_0+1$ and $x\in (x_{i-1},x_i)$ we have
\[u(t,x)=u(t,x_i)-\int_x^{x_i} \p_x u(t.s)\diff s,\]
hence integrating in $(x_{i-1},x_i)$ we get
\[\begin{array}{lll}\dst\int_{x_{i-1}}^{x_i} u(t,x)\diff x &= &\ep u(t,x_i) -\int_{x_{i-1}}^{x_i} \p_x u(t,s) \int_{x_{i-1}}^s \diff x \diff s 
\\
&= &\ep c_{i+1}(t)  -\int_{x_{i-1}}^{x_i} \p_x u(t,s) (s-x_{i-1}) \diff s,
\end{array}\]
and using the approximation $\int_{x_{i-1}}^{x_i} \p_x u(t,s) (s-x_{i-1}) \diff s=\f{\ep}{2} \int_{x_{i-1}}^{x_i} \p_x u(t,s)\diff s +o(\ep^2)$ we get
\begin{equation}\label{approx_ci} c_{i+1}^\ep=\f{1}{\ep}\int_{x_{i-1}}^{x_i} u^\ep_\infty(t,x)\diff x +\f{1}{2} \int_{x_{i-1}}^{x_i} \partial_x u^\ep_\infty (t,s)\diff s +o(\ep)
.\end{equation}
Gathering all this, we formally have
\[\deriv{t} P_0= {b} u^\ep_\infty(t,0) + \ep b\partial_x u^\ep_\infty(t,0) +   \deriv{t} \int_0^\infty u^\ep_\infty(t,x)\diff x +\f{\ep}{2} \deriv{t} \int_0^\infty\partial_x u^\ep_\infty (t,s)\diff s +o(\ep)\]
hence, applying the transport-diffusion equation,
\[\begin{array}{lll}
0&=&\f{b}{\ep} u^\ep_\infty(t,0) + b\partial_x u^\ep_\infty(t,0)+ \f{1}{\ep} \int_0^\infty (b \partial_x+ \f{b\ep}{2} \partial_x^2 ) u^\ep_\infty \diff x   -\f{1}{2} \deriv{t} u^\ep_\infty (t,0)+o(1)
\\ \\
&=&\f{b}{\ep} u^\ep_\infty(t,0) + b\partial_x u^\ep_\infty(t,0)- \f{b}{\ep} u^\ep_\infty (t,0) -\f{b}{2} \partial_x  u^\ep_\infty (t,0)   -\f{1}{2} \partial_t u^\ep_\infty (t,0)+o(1) 
\\ \\
&=& \f{b}{2}\partial_x u^\ep_\infty(t,0)  -\f{1}{2} \partial_t u^\ep_\infty (t,0)+o(1),
\end{array}
\]
which provides us with the transport equation for the boundary condition.}  

To study first the well-posedness and some regularity properties of the equation, we define $\stateSpaceI = \Ltwo[\R^+] \times \R$, equipped with the norm
\[
    \forall \state = (u,v) \in \stateSpaceI, \quad \norm{\state}_{\stateSpaceI}^2 = \norm{u}_{\Ltwo[\R^+]}^2 + |v|^2.
\]
We introduce the operator $(\modelOpI,\mathcal{D}(\modelOpI))$ defined by
\[
    \modelOp_{\infty} \state = \begin{pmatrix}
    - b u' -\f{b\ep}{2} u'' \\
    - b\sqrt{\f{\vareps}{2}} u'(0)
    \end{pmatrix},
\]
where
\[
    \mathcal{D}(\modelOpI) = \left\{ \state=\begin{pmatrix}u \\ v \end{pmatrix}\in \stateSpaceI \, \Big| \, u \in \mathrm{H}^2(\R^+), \, \sqrt{\f{\ep}{2}} u(0) = v \right\},
\]
such that a solution $\solSinf$ of \eqref{eq:second-order} is defined from the $\stateI = \begin{pmatrix}\solSinf \\ v \end{pmatrix}$ solution of the initial value problem
\begin{equation}\label{eq:semigroup}
\begin{cases}
	\dsp \deriv{t} \stateI + \modelOpI \stateI = 0, \quad t > 0 \\
	\stateI(0) = \stateI^0 
\end{cases}
\end{equation}
with $\stateI^0 = \begin{pmatrix} \solSinf(0,x) \\ \sqrt{\f{\ep}{2}} \solSinf(0,0) \end{pmatrix}$.

\begin{proposition}\label{prop:second-order:exist}
	The operator $(\modelOpI,\mathcal{D}(\modelOpI))$ is maximal {accretive} in $\stateSpaceI$. Therefore, for any initial condition $\stateI^0 \in \stateSpaceI$, there exists a unique mild solution $\stateI \in \Czero[(0,T);\stateSpaceI]$ to \eqref{eq:semigroup}.	Furthermore for any initial condition $\stateI^0 \in \mathcal{D}(\modelOpI)$, there exists a unique strict solution to \eqref{eq:semigroup} with
	\[
		\stateI \in \Czero[(0,T);\mathcal{D}(\modelOp_\infty)] \cap \Cone[(0,T);\stateSpaceI].
	\]
\end{proposition}
\begin{proof}
First, we show that $\modelOpI$ is accretive. We have indeed, for all $\state = (u,v) \in \mathcal{D}(\modelOp_\infty)$,
\begin{align*}
 (\modelOpI \state,\state)_{\stateSpaceI} &=- \int_0^\infty b u u' \diff x - \int_0^\infty \f{b\eps}{2} u u'' \diff x- b\sqrt{\f{\eps}{2}} u'(0)v  \\
&= \frac{b}2 u(0)^2 + \f{b\ep}{2} \int_0^\infty | u'|^2\diff x {\geq} 0.    
\end{align*}

Second, in order to show that $\modelOpI$ is maximal, we introduce the bilinear form
\begin{multline}\label{eq:blilinear}
    a_{\infty,\lambda}((u,v),(w,k)) = \lambda(u,w)_{\Ltwo[\R^+]} + \lambda v k \\
    	- b( u', w)_{\Ltwo[\R^+]} + \f{b\ep}{2} (u', w')_{\Ltwo[\R^+]},
\end{multline}
defined in 
\[
	\VI = \Bigl\{ (u,v)\in \Hone[\R^+] \times \R \, | \, \sqrt{\f{\ep}{2}} u(0) = v \Bigr\}.
\]		
Let $(f,g) \in \stateSpaceI$ and $\lambda > 0$. We seek $\state = (u,v) \in \mathcal{D}(\modelOpI)$ such that $(\lambda \Id + \modelOpI)\state = (f,g)$. This means that for all $(w,k) \in \mathcal{V}$
\begin{equation}\label{eq:LaxMil}
    a_{\infty,\lambda}((u,v),(w,k)) = (f,w)_{\Ltwo[\R^+]} + gk,
\end{equation}
For $\lambda > 0$, $a_{\infty,\lambda}$ is coercive \inserted{in $\VI$} since
    \[
        a_{\infty,\lambda}((u,w),(u,w)) = \lambda(u,w)_{\Ltwo[\R^+]} + \lambda \abs{v}^2 + \f{b\ep}{2} (u',u')_{\Ltwo(\R^+)} + \frac{b}2 u(0)^2.
    \]
Therefore, from Lax-Milgram theorem, there exists one and only one solution $(u,w) \in  \VI$ to \eqref{eq:LaxMil}. Furthermore, for all $w \in {D(\R^+) =} \mathrm{C}^\infty_c(\R^{+\inserted{*}})$, we have
\[
\langle \f{b\ep}{2}  u'', w \rangle_{D'(\R^+),D(\R^+)} = -(f,w)_{\Ltwo[\R^+]} - (bu',w)_{\Ltwo[\R^+]} +\lambda (u,w)_{\Ltwo[\R^+]} 
\]
with $f-bu' \in \Ltwo[\R^+]$. Therefore, $ u'' \in \Ltwo[\R^+]$, hence $u \in \Htwo[\R^+]$, hence $(u,v) \in \mathcal{D}(\modelOpI)$. The existence then follows Lumer-Philips theorem, see for instance \cite[Theorem 2.6, Chapter 1 of Part II]{bensoussan-daprato-book}.
\end{proof}

Before proving an asymptotic estimate between the solutions to~\eqref{eq:BD} and to~\eqref{eq:second-order}, we  need to state a regularity result for~\eqref{eq:second-order}. This is given by the following lemma.

\begin{lemma}\label{prop:second-order:reg}
Let $\uin\in \Hthree[\R^{+}]$ and $\solSinf$ the unique solution to~\eqref{eq:second-order} defined in Prop.~\ref{prop:second-order:exist}. Then $\solSinf \in C^0\left((0,T);\Hthree[\R^{+}]\right)$ and we have
$$\int_{\R_+} \vert \p^2_{x} \solSinf(t,x)\vert^2 \diff x 
\leq \Vert \sol_0\Vert^2_{\Htwo}, \qquad
{b\ep}\int_0^{\TM} \int_{\R^+}\vert \p^3_{x}  \solSinf\vert^2 \diff x \diff t  \leq \Vert \sol_0\Vert^2_{\Htwo}.$$
\end{lemma}
\begin{proof}
Let us denote $v=\p_x \solSinf,$ $w=\p_{x}^2 \solSinf.$ Deriving~\eqref{eq:second-order} once or twice in space, we have respectively:
\begin{equation}\label{eq:second-order:v}
\left\{
\begin{aligned}
        \partial_t v-b \partial_x v- \f{b\ep}{2}\partial_{x}^2 v &= 0, & (t,x)\in  [0,T]\times\R^+, \\
        \p_x v(0,t) &=0, & t\in[0,T],  \\
        v(0,x) &=  \sol_0'(x),&  x\in \R^+,
\end{aligned}
\right.
\end{equation}
where the Neumann boundary condition is given by using the equation for $\solSinf$ taken at $x=0,$
and similarly
\begin{equation}\label{eq:second-order:w}
\left\{
\begin{aligned}
        \partial_t w-b \partial_x w- \f{b\ep}{2} \partial_{x}^2 w &= 0, & (t,x)\in  [0,T]\times\R^+, \\
        w(t,0) &=0, & t\in[0,T],  \\
        w(0,x) &=  u_0''(0,x),&  x\in {\R^+}.
\end{aligned}
\right.
\end{equation}
We integrate in space  the equation for $w$ multiplied by $w$ and find
$$\f{1}{2}\f{\diff}{\diff t} \int_{\R^+} \vert w(t,x)\vert^2 \diff x
+\f{b\ep}{2} \int_{\R^+} \vert \p_x w\vert^2 \diff x=0.
$$
Integrating now in time gives
$$\f{1}{2} \int_{\R^+} \vert w(t,x)\vert^2 \diff x + \int_0^t \f{b\ep}{2} \int_{\R^+}\vert \p_x w (s,x)\vert^2 \diff x \diff s =
\f{1}{2} \int_{\R^+} \vert w(0,x)\vert^2 \diff x,$$
which provides us with the two desired inequalities.

\end{proof}
We are now ready to state an asymptotic estimate between~\eqref{eq:BD} and~\eqref{eq:second-order}, which reveals better  than the estimate between~\eqref{eq:BD} and~\eqref{eq:first-order} obtained in Prop.~\ref{prop:first-order:estim}. 
\begin{proposition}\label{prop:second-order:estim}
Let  $\uin \in \mathrm{H}^2(\R^+)$ and $\solSinf\in 	 \Czero[(0,T);\Htwo[\R^+]] \cap \Cone[(0,T);\Ltwo[\R^+]]$ the unique solution to~\eqref{eq:second-order} given by Proposition~\ref{prop:second-order:exist} with $\uin$ as an initial condition. Let $(c_i^{0,\ep})\in \ell_2 (\N)$, $(c_i^\ep)$ solution to~\eqref{eq:BD} and $c^\ep$ defined by~\eqref{def:ueps}  such that 
\[\|c^\vareps(0,\cdot)- \sol^0 \|_{2,\vareps} \leq \varepsilon^\alpha,\] 
for $\alpha>0.$ There exists a constant $\Cst>0$ such that for all $t > 0$,   we have
	 \[
	 \norm{ c^\vareps(t,\cdot) - \solSinf(t,\cdot)}_{2,\vareps} \leq \Cst (\ep^\alpha + t^{\f{1}{2}} \vareps^{\f{3}{2}}).
	 \]
\end{proposition}
{\inserted{
\begin{remark}
As for the first order approximation, if we simply assume $c_i^{0,\ep}=\sol^0(x_i^\ep),$ the right-hand side of the estimate is reduced to $\Cst  t^{\f{1}{2}} \vareps^{\f{3}{2}}.$ It is also likely that assuming more regularity on the initial condition would lead to a convergence in the order of $\ep^2.$
\end{remark}}}

\begin{proof}
 Let us first assume $\sol^0 \in \mathrm{H}^3(\R^+):$ the result for $\sol^0\in \Htwo[\R^+]$ follows by density of $\Hthree[\R^+]$ in  $\Htwo[\R^+]$.  Lemma~\ref{prop:second-order:reg} implies that	$\solFinf \in \Czero[(0,T);\mathrm{H}^3(\R^+)]$.
Defining again \inserted{$\sol^\ep$ by~\eqref{def:ueps} and}
$$\tilde{\sol}^{\vareps,1}(t,\cdot) \coloneqq c^\vareps(t,\cdot) - \solSinf(t,\cdot),$$ 
we have for $i >  i_0$
\begin{align*}
	\frac{\diff}{\diff t} \tilde{\sol}^{\vareps,1}(t,x^\vareps_i) 
	&=  \frac{b}{\vareps}(\tilde{\sol}^{\vareps,1}(t,x^\vareps_{i+1}) - \tilde{\sol}^{\vareps,1}(t,x^\vareps_{i})) \\
	&\quad\quad + \left[ \frac{b}{\vareps}(\solSinf(t,x^\vareps_{i+1}) - \solSinf(t,x^\vareps_{i})) - b \partial_x \solSinf (t,x_i) - \frac{b\varepsilon}2 \partial_x^2 \solSinf (t,x_i) \right],
\end{align*}
whereas we keep \eqref{eq:eq-diff-asympt-0} for $i=i_0$.

By Taylor's expansion with remainder in integral form applied to $\solSinf(t,x^\vareps_{i+1}) \in \mathrm{H}^3[x_i^\vareps,x_{i+1}^\vareps]$,  
\begin{multline*}
		\solSinf(t,x^\vareps_{i+1}) = \solSinf(t,x^\vareps_{i}) + \vareps \partial_x \solSinf(t,x^\vareps_{i}) + \frac{\vareps^2}2 \partial_x^2 \solSinf(t,x^\vareps_{i}) \\
		 + \frac{1}2 \int_{x_i}^{x_{i+1}} \partial^3_{x} \solSinf(t,x) (x_{i+1} - x)^2\, \diff x,
\end{multline*}
leading to, for $i>i_0$,
\[
	\frac{\diff}{\diff t} \tilde{\sol}^{\vareps,1}(t,x^\vareps_i) = \frac{b}{\vareps}\bigl(\tilde{\sol}^{\vareps,1}(t,x^\vareps_{i+1}) - \tilde{\sol}^{\vareps,1}(t,x^\vareps_{i})\bigr) {+} \frac{b}{2\vareps}\int_{x_i}^{x_{i+1}} \partial^3_{x} \solSinf(t,x) (x_{i+1} - x)^2\, \diff x.
\]
whereas we keep \eqref{eq:eq-diff-asympt-0b} for $i=i_0$.
Multiplying this equation by $\tilde{\sol}_i(t) \coloneqq \tilde{\sol}^{\vareps,1}(t,x^\vareps_i)$, we obtain for $i >i_0$
\begin{align*}
	{\f{1}{2}}\frac{\diff}{\diff t} | \tilde{u}_i|^2 &= \frac{b}{\vareps}\left(\tilde{u}_i \tilde{u}_{i+1}-(\tilde{u}_i)^2\right) {+} \frac{b}{6\vareps}\int_{x_i}^{x_{i+1}} \tilde{u}_i \partial^3_{x} \solSinf(t,x) (x_{i+1} - x)^2\, \diff x  \\ 
&=\frac{b}{2\vareps}\left(|\tilde{u}_{i+1}|^2-|\tilde{u}_i|^2\right)-\frac{b}{2\vareps}(\tilde{u}_{i+1}-\tilde{u}_i)^2 {+} \frac{b}{2\vareps}\int_{x_i}^{x_{i+1}} \tilde{u}_i \partial^3_{x} \solSinf(t,x) (x_{i+1} - x)^2\, \diff x,
\end{align*}
and still \eqref{eq:eq-diff-asympt-0c} for $i=i_0$.
By summing these quantities --~with $|\tilde{u}_i(t)|^2 \xrightarrow[]{i\to \infty} 0$~--  we find 
\begin{multline*}
	{\f{1}{2}}	\frac{\diff}{\diff t}  \sum_{i \geq i_0} \vareps|\tilde{u}_i(t)|^2 \leq -\frac{b}{2} |\tilde{u}_{i_0}|^2 {+} {b \tilde{u}_{i_0}(t)} \int_{x_{i_0}}^{x_{i_0+1}} \partial^2_{x} \solSinf(t,x) (x_{i+1} - x)\, \diff x \\ {+}  \sum_{i > i_0} \frac{b \tilde{u}_i(t)}2 \int_{x_i}^{x_{i+1}} \partial^3_{x} \solSinf(t,x) (x_{i+1} - x)^2\, \diff x. 
\end{multline*}
\inserted{We gather the first two terms of the right-hand side and use a Young's inequality
\begin{multline*}-\frac{b}{2} |\tilde{u}_{i_0}|^2 + {b \tilde{u}_{i_0}(t)} \int_{x_{i_0}}^{x_{i_0+1}} \partial^2_{x} \solSinf(t,x) (x_{i+1} - x)\, \diff x  
\\ \leq - \f{b}{4} |\tilde{u}_{i_0}|^2 + 
\f{b}{2}\left(\int_{x_{i_0}}^{x_{i_0+1}} \vert \partial^2_{x} \solSinf(t,x) (x_{i+1} - x)\vert \, \diff x  \right)^2
\\ \leq
\f{b}{2} \f{\ep^3}{3}\Vert \p^2_{xx}\solSinf (t,\cdot)\Vert_{\Ltwo}^2\lesssim \ep^3\Vert \sol^0 \Vert_{\Htwo}^2,
\end{multline*}
by Lemma~\ref{prop:second-order:reg}.
}
Then, we integrate in time and get 
\begin{multline}\label{eq:estimate-l2-3}
	\sum_{i \geq i_0} \vareps|\tilde{u}_i(t)|^2 \lesssim   \sum_{i \geq i_0} \vareps|\tilde{u}_i(0)|^2  +
{\ep^{{3}}\Vert \sol^0\Vert_{\Htwo}^2 t} \\
	+  \int_0^t \left(\sum_{i \geq i_0} \vareps|\tilde{u}_i(s)|^2\right)^{\frac{1}2} \left( \sum_{i \geq i_0} \frac{1}{\vareps} \left(\int_{x_i}^{x_{i+1}} \partial^3_{x} \solSinf(s,x) (x_{i+1} - x)^2 \diff x\right)^2\right)^{\frac{1}2} \diff s.
\end{multline}
Using Cauchy-Schwarz inequality for the last term of the right-hand side, we find
\begin{align*}
&\sum_{i \geq i_0}\frac{1}{\vareps} \left(\int_{x_i}^{x_{i+1}} \partial^3_{x} \solSinf(s,x) (x_{i+1} - x)^2 \diff x\right)^2 \\
&\hspace{2cm}\leq \sum_{i \geq i_0} \frac{1}{\vareps} \left( \int_{x_i}^{x_{i+1}} \abs{\partial^3_{x} \solFinf(s,x)}^2 \diff x \right) \left(  \int_{x_i}^{x_{i+1}} (x_{i+1} - x)^4 \diff x \right)\\
&\hspace{2cm}\lesssim \vareps^4  \int_{x_{i_0}}^{\infty} \abs{\partial^3_{x} \solSinf(s,x)}^2 \diff x ,
\end{align*}
leading to
\begin{multline}\label{eq:estimate-order1-l2}
	\sum_{i \geq i_0} \vareps|\tilde{u}_i(t)|^2 \lesssim   \sum_{i \geq i_0} \vareps|\tilde{u}_i(0)|^2 
	+ {\ep^{3}\Vert \sol^0\Vert_{\Htwo}^2 T}\\ 
	 + \vareps^2 \int_0^t \left(\sum_{i \geq i_0} \vareps|\tilde{u}_i(s)|^2\right)^{\frac{1}2} \left( \int_{0}^{\infty} |\partial_{x}^3 \solSinf  (s,x)|^2 \diff x\right)^{\frac{1}2} \diff s
\end{multline}
Hence applying the Gronwall Lemma~\ref{lemma:Gronwall} to \eqref{eq:estimate-order1-l2}, we obtain
\begin{multline*}
	\left(\sum_{i \geq i_0} \vareps|\tilde{u}_i(t)|^2\right)^{\frac{1}2} \lesssim \ep^\alpha +  \ep^{\f{3}{2}} t^{\f{1}{2}} \norm{\sol^0}_{\Htwo[\R^+]} + 
	\ep^2 \int_0^t \left( \int_{0}^{\infty} |\partial_{x}^3 \solSinf  (s,x)|^2 \diff x\right)^{\frac{1}2} \diff s
\\
\lesssim 
\ep^\alpha +  \ep^{\f{3}{2}} t^{\f{1}{2}} \norm{\sol^0}_{\Htwo[\R^+]} 
+
\sqrt{ \int_0^t \ep^3 \diff s} \sqrt{\int_0^t \ep \int_{0}^{\infty} |\partial_{x}^3 \solSinf  (s,x)|^2 \diff x \diff s}
\\
\lesssim
\ep^\alpha +  \ep^{\f{3}{2}} t^{\f{1}{2}} \norm{\sol^0}_{\Htwo[\R^+]} ,
\end{multline*}	
on which we have applied Lemma~\ref{prop:second-order:reg} to estimate the term with $\vert \p^3_x \solSinf \vert.$
This concludes the proof.
 \end{proof}

{\inserted{\begin{remark}
As for the first order system, our proof relies  on a Taylor-Lagrange expansion and adequate estimates in $L^2$ type spaces. Under our set of regularity assumptions, we expect these convergence rates - of order $\ep$ for the first order system, $\ep^{\f{3}{2}}$ for the second order one - to be optimal. 
 Let us  discuss briefly the link between these estimates and the existing ones, which have been most often carried out under more general assumptions for the full  Becker-D\"oring system.

We chose to act on the timescale rather than on the average size to introduce the parameter $\ep;$ however, with our choice $x_i^\ep=\ep i,$ it is clear that $\ep$ is in the order of the inverse of the average size of the polymers, as standard. In several studies~\cite{ColletGoudonPoupaudVasseur_2002,Niethammer_2003,LaurencotMischler_2002,Schlichting_2019},  weak convergence results in certain spaces of measures are obtained to the  Lifshitz-Slyozov system (the first order limit). Of note, the weak convergence may take place without any regularity assumption on the initial condition, except some uniform boundedness of moments: only regularity of the reaction rates (obvious in our case since the depolymerisation rate is constant) is needed. Recently, in~\cite{Schlichting_2019},  convergence results along so-called "curves of finite action"  are obtained through a gradient flow approach. 

Up to our knowledge, the only article where a quantitative rate of convergence between the Becker-D\"oring system and its Fokker-Plank ({\it i.e.} the second order) approximation is obtained is~\cite{StoltzTerrier_2019},  in the case of infinitely differentiable initial data. As in our case, the proof is based on estimates on the remainder of the Taylor expansion, and the authors link  the size of the considered clusters - which tend to infinity in large times in the case of gelation - with the rate of convergence, which is, in an $L^\infty$ norm, of order $n^{-\gamma/2}$ when the reaction rates grow as $x^\gamma$ and $n$ is the cluster size. 
\end{remark}
}}

\section{Setting of the inverse problems}
\label{sec:inverse}

Our objective is to reconstruct the initial polymer distribution using a measurement of the first or second moment of the concentration function. In this respect, we prefer to rely on an asymptotic formulation rather than  the discrete model \eqref{eq:depol-ode}. Indeed, asymptotic models are classically simpler to analyse in term of observability and offer the benefit of reducing the dimensionality of the discrete model when relying on a controlled coarser discretization of the asymptotic continuous model. Here, we aim at evaluating the interest of relying on the first order asymptotics or the second order asymptotics in the initial distribution reconstruction. We face a classical accuracy versus stability trade-off: the second order asymptotic model is a more accurate modeling choice but leads to stronger ill-posedness when setting up the inverse reconstruction. 
\inserted{\subsection{The Becker-D\"oring inverse problem}
Let us first set the original Becker-D\"oring inverse problem. We assume that one of the first three moments has been measured during the time period $[0,{\mathcal T}]$, namely
\inserted{\[
	M_k(\tau) \coloneqq  \sum\limits_{i=i_0+1}^\infty i^k C_i (\tau), \quad \tau\in [0,{\mathcal T}],
\] 
for some $k \in \{0,1,2\},$ from which we expect to reconstruct $C_i(0)$.} The fact that the measurement only begins for $i\geq i_0+1$ and does not include $i_0$ is meant to model the fact that smallest polymers are undetectable, the detection level being identified as $i_0$.  From~\eqref{eq:depol-ode}, we compute the differential system for the moments:
\begin{equation}\label{eq:momentsBD:exact}\f{\diff M_k}{\diff \tau}=-b\sum\limits_{i_0+1}^\infty \left(i^k-(i-1)^k\right)C_i - bi_0^k C_{i_0+1},
\end{equation}
so that for the three first moments we have
\begin{equation}\label{eq:observ-BD0}
	\frac{\diff}{\diff \tau} M_k = 
	\begin{cases}
		\inserted{-}b C_{i_0+1} (\tau), & \text{if } k=0, \\
		- b   M_{0}(\tau)-b i_0 C_{i_0+1} (\tau), & \text{if } k=1, 	
		 \\
		- 2b   M_{1} +bM_0 -bi_0^2 C_{i_0+1}, & \text{if } k= 2.	
	\end{cases}
\end{equation}
To make the link with the first and second order approximate systems, we introduce the appropriate scaling for the moments, defining
\[M_k^\ep (t)\coloneqq \ep^{k+1} M_k(\f{t}{\ep})=\ep\sum\limits_{i=i_0+1}^\infty (\ep i)^k c_i^\ep ({t}) .\]
From~\eqref{eq:momentsBD:exact}, we derive the system satisfied by $\mu_k^\ep$ and neglect the terms of the kind $\ep\sum\limits_i \ep^k i^\ell$ with $\ell \leq k-2,$ since they are (formally) of order at least $\ep^2$. We obtain:
\begin{equation}\label{eq:observ-BD}
	\frac{\diff}{\diff t} M_k^\ep = 
	\begin{cases}
		\inserted{-}b c_{i_0+1}^\ep (t), & \text{if } k=0, \\
		- b   M^\ep_{0}(t)-\ep b i_0 c_{i_0+1}^\ep (t), & \text{if } k=1, 	
		 \\
		- 2 b   M^\ep_{1} + \ep b M_{0}^\ep + O(\ep^2), & \text{if } k= 2.	
	\end{cases}
\end{equation}
}

\subsection{The pure advection setting of the inverse problem}
\label{subsec:inverse:advection}
Let us start by recalling some results obtained in \cite{armiento:2016} when using the pure advection asymptotic model \eqref{eq:first-order}. 
We typically consider that there exists $L$ such that
$\text{supp}(\solFinf(0,\cdot)) \subset [0,L)$ imposing that  
\[
	\forall t\in [0,T], \quad \text{supp}(\solFinf(t,\cdot)) \subset [0,L)
\]
so that we can rely on the bounded, hence computable, version of \eqref{eq:first-order}, namely
\begin{equation}
\label{eq:first-order-L}
\left\{
\begin{aligned}
	\partial_t \solF - b \partial_x \solF&=0,&  (t,x) \in  [0,T]\times [0,L] ,\\
	\solF(0,x) &= \uin(x), & x \in [0,L], \\
	\inserted{\solF(t,L)}&\inserted{=0,}& \inserted{t\in [0,T]. }
\end{aligned}\right.
\end{equation}
From the method of characteristics we have
\begin{equation}\label{eq:initial-order-0}
	\solF(t,0) = \solF(bt,0) = \solF_0(bt), \quad t \in [0,T].
\end{equation}
\inserted{The question is now to relate the continuous moments for the transport equation solution, that we denote $\mu_k$, to the moments of the discrete system $M_k^\ep.$ For this first order approximation, since we neglect all terms of order $\ep$ and above, we can identify $\mu_k$ with $M_k^\ep$ and $c_{i_0+1}^\ep$ with $\solF(t,0).$ We obtain the system:
\begin{equation}\label{eq:observ-LS}
	\frac{\diff}{\diff t} \mu_k = 
	\begin{cases}
		\inserted{-}b u(t,0), & \text{if } k=0, \\
		- b   \mu_{0}(t)& \text{if } k=1, 	
		 \\
		- b k  \mu_{k-1}  & \text{if } k\geq 2.	
	\end{cases}
\end{equation}
We notice that this system is identical to the one obtained directly by multiplying~\eqref{eq:first-order-L} by $x^k$ and integrate in space.}
Therefore, to reconstruct $\solF_0$ from an observed $M_k$, identified (up to the rescaling) with $\mu_k,$ with $k \in \{0,1,2\}$ given, we have the stability estimate
\begin{equation}\label{eq:stability-order-0}
	\forall T > T_0 \coloneqq \frac{1}b, \quad \| \solF^0 \|^2_{\Ltwo} \lesssim \norm{\frac{\diff^{k{+1}}}{\diff t^{k{+1}}} \inserted{\mu}_k}_{\Ltwo[0,T]}^2.
\end{equation}
This leads us in \cite{armiento:2016} to saying that if we define 
\[
	\Psi_T^{\solF^0 \to {\inserted{\mu}_{k}}} : \left|
	\begin{aligned}
		\Ltwo[0,L] &\to \Ltwo[0,T] \\
		\solF^0 &\mapsto y : ([0,T] \in t \mapsto  \inserted{\mu}_k(t) \in \R)
	\end{aligned}\right.
\]
inverting $\Psi_T^{\solF^0 \to \inserted{\mu}_{k}}$ is  ill-posed of order $k+1$, sometimes called \emph{mildly ill-posed}~\cite{wahba1980ill}.

But we could also divide the inverse problem into two successive inverse problems by introducing 
\[
	\Psi_T^{\solF^0 \to \text{Tr}} : \left|
	\begin{aligned}
		\Ltwo[0,L] &\to \Ltwo[0,T] \\
		\solF^0 &\mapsto y : ([0,T] \ni t \mapsto  \solF(t,0) \in \R)
	\end{aligned}\right.
\]
where $\solF$ is solution of \eqref{eq:first-order-L}
and then
\[
	\Psi_T^{\text{Tr} \to \inserted{\mu}_{k}} : \left|
	\begin{aligned}
		\Ltwo[0,L] &\to \Ltwo[0,T] \\
		\solF(t,0) &\mapsto y : ([0,T] \ni t \mapsto \inserted{\mu}_k(t)  \in \R)
	\end{aligned}\right.
\]
where $\inserted{\mu}_k$ is solution of \eqref{eq:observ-LS}, so that we have
\[
	\Psi_T^{\solF^0 \to \mu^\ep_{k}} = \Psi_T^{\text{Tr} \to \inserted{\mu}_{k}} \circ \Psi_T^{\solF^0 \to \text{Tr}}.
\]
Here inverting $\Psi_T^{\text{Tr} \to \mu_{k}}$ is a mildly ill-posed problem of order $k{+1}$, whereas inverting $\Psi_T^{\solF^0 \to \text{Tr}}$ is well-posed as a direct consequence of \eqref{eq:initial-order-0}. 

\subsection{The advection-diffusion inverse problem}
\label{subsec:inverse:advection-diffusion}
We  proceed with the study of the same reconstruction problem using this time the advection-diffusion problem which was proved to be a more accurate ``direct'' model. 
\inserted{Expecting an accuracy of order at least $\ep^{3/2},$ we keep the term of order $\ep$ in~\eqref{eq:observ-BD}, and as for the pure advection approximation we identify  $c_{i_0+1}^\ep (t)$ with $\solS(t,0).$ 
We define
\[\mu_k^\ep\coloneqq \int\limits_0^\infty x^k \solS(t,x)\diff x. \]
Contrarily to the first order approximation, we cannot directly identify $M_k^\ep$ with $\mu_k^\ep.$ Carrying out the same asymptotic expansion as done in Section~\ref{subsec:direct:advection-diffusion} to justify the transport boundary condition, we obtain
\begin{equation}\label{eq:equivmoments}M_0^\ep = \mu_0^\ep +\f{\ep}{2} \solS(t,0)+O(\ep^2),\qquad M_k^\ep = \mu_k^\ep +k(i_0+1) \ep \mu_{k-1}^\ep+ O(\ep^2) \quad {\text{if }}k\geq 1.  \end{equation}
To obtain a system equivalent to~\eqref{eq:observ-BD} for $\mu_k^\ep,$ we may either multiply~\eqref{eq:second-order} with $x^k$ and integrate in space or use the above equality and replace in~\eqref{eq:observ-BD}. We get
\begin{equation}\label{eq:observ-AdvDiff}
	\frac{\diff}{\diff t} \mu_k^\ep  = 
	\begin{cases}
		\inserted{-}b \solS (t,0)-\f{\ep }{2}\p_t \solS(t,0), & \text{if } k=0, \\
		- b  \mu_0^\ep  +\f{\ep b}{2}  \solS (t,0), & \text{if } k=1, 	
		 \\
		- 2b \mu_1^\ep + \ep b\mu_0^\ep   & \text{if } k=2.
	\end{cases}
\end{equation}
We can thus follow two heuristically equivalent strategies in order to estimate $\solS(t,0)$ - other said, to invert $\Psi_{T,\vareps}^{\text{Tr} \to \inserted{M}^\ep_{k}}$.
\begin{enumerate}
\item From the observation of $M_k^\ep,$ invert the system~\eqref{eq:observ-BD} to estimate $c_{i_0+1}^\ep (t),$ then use $c_{i_0+1}^\ep (t) =\solS(t,0)+O(\ep^2).$
\item From the observation of $M_k^\ep,$ link the discrete moments $M_k^\ep$ to the continuous ones $\mu_k^\ep$ and the boundary $\solS(t,0)$ thanks to~\eqref{eq:equivmoments}; use these relations, together with the system~\eqref{eq:observ-AdvDiff},  to estimate $\solS(t,0)$ from  $M_k^\ep$.
\end{enumerate}
In practice, we use a third "blind" strategy, which inverts directly $\Psi_{T,\vareps}^{\uin \to \inserted{M}^\ep_{k}}$ without using the  decomposition 
\[ \Psi_{T,\vareps}^{\text{Tr} \to \inserted{M}^\ep_{k}}=   \Psi_{T,\vareps}^{\text{Tr} \to \inserted{M}^\ep_{k}}\circ \Psi_{T,\vareps}^{\uin \to \text{Tr}},\]
 see Section~\ref{sec:num}. 
\inserted{As for the first-order inverse problem, we} decompose the inverse problem into
\[
	\Psi_{T,\vareps}^{\solF^0 \to \text{Tr}} : \left|
	\begin{aligned}
		\Ltwo[0,L] &\to \Ltwo[0,T] \\
		\uin &\mapsto y : ([0,T] \ni t \mapsto  \solS(t,0) \in \R)
	\end{aligned}\right.
\]
where $\solS$ is solution of
\begin{equation}\label{eq:second-order-L}
\left\{
\begin{aligned}
        \partial_t \solS-b \partial_x\solS- \f{b\ep}{2}\partial_{x}^2 \solS &= 0, & (t,x)\in [0,T]\times [0,L], \\
        \partial_t\solS(t,0) - b \partial_x \solS(t,0) &=0, & t\in[0,T],  \\
        \solS(t,L) &=0, & t\in[0,T], \\
        \solS(0,x) &= \uin(x),&  x\in (0,L),
\end{aligned}
\right.
\end{equation}
and then, as before,
\[
	\Psi_{T,\vareps}^{\text{Tr} \to \mu_{k}^\ep} : \left|
	\begin{aligned}
		\Ltwo[0,\inserted{T}] &\to \Ltwo[0,T] \\
		\solSinf(t,0) &\mapsto y : ([0,T] \ni t \mapsto \inserted{\mu_k^\ep}(t)  \in \R)
	\end{aligned}\right.
\]
where $\inserted{\mu_k^\ep}$ is solution of \eqref{eq:observ-AdvDiff}  so that we can ultimately invert 
\[
	\Psi_{T,\vareps}^{\uin \to \inserted{\mu}^\ep_{k}} =   \Psi_{T,\vareps}^{\text{Tr} \to \inserted{\mu}^\ep_{k}}\circ \Psi_{T,\vareps}^{\uin \to \text{Tr}}.
\]

We would like to point out that in the advection setting the operator $\Psi_{T,\vareps}^{{\text{Tr} \to \inserted{\mu}^\ep_k}}$ inversion is now ill-posed of degree {$\lfloor \f{k+1}{2} \rfloor$} \inserted{for $k\geq 0$,} as opposed to {$k+1$} in the advection setting, \inserted{due to the non-neglected term of order $\ep$ in the equations for $\mu_k^\ep:$ the degree of ill-posedness to invert $\Psi_{T,\vareps}^{\text{Tr} \to \mu_{k}^\ep}$ is $+1$ the degree of ill-posedness to invert $\Psi_{T,\vareps}^{\text{Tr} \to \mu_{k-2}^\ep}$ for $k\geq 2,$ and is equal to $0$ for $k=0$ and to $1$ for $k=1$.} 
Indeed, we have from \eqref{eq:observ-AdvDiff} that 
\inserted{\[
	\frac{\diff^2}{\diff t^2} \mu_1^\ep = b^2  \solS(t,0) - \ep b (i_0-\f{1}{2}) \p_t  \solS(t,0) =-\ep b (i_0-\f{1}{2}) \p_t \left(\solS(t,0)e^{-\f{b}{\ep (i_0-1/2)} t} \right)e^{\f{b}{\ep (i_0-1/2)} t},
\]}
which gives 
\begin{align*}
	\norm{\solS(t,0)}_{\Ltwo[0,T]}^2 +  |\solS(0,0)|^2 & \lesssim \Cst(\vareps^{-1}) \norm{ \int_0^t  \frac{\diff}{\diff t} \Bigl(e^{\inserted{-\frac{bt}{\varepsilon (i_0-1/2)}}} \solS(t,0) \Bigr) \diff t }^2 \\
	& \lesssim \Cst(\vareps^{-1}) \norm{\int_0^t e^{\frac{-bt}{\varepsilon (i_0-1/2)}} \frac{\diff^2}{\diff t^2} \mu_1^\ep \diff t}^2 \\
	& \lesssim \Cst(\vareps^{-1})  \Bigl[\norm{\deriv{t}\mu_1^\ep}_{\Ltwo[0,T]}^2 + |\deriv{t} \mu_1^\ep(0)|^2 \Bigr],
\end{align*}
On the left-hand side we have a $\mathrm{H}^{\frac{1}2}$-like norm controlled by a $\mathrm{H}^{\frac{3}2}$-like norm of the measured moment $\mu_0$.  Then \eqref{eq:observ-AdvDiff} shows that we lose one order of derivative for $\mu_{k}^\ep$ compared to $\mu_{k-2}^\ep$.
The stability constant however explodes when $\vareps$ decreases as expected from the advection setting result: in such a case we are back to the pure transport problem, mildly ill-posed of order \inserted{two for $k=1$}.

However, we prove in the following sections that the inversion of $\Psi_{T,\vareps}^{\uin \to \text{Tr}}$, which was a well-posed problem for the \inserted{pure advection} problem, is severely ill-posed due to the diffusion, hence the non-reversibility of the dynamics. This obligates us to a compromise stability-precision in the initial condition reconstruction using asymptotic models. 

\subsection{Analysis of the advection-diffusion problem on a bounded domain}
{\inserted{Before turning to the inverse problem solution, let us study the direct problem~\eqref{eq:second-order-L}, since two inequalities are necessary to prove observability. In Proposition~\ref{prop:diffusion-energy-estimate}, we prove a dissipativity inequality, that is  useful for the proof of a final time observability  (see Section~\ref{sec:carleman}) ;  the proof also provides us with the well-posedness of~\eqref{eq:second-order-L}, obtained in a similar manner as for~\eqref{eq:second-order}. In Proposition~\ref{prop:logconv}, we prove a reverse inequality bounding the initial condition by the solution at time $t$: this is used in Section~\ref{sec:inverse:anal}, Theorem~\ref{th:initobserv} to obtain an initial time observability.}

\label{subsec:expostab}

As in Section~\ref{sec:second-order}, we introduce the space
$\stateSpace =   \Ltwo\times \R$, equipped with the norm
\[
    \forall \state = (u,v) \in \stateSpace, \quad \norm{\state}_{\stateSpace}^2 = \norm{u}_{\Ltwo}^2 + |v|^2,
\]
and the
operator $(\modelOp,\mathcal{D}(\modelOp))$ defined by
\begin{equation}\label{eq:operator-L}
    \forall z = \begin{pmatrix}
    u \\ v
    \end{pmatrix} \in \mathcal{D}(A), \quad 
    \modelOp \state = \begin{pmatrix}
    - b u' - \f{b \vareps}{2} u'' \\
    - b\sqrt{\f{\ep}{2}}  u'(0)
    \end{pmatrix},
\end{equation}
where
\begin{equation}\label{eq:operator-domain-L}
    \mathcal{D}(\modelOp) = \left\{ \state=\begin{pmatrix}u \\ v \end{pmatrix}\in \stateSpace \, \Big| \, u \in H^2(0,L) \cap \HoneR, \, \sqrt{\f{\varepsilon}{2}}\, u(0) =  v \right\},
\end{equation}
where $\HoneR=\{u\in \Hone,\;u(L)=0\}$ and
such that a solution $\sol$ of \eqref{eq:second-order-L} is defined from the $\state = \begin{pmatrix}\sol \\ v \end{pmatrix}$ solution of the initial value problem
\begin{equation}\label{eq:semigroup-L}
\begin{cases}
	\dsp \deriv{t} \state + \modelOp \state = 0, \quad t > 0 \\
	\state(0) = \state^0 
\end{cases}
\end{equation}
with $z^0 = \begin{pmatrix} \sol^0(x) \\ \sqrt{\frac{\vareps}{2}} \sol^0(0) \end{pmatrix}$. The proof of the well-posedness is then identical to that of Proposition~\ref{prop:second-order:exist} carried out on $\R^+$ instead of $(0,L),$ and we moreover have a dissipativity estimate, given in the following proposition. 

\begin{proposition}\label{prop:diffusion-energy-estimate}
We have the exponential stability estimate
\begin{equation}\label{ineq:expostab}
	\norm{\state(t)}_{\stateSpace} \leq e^{\frac{2L-bt}{\vareps} } \norm{\state(0)}_\stateSpace,  
\end{equation}
for any mild solution $\state \in C^0([0,T],\stateSpace)$ of~\eqref{eq:operator-L}. 
\end{proposition}
 \begin{proof}
We first assume the initial condition $\state^0 \in \mathcal{D}(\modelOp)$, and we  conclude by density. The proof is decomposed in 3 steps. We (1) proceed to a change of unknown \inserted{which symetrises the operator} in order to (2) use a spectral decomposition on the new system and (3) obtain energy estimates for the initial problem. \\

\textup{(1)}\hskip\labelsep \emph{Change of unknown}~-- Let us introduce the function 
\[
	\solShift(t,x) = e^{{\f{x}{\ep}}} \sol(t,x),
\]
so that
\[
	\partial_x \sol(t,x) = -\frac{1}{\vareps}  e^{\frac{-x}{\vareps}} \solShift(t,x) + e^{-\frac{x}{\vareps}} \partial_x \solShift(t,x)
\]
and
\[
	\partial_{xx}^2 \sol(t,x) = \frac{1}{\vareps^2} e^{\frac{-x}{\vareps}} \solShift(t,x) - \frac{\inserted{2}}{\vareps} e^{-\frac{x}{\vareps}} \partial_x \solShift(t,x) + e^{\frac{-x}{\vareps}} \partial_{xx}^2 \solShift(t,x).
\]
Therefore, $\solShift$ is the solution of 
\begin{equation}\label{eq:new-unknown-model}
\left\{
\begin{aligned}
        \partial_t \solShift(t,x)  + \frac{b}{2\vareps} \solShift(t,x) - \f{b\vareps}{2} \partial_{xx}^2 \solShift(t,x) &= 0, & (t,x)\in  [0,T]\times (0,L), \\
		\solShift(t,L) &= 0, & t\in[0,T],         \\ 
        \partial_t \solShift(t,0) - b \partial_x \solShift(t,0) + \frac{b}{\vareps}\solShift(t,0) &=0, & t\in[0,T],  \\
        \solShift(0,x) &= e^{\frac{x}{\vareps}}\sol^0(x),&  x\in (0,L).
\end{aligned}
\right.
\end{equation}
We hence introduce the operator $(\modelOpShift,\mathcal{D}(\modelOpShift))$ defined by
\begin{equation}\label{eq:operator-new-unknown}
    \forall z = \begin{pmatrix}
    \sol \\ v
    \end{pmatrix} \in \mathcal{D}(\modelOpShift), \quad 
    \modelOpShift \state = \begin{pmatrix}
    \frac{b}{2\vareps} \sol - \f{b\vareps}{2} \sol'' \\[0.1cm]
   {-b\sqrt{\f{\ep}{2}}  u'(0) + \f{b}{\ep}v} 
    \end{pmatrix},
\end{equation}
where $\mathcal{D}(\modelOpShift) = \mathcal{D}(\modelOp)$. The associated bilinear form
 is
\[
	\modelFormShift((\sol,v),(w,k)) = \f{b \vareps}{2} (u', w')_{\Ltwo} + \frac{b}{2\varepsilon}(u,w)_{\Ltwo}  + \frac{b}{\varepsilon}vk,
\]
 defined on 
$${\mathcal V}\coloneqq\Bigl\{ (u,v)\in \HoneR \times \R \, | \, \sqrt{\f{\ep}{2}} u(0) = v \Bigr\}$$
since it is such that
\[
\begin{split}
	\forall \state = \begin{pmatrix}
    \sol \\ v
    \end{pmatrix}&\in \mathcal{V}, \adjoint =  \begin{pmatrix}
    w \\ k
    \end{pmatrix}\in \mathcal{V}, \\
    (\modelOpShift z,q) &= -\int_0^L \f{b\vareps}{2} \sol''(x) w(x)\, \diff x + \int_0^L \frac{b}{2\vareps} \sol(x) w(x) \,\diff x  - \sqrt{\f{\ep}{2}} b u'(0) k + \frac{b}{\vareps} v k \\[8pt]
    &=\int_0^L \f{b\vareps}{2} \sol' w'\, \diff x + \f{b\varepsilon}{2} u'(0)w(0) + \int_0^L \frac{b}{2\vareps} \sol w\,\diff x  - \sqrt{\f{\ep}{2}} b u'(0) k + \frac{b}{\vareps} v k \\[8pt]
    &=\tilde{a}((\sol,v),(w,k)).
\end{split}
\]
We deduce that $\modelOpShift$ is a symmetric operator. Moreover, as  $\modelFormShift$ is coercive in $\mathcal{V}$, we prove as in Proposition~\ref{prop:second-order:exist} that $(\modelOpShift,{\mathcal D}(\modelOpShift))$ is maximal accretive.

\textup{(2)}\hskip\labelsep \emph{Spectral decomposition}~-- From \cite[Proposition 7.6]{brezis-book}, $\modelOpShift$ is therefore a self-adjoint positive operator. Moreover, $\modelOpShift$ is invertible with  $\modelOpShift^{-1} : \mathcal{\stateSpace} =  \Ltwo\times \R  \to \mathcal{D}(\modelOp) \subset  H^2(0,L)\times \R $, hence $\modelOpShift^{-1} \in \mathcal{L}(\stateSpace,\stateSpace)$ is compact from Rellich–Kondrachov theorem. We conclude that there exists a Hilbert basis of $\stateSpace$ composed of eigenvectors $(\varphi_n)_{n\in\N}$ of $\modelOpShift$ associated with a sequence of positive eigenvalues $(\lambda_n)_{n\in \N}$ with $\lambda_n^{-1} \xrightarrow{n\to+\infty} 0$. We additionally have that 
\begin{align}
	\forall z = \begin{pmatrix}
    \sol \\ v
    \end{pmatrix} \in {\mathcal{D}}(\modelOpShift), \quad \modelFormShift((\sol,v),(\sol,v)) &\geq \frac{b}{2\varepsilon}(u,u)_{L^2(0,L)}  + \frac{b}{\varepsilon} v^2 \notag\\
	&\geq \frac{b}{2\varepsilon} \norm{z}_{\stateSpace}^2
\end{align}
Hence, we have for all $n\in\N$,  $\lambda_n \geq \frac{b}{2\vareps}$.

\textup{(3)}\hskip\labelsep \emph{Energy estimate}~--
We denote $\varphi_n = (\mu_n,\kappa_n)$ the eigenvector decomposition in $\stateSpace$.
We have
\[
	\forall \state = \begin{pmatrix}
    \sol \\ v
    \end{pmatrix} \in \stateSpace, \quad \state = \sum_{k\in\N} (\state,\varphi_n)_\stateSpace \varphi_n = \sum_{k\in\N} \Bigl[ (\sol,\mu_n)_{\Ltwo} + v\kappa_n \Bigr] \begin{pmatrix}
    \mu_n \\ \kappa_n
    \end{pmatrix}.
\]
For any $\tilde{\state}^0 \in \stateSpace$, we deduce that $\tilde{\state}(t) = e^{-t \modelOpShift} \tilde{\state}^0$ satisfies
\[
	\tilde{\state} = \begin{pmatrix} \solShift \\ \tilde{v}\end{pmatrix}= e^{-t \modelOpShift} \solShift^0 = \sum_{k\in\N} e^{-\lambda_n t} (\tilde{\state}^0,\varphi_n)_\stateSpace \varphi_n.
\]
Therefore, we have
\[
\begin{aligned}
		\norm{\solShift(t)}_{L^2(0,L)}^2 + \abs{\tilde{v}(t)}^2 
		&\leq  e^{-2\lambda_0 t}\left(  \norm{\solShift^0}_{\Ltwo}^2+\vert \t v^0\vert^2\right).
\end{aligned}
\]
Moreover $\norm{\sol(t)}_{L^2(0,L)} \leq \norm{\solShift(t)}_{L^2(0,L)}$ and similarly
\[
	\norm{\solShift^0}_{\Ltwo}^2 = \int_0^L e^{2\frac{x}{\vareps}} |\sol^0(x)|^2 \, \diff x \leq e^{2\frac{L}{\vareps}} \norm{\sol^0}_{\Ltwo}^2.
\]
As we also have $\tilde{v} = v$, we finally conclude that 
\[
	\norm{\sol(t)}_{L^2(0,L)}^2 + \abs{v(t)}^2 \leq e^{\frac{2L-bt}{\vareps} } \left(\norm{\sol^0}_{L^2(0,L)}^2 + \abs{v^0}^2\right).
\]
By density of $\mathcal{D}(\modelOp)$ in $\stateSpace$, the result is still valid for any mild solution $\state = (\solShift,v) \in C^0([0,T],\stateSpace)$. 
\end{proof}
}}
\inserted{The reverse inequality, bounding the initial condition with respect to the final time solution, is known to be a much more difficult problem for dissipative second-order equations. We can however obtain an estimate by log-convexity arguments that follow the ideas in~\cite{Bardos:1973aa}, see also \cite{Phung:2004aa,vo:tel-02081052}.  This result, together with Carleman's inequalities (Section~\ref{sec:carleman}), is useful to obtain the initial time observability result stated in Theorem~\ref{th:initobserv}.}

\begin{proposition}\label{prop:logconv}
Let $\solShift$ be a solution of \eqref{eq:new-unknown-model} with initial condition in $\HoneR$, denote $\t\state(t)=\begin{pmatrix}\solShift(t,\cdot) \\ \sqrt{\f{\ep}{2}} \solShift(t,0) \end{pmatrix},$ we have the following initial condition estimate
\begin{equation}\label{eq:backward-estimate}
\left\|\t\state(0)\right\|_{\stateSpace} \leq \exp \left(\frac{\left(\modelOpShift \t\state(0),\t\state(0)\right)_\stateSpace}{\left\|\t\state(0)\right\|_{\stateSpace}^{2}} T\right)\left\|\t\state(T)\right\|_{\stateSpace}.
\end{equation}

\end{proposition}
\begin{proof}

\inserted{Let us recall that by definition $\f{\diff \t\state}{\diff t}=-\modelOpShift \t\state$, and define 
$${\mathcal F}:=\f{(\modelOpShift \t\state,\t\state)_\stateSpace}{\left\|\t\state\right\|_{\stateSpace}^{2}}=-\f{\f{1}{2}\f{\diff}{\diff t}\left\|\t\state\right\|_{\stateSpace}^{2}}{\left\|\t\state\right\|_{\stateSpace}^{2}}=-\f{1}{2}\f{\diff}{\diff t} \log \left(\left\|\t\state\right\|_{\stateSpace}^{2}\right).$$ We prove that $\mathcal F$ is decreasing by computing, using that $\modelOpShift$ is self-adjoint on ${\mathcal D}(\modelOp)$, \begin{align*}\f{1}{2}\f{d{\mathcal F}}{dt}&=\f{(\modelOpShift \t\state,\f{d\t\state}{dt})_\stateSpace\left\|\t\state\right\|_{\stateSpace}^{2}-(\t\state,\f{d\t\state}{dt})_\stateSpace (\modelOpShift \t\state,\t\state)_\stateSpace}{\left\|\t\state\right\|_{\stateSpace}^{4}}
=\f{(\modelOpShift \t\state,-\modelOpShift \t\state+{\mathcal F}(t) \t\state)_\stateSpace}{\left\|\t\state\right\|_{\stateSpace}^{2}}
\end{align*} 
Since $(\t\state,-\modelOpShift \t\state + {\mathcal F} (t) \t\state)_\stateSpace=-(\modelOpShift \t\state,\t\state)_\stateSpace + {\mathcal F} (t) \left\|\t\state\right\|_{\stateSpace}^{2}=0,$ we have
\begin{align*}\f{1}{2}\f{d{\mathcal F}}{dt}&=\f{(\modelOpShift \state,-\modelOpShift \state+{\mathcal F}(t) \state)_\stateSpace}{\left\|\state\right\|_{\stateSpace}^{2}}=\f{(\modelOpShift \t\state-{\mathcal F}(t)\t\state,-\modelOpShift \t\state+{\mathcal F}(t) \t\state)_\stateSpace}{\left\|\t\state\right\|_{\stateSpace}^{2}}\leq 0.
\end{align*} 
We thus have 
$${\mathcal F}(t)=-\f{1}{2}\f{\diff}{\diff t} \log \left(\left\|\t\state\right\|_{\stateSpace}^{2}\right)\leq {\mathcal F}(0),$$
hence
$$-\f{1}{2}\log \left(\left\|\t\state(T)\right\|_{\stateSpace}^{2}\right)\inserted{+}\f{1}{2}\log\left(\left\|\t\state(0)\right\|_{\stateSpace}^{2}\right)
=\log\left(\f{\left\|\t\state(0)\right\|_{\stateSpace}}{\left\|\t\state(T)\right\|_{\stateSpace}}\right)
\leq {\mathcal F}(0) T,$$
from which, taking the exponential, we obtain~\eqref{eq:backward-estimate}. 

}

\end{proof}

\section{Inverse problem solution}
\label{sec:inverse:anal}

\inserted{The aim of this section is to invert the operator $\Psi_{T,\vareps}^{u_0 \to \text{Tr}}$ defined in Section~\ref{subsec:inverse:advection-diffusion}. The injective character of the inverse is obtained thanks to an observability inequality, stated in Theorem~\ref{th:initobserv}. This result leads us to define an appropriate generalised Tikhonov regularisation, for which we prove an error estimate in Theorem~\ref{th:tikhonov}. The first and main ingredient for the observability inequality lies in a Carleman-type inequality~\eqref{ineq:Carleman}, to which the following subsection is devoted.}
\label{sec:observ}
\subsection{Carleman inequality and initial time observability}
\label{sec:carleman}

In this section, we \inserted{first} prove the Carleman estimate~\eqref{ineq:Carleman}, inspired from controllability inequalities obtained in~\cite{Cornilleau:2012dba}, see also~\cite{Coron:2005vd}. Together with the previous exponential stability result~\eqref{ineq:expostab}, it provides us with the final time observability of our advection-diffusion system; together with the log-convexity estimate~\eqref{eq:backward-estimate}, the initial time observability~\eqref{estim:observ:final} is obtained. 
\begin{theorem}
Let $\uin \in \HoneR\inserted{\cap \Htwo}$.

Let $\sol  \in \Czero[(0,T);\HoneR] \cap \Cone[(0,T);\Ltwo]$ the unique solution of~\eqref{eq:second-order-L}. There exist two constants $c>0$ and $C_C>0$, with $C_C\sim \inserted{\ep^{-2}T^{-5} (1+T)^3}e^{\f{c}{\ep}(1+\TM^{-1})}$, such that
\begin{multline}\label{ineq:Carleman}
 \Vert u\Vert_{L^2([\f{T}{4},\f{3T}{4}]\times[0,L])}^2
\lesssim
C_{C} \int_0^{\TM} \abs{\sol}^2(t,0) \diff t \sim \inserted{\ep^{-2}T^{-5} (1+T)^3}e^{\f{c}{\ep}(1+\TM^{-1})} \int_0^{\TM} \abs{\sol}^2(t,0) \diff t .
\end{multline}\label{coroll:Carleman}
\end{theorem}
To prove this estimate, we \inserted{follow the scheme of proof of the controllability inequality from~\cite{Cornilleau:2012dba}.} We first rescale  the system by defining
\begin{equation}
\t \TM:=\f{b \ep}{2}\TM,\quad \t\sol (t,x):=u(\f{2t}{b\ep},x),
\end{equation}
hence the system becomes 
\begin{equation}
\label{eq:model:redim}
\left\{
\begin{array}{lcr}
\f{\p}{\p t} \t\sol =\f{2}{\ep}\f{\p}{\p x} \t\sol +\f{\p^2}{\p x^2} \t\sol,&& (t,x)\in \t\Omega:=[0,\t\TM]\times[0,\LM],
\\ \\
\f{\p}{\p t} \t\sol (t,0) =\f{2}{\ep}\f{\p}{\p x} \t\sol (t,0),&
\t\sol(t,\LM)=0, & t \in [0,\t\TM],
\\ \\ \t\sol(0,x)=\uin (x) &&x\in[0,\LM].
\end{array}
\right.
\end{equation}
We now work with an intermediate function $\psi$ which is the  solution $\t\sol$ weighted  as follows:
\begin{equation}\label{def:psi}
\psi(t,x):=\t u (t,x) e^{-s \alpha(t,x)}, 
\end{equation}
with $s>0 $ to be specified later and $ \alpha\inserted{>0}$ defined by
\begin{equation}
\label{def:alpha}
 \alpha(t,x)\coloneqq\f{\lambda - e^{\eta(x)}}{t(\t\TM -t)},\qquad \eta(x)\coloneqq2\LM -x, \qquad \phi\coloneqq\f{e^{\eta(x)}}{t(\t\TM - t)},\qquad \lambda>e^{2\LM}.
\end{equation}

\begin{lemma}[Estimates for the weight functions derivatives] 
We have $\eta(x)\in [L,2L],$ and assuming $e^{2L}\inserted{<}\lambda\leq e^{2L}+e^{L},$ we have $\p_x \alpha=\phi$, $\p^2_{xx} \alpha =-\phi$ and the following inequalities
\begin{equation}\begin{array}{c}
\abs{\p_t\alpha}\leq \t\TM \phi^2,\qquad\abs{\p_{xt}^2\alpha}\leq\t\TM \phi^2,\qquad\abs{\p^2_{tt}\alpha}\leq 2\t\TM^2\phi^3,
\\ \\ \abs{\p_t\phi}\leq \t\TM \phi^2,\qquad \abs{\p_{tt}^2\phi}\leq 2\t\TM^2 \phi^3,\qquad \vert \p^3_{ttt}\alpha\vert, \;\vert \p^3_{ttt}\phi\vert  \leq 6 \t\TM^3 \phi^4.
\end{array}\label{ineq:alphax}
\end{equation}\label{lem:alphax}
\end{lemma}
\begin{proof}

We compute $\eta'(x)=-1$ hence
$$\displaystyle{\begin{array}{lcr}\p_x \alpha=\phi,& \p^2_{xx}\alpha =-\phi, &\p_t \alpha=\f{(\lambda -e^{\eta(x)} )(-\t\TM +2t)}{t^2(\t\TM-t)^2},
\\ \\
\p_{tx}^2 \alpha=\p_t \phi=\f{ e^{\eta(x)}(-\t\TM+2t)}{t^2(\t\TM-t)^2}, &&
\p_{tt}^2 \alpha=\f{2(\lambda -e^{\eta(x)} )(3t^2 +\t\TM^2-3t\t\TM)}{t^3(\t\TM-t)^3},
\\ \\
\p^3_{ttt}\alpha=\f{6(\lambda-e^{\eta(x)})(4t^3-6t^2\t\TM+4t\t\TM^2-\t\TM^3)}{t^4(\t\TM-t)^4}.
\end{array}}
$$

Since $e^{\eta(x)}\leq e^{2\eta (x)}$ and $\abs{-\t\TM + 2t} \leq \t\TM$ we have $\abs{\p_{tx}^2 \alpha}=\abs{\p_t\phi} \leq \t\TM \phi^2.$ 

For $\lambda\leq e^{2L}+e^L$ we have $\lambda -e^{\eta(x)} \leq e^{2\eta (x)}$ hence $\abs{\p_t \alpha}\leq \t\TM \phi^2.$

We compute similarly $\abs{\p_{tt}^2 \alpha}\leq 2\t\TM^2\phi^3$ and $\abs{ \p^3_{ttt}\alpha} \leq 6 \t\TM^3 \phi^4.$ 

\end{proof}
Let us now state a first intermediate inequality for the rescaled system.
\begin{proposition}\label{prop:12fromCornilleauGuerrero}
There exists a constant $C>0$ such that, for every $\ep\in (0,1)$ and $s\geq C(\inserted{\ep}+\tT + \tT^2/\ep),$ we have the following inequality
\begin{multline}\label{ineq:14CornilleauGuerrero}
\iint_{\tO} \left\{s^{-1}\phi^{-1}e^{-2s\alpha}\left(\abs{\p_t\t\sol}^2
+\abs{\p^2_{xx}\t\sol}^2 \right) +  \left(s^3\phi^3 e^{-2s\alpha} \abs{\t\sol}^2 +s\phi e^{-2s\alpha} \abs{\p_x\t\sol}^2\right)\right\}\diff x\diff t 
\\
+ s\int_0^{\t \TM}\phi e^{-2s\alpha}\abs{\p_x\t\sol}^2(t,L)\diff t
\\ 
\lesssim \inserted{\ep^2} s^{\inserted{5}}\int_0^{\t\TM} \phi^5 e^{-2s\alpha}\abs{\t\sol}^2(t,0) \diff t 
+ s{\ep^2}\int_0^{\t\TM} \phi  e^{-2s\alpha}\abs{\p_t\t\sol}^2(t,0)\diff t.
\end{multline}
\end{proposition}

\begin{proof} \inserted{We prove the result for smooth initial conditions; the final result follows by a density argument.}
We first compute the equation satisfied by $\psi:$
$$\p_t \psi + s(\p_t \alpha )\psi = \f{\inserted{2}}{\ep} \left(\p_x \psi + s(\p_x\alpha) \psi\right) + \p_{xx}^2 \psi + s(\p_{xx}^2 \alpha) \psi + 2 s(\p_x\alpha)\p_x \psi + s^2(\p_x\alpha)^2\psi,$$
that we write under the form
\begin{equation}\label{eq:psi}
P_1\psi + P_2\psi = P_3\psi
\end{equation}
with $P_1,$ $P_2$ and $P_3$ defined by
\begin{equation}\label{def:P123}
\begin{array}{lll}
P_1&:=&\p_t\psi -2s(\p_x\alpha)\p_x \psi-\f{\inserted{2}}{\eps} \p_x\psi,
\\ \\
P_2&:=&-\p_{xx}^2\psi-s^2(\p_x\alpha)^2\psi+s(\p_t\alpha)\psi-\f{\inserted{2}}{\eps}s(\p_x\alpha)\psi,
\\ \\
P_3&:=&s(\p^2_{xx}\alpha)\psi.
\end{array}
\end{equation}
The boundary conditions are
\begin{equation}\label{eq:psi:bound}
\left\{\begin{array}{ll}
\psi(t,L)=0, \\ \\
\p_t\psi(t,0) +s\left(\p_t \alpha(t,0)\right)\psi(t,0)=\f{2}{\eps}\biggl(\p_x\psi(t,0)+s(\p_x\alpha(t,0))\psi(t,0)\biggr).
\end{array}\right.
\end{equation}
We take the $\Ltwo[\tO]$ square norm of~\eqref{eq:psi} to find on $\tO\coloneqq[0,\t\TM]\times[0,L]:$
\begin{equation}
\label{eq:psiL2}
\norme{P_1\psi}_{\Ltwo[\tO]}^2 +\norme{P_2\psi}_{\Ltwo[\tO]}^2 +2\langle P_1\psi,P_2\psi\rangle_{\Ltwo[\tO]} = \norme{P_3\psi}_{\Ltwo[\tO]}^2.
\end{equation}

\

{\bf First step: estimate for $P_3$.}

By Lemma~\ref{lem:alphax} we have
\begin{equation}
\label{ineq:P3}
\norme{P_3\psi}_{\Ltwo[\tO]}^2= s^2 \iint_{\tO} \phi^2\abs{\psi}^2\diff x \diff t.
\end{equation}

{\bf Second step: development of the scalar product} $\langle P_1\psi,P_2\psi\rangle_{\Ltwo[\tO]}.$

We develop the scalar product and define
$$
\begin{array}{lcl}
\langle P_1\psi,P_2\psi\rangle_{\Ltwo[\tO]}&=&\langle
\p_t\psi -2s(\p_x\alpha)\p_x \psi-\f{2}{\eps} \p_x\psi, \\ \\
&&
-\p_{xx}^2\psi-s^2(\p_x\alpha)^2\psi+s(\p_t\alpha)\psi-\f{2}{\eps}s(\p_x\alpha)\psi\rangle
\\ \\
&:=&\sum\limits_{i=1}^3\sum\limits_{j=1}^{4} T_{ij}(\psi).
\end{array}$$
We compute, by using integration by parts and the property $\psi(0,x)=\psi(\t\TM,x)=0$ (and the same for the derivatives: $\p_x \psi(0,x)=\p_x\psi(\t\TM,x)=0$ etc.):
$$\begin{array}{ll}
 T_{11}
&:=
\langle \p_t\psi,-\p^2_{xx} \psi\rangle_{\Ltwo[\tO]}
 \\ \\ 
 &=
\f{1}{2} \dst \int_0^{\tT} \f{d}{\diff t} \norme{\p_x \psi (t,\cdot)}^2  \diff t+ \int_0^{\tT} \left( \p_t \psi \p_x \psi (t,0) -\p_t\psi \p_x \psi(t,L) \right)\diff t
\\ \\
&=  \dst \int_0^{\tT}  \p_t \psi (t,0)\left(-s(\p_x\alpha)\psi(t,0) +\f{\ep}{2}\left(\p_t\psi (t,0)+s(\p_t\alpha)\psi(t,0)\right)\right) \diff t
\\ \\
&=
+\f{s}{2} \dst \int_0^{\tT} (\p_{tx}^2\alpha -\f{\ep}{2}\p^2_{tt}\alpha) \abs{\psi(t,0)}^2\diff t
+\f{\ep}{2} \int_0^{\tT}  \abs{\p_t \psi (t,0)}^2 \diff t 
\\ \\
&\gtrsim - \inserted{\f{s}{2}} \dst \int_0^{\tT} (\tT \phi^2 +{\ep} \tT^2\phi^3) \abs{\psi(t,0)}^2\diff t {\color{black}+\f{\ep}{2} \int_0^{\tT}  \abs{\p_t \psi (t,0)}^2 \diff t}. 
\end{array}$$

The next three terms are treated in a similar way:
$$\begin{array}{ll}T_{12}
&:=\langle \p_t\psi,-s^2(\p_x\alpha)^2\psi\rangle_{\Ltwo[\tO]} 
\\ \\
&=s^2\dst\iint_{\tO} (\p_x\alpha) (\p^2_{tx}\alpha) \abs{\psi}^2 \diff x\diff t -s^2\int_0^L \left[(\p_x\alpha)^2 \abs{\psi}^2(\cdot,x)\right]_{t=0}^{t=\tT}\diff x 
\\ \\
&=s^2\dst\iint_{\tO} (\p_x\alpha) (\p^2_{tx}\alpha) \abs{\psi}^2 \diff x\diff t  \gtrsim {\color{myred}-s^2 \iint_{\tO} \t\TM \phi^3 \abs{\psi}^2 \diff x\diff t},
\end{array}
$$

$$ T_{13}:=\langle \p_t\psi,s(\p_t\alpha)\psi\rangle_{\Ltwo[\tO]}=-\f{s}{2}\dst\iint_{\tO} \p^2_{tt}\alpha \abs{\psi}^2\diff x\diff t\gtrsim {\color{myred}- s \iint_{\tO} \t\TM^2\phi^3 \abs{\psi}^2\diff x\diff t},$$
 and
 $$T_{14}:=\langle \p_t\psi,-\f{2}{\eps}s(\p_x\alpha)\psi\rangle_{\Ltwo[\tO]}=\f{\inserted{s}}{\eps} \dst\iint_{\tO} (\p^2_{xt}\alpha)\abs{\psi}^2 \diff x\diff t\gtrsim {\color{myred}-\f{\inserted{s}}{\ep} \iint_{\tO} \t\TM \phi^2 \abs{\psi}^2 \diff x\diff t}.$$
 For all the remaining terms, we integrate by parts in $x$ as for $T_{11}:$
 $$\begin{array}{ll}T_{21}&:=\langle -2s(\p_x\alpha)\p_x \psi, 
 -\p_{xx}^2\psi\rangle_{\Ltwo[\tO]}
 \\ \\
 &=
- s\dst\iint_{\tO} \p^2_{xx}\alpha \abs{\p_x\psi}^2 \diff x\diff t
 +s\int_0^{\t T}  \left((\p_x\alpha) \abs{\p_x\psi}^2(t,L) - (\p_x\alpha) \abs{\p_x\psi}^2(t,0)\right)\diff t
 \\ \\
 &=
  {\color{mygreen}s\dst\iint_{\tO} \phi \abs{\p_x\psi}^2 \diff x\diff t
 +s\int_0^{\t T}  \phi(t,L) \abs{\p_x\psi}^2(t,L)\diff t} {\color{myorange}- s\dst\int_0^{\t\TM}\phi (t,0)\abs{\p_x\psi}^2(t,0)\diff t},
 \end{array}$$

 $$\begin{array}{ll}T_{22}&:=\langle -2s(\p_x\alpha)\p_x \psi,
 -s^2(\p_x\alpha)^2\psi \rangle_{\Ltwo[\tO]}
 \\ \\
 &=-3s^3 \dst\iint_{\tO}(\p_{xx}^2\alpha)(\p_x\alpha)^2 \abs{\psi}^2\diff x\diff t + s^3\int_0^{\t T}\left((\p_x\alpha)^3 \abs{\psi}^2(t,L)-(\p_x\alpha)^3 \abs{\psi}^2(t,0)\right)\diff t
 \\ \\
 &= {\color{myblue}3s^3\dst \iint_{\tO}\phi^3 \abs{\psi}^2\diff x\diff t }- s^3\dst\int_0^{\t \TM}\phi^3 \abs{\psi}^2(t,0)\diff t,
 \end{array}
 $$

  $$\begin{array}{ll}T_{23}&:=\langle -2s(\p_x\alpha)\p_x \psi, 
  +s(\p_t\alpha)\psi \rangle_{\Ltwo[\tO]}
  \\ \\
  &=s^2\dst\iint_{\tO} \left((\p^2_{xx}\alpha)(\p_t\alpha) + (\p_x\alpha)(\p^2_{tx}\alpha)\right)\abs{\psi}^2 \diff x\diff t + s^2 \int_0^{\t\TM} (\p_x\alpha)(\p_t\alpha)\abs{\psi}^2(t,0)\diff t
  \\ \\
  &\gtrsim -{\color{myred} \inserted{2} s^2\tT\dst\iint_{\tO} \phi^3 \abs{\psi}^2 \diff x\diff t} - s^2 \tT\dst \int_0^{\t\TM}  \phi^3 \abs{\psi}^2(t,0)\diff t,
  \end{array}$$
  
   $$\begin{array}{ll}T_{24}&:=\langle -2s(\p_x\alpha)\p_x \psi, -\f{2}{\eps}s(\p_x\alpha)\psi\rangle_{\Ltwo[\tO]}
   \\ \\
   &=-4s^2\f{1}{\ep} \dst\iint_{\tO} (\p_{xx}^2\alpha)(\p_x\alpha)\abs{\psi}^2 \diff x\diff t
   - 2s^2 \f{1}{\ep}\int_0^{\t\TM} (\p_x\alpha)^2\abs{\psi}^2(t,0)\diff t
   \\ \\
   &={\color{myblue}4s^2\f{1}{\ep}\dst \iint_{\tO} \phi^2 \abs{\psi}^2 \diff x\diff t}
   - 2s^2\f{1}{\ep}\dst\int_0^{\t\TM} \phi^2\abs{\psi}^2(t,0)\diff t
   \end{array}$$ 
  
   $$T_{31}=\langle -\f{2}{\eps} \p_x\psi, 
 -\p_{xx}^2\psi\rangle_{L^2}
 ={\color{mygreen}\f{1}{\ep} \dst\int_0^{\t\TM} \abs{\p_x\psi}^2(t,L)\diff t}{\color{myorange} - \f{1}{\ep} \int_0^{\t\TM} \abs{\p_x\psi}^2(t,0)\diff t}
 $$
 
   $$
   \begin{array}{ll}
   T_{32}&:=\langle -\f{2}{\eps} \p_x\psi, 
 -s^2(\p_x\alpha)^2\psi\rangle_{\Ltwo[\tO]}
 \\ \\
 &=-\f{2 s^2 }{\ep} \dst\iint_{\tO} (\p_x\alpha)(\p_{xx}^2\alpha)\abs{\psi}^2\diff x\diff t + \f{s^2 }{\ep}\int_0^{\t\TM} \left((\p_x\alpha)^2 \abs{\psi}^2(t,L)-(\p_x\alpha)^2 \abs{\psi}^2(t,0)\right)\diff t
 \\ \\
 &={\color{myblue}\f{2s^2 }{\ep} \dst\iint_{\tO} \phi^2\abs{\psi}^2\diff x\diff t }-  \f{s^2 }{\ep}\dst\int_0^{\t\TM} \phi^2 \abs{\psi}^2(t,0)\diff t
 \end{array}$$

   $$\begin{array}{ll}
   T_{33}&:=\langle -\f{2}{\eps} \p_x\psi, 
 +s(\p_t\alpha)\psi\rangle_{L^2}
 \\ \\
 &=\f{s}{\ep}\dst\iint_{\tO} (\p^2_{tx} \alpha)\abs{\psi}^2 \diff x\diff t + \f{s}{\ep} \int_0^{\t\TM} (\p_t\alpha)\abs{\psi}^2(t,0)\diff t
 \\ \\
 &\gtrsim - {\color{myred}\f{s}{\ep} \dst\iint_{\tO} \tT\phi^2 \abs{\psi}^2 \diff x\diff t} -\f{s}{\ep} \dst\int_0^{\t\TM} \tT \phi^2 \abs{\psi}^2(t,0)\diff t,
 \end{array}
 $$
   and finally
   $$\begin{array}{ll}
   T_{34}&:=\langle -\f{2}{\eps} \p_x\psi, 
 -\f{2}{\eps}s(\p_x\alpha)\psi\rangle_{\Ltwo[\tO]}
 \\ \\
 &=-\f{2s}{\ep^2} \dst\iint_{\tO}(\p_{xx}^2\alpha)\abs{\psi}^2 \diff x\diff t  -\f{2s}{\ep^2}\int_0^{\t\TM} (\p_x\alpha) \abs{\psi}^2(t,0)\diff t 
 \\ \\
 &={\color{myblue}\f{2s}{\ep^2} \dst\iint_{\tO}\phi\abs{\psi}^2 \diff x\diff t } -\f{2s}{\ep^2}\dst\int_0^{\t\TM} \phi \abs{\psi}^2(t,0)\diff t .
 \end{array}
 $$
 {\bf Third step: Rearrangement of~\eqref{eq:psiL2}.}
 
From the estimates of $T_{ij},$  we write~\eqref{eq:psiL2} as an inequality between nonnegative terms, namely
\begin{equation}\label{ineq:PJL} \norme{P_1\psi}_{L^2(\tO)}^2 +\norme{P_2\psi}_{L^2(\tO)}^2 +{\color{myblue}I_1(\psi)}+{\color{mygreen}I_2(\psi_x)}
+{\color{black}I_3(\psi)}\lesssim {\color{myred}J_1(\psi)}+{\color{myorange}J_2(\psi_x(\cdot,0)}+L(\psi(\cdot,0)),
\end{equation}
where we define
 \begin{equation}
 \label{def:IJL}
 \begin{array}{ll}
 I_1(\psi)&:={\color{myblue}3s^3 \dst \iint_{\tO}\phi^3 \abs{\psi}^2\diff x\diff t }{\color{myblue}+\f{6s^2 }{\ep} \dst\iint_{\tO} \phi^2 \abs{\psi}^2 \diff x\diff t}
 {\color{myblue}+\f{2s}{\ep^2} \dst\iint_{\tO}\phi\abs{\psi}^2 \diff x\diff t },
 \\ \\
 I_2(\psi_x)&:={\color{mygreen}s\dst\iint_{\tO} \phi \abs{\p_x\psi}^2 \diff x\diff t
 +s\int_0^{\t\TM}  \phi(t,L) \abs{\p_x\psi}^2(t,L)\diff t}
 {\color{mygreen}+\f{1}{\ep} \dst\int_0^{\t\TM} \abs{\p_x\psi}^2(t,L)\diff t} 
,
\\ \\
I_3(\p_t\psi(\cdot,0))&:={\color{black}\f{\ep}{2} \dst\int_0^{\tT}  \abs{\p_t \psi (t,0)}^2 \diff t },
 \\ \\
 J_1(\psi)&:={\color{myred}s^2 \dst\iint_{\tO} \t\TM \phi^3 \abs{\psi}^2 \diff x\diff t}{\color{myred}+ s \dst\iint_{\tO} \t\TM^2\phi^3 \abs{\psi}^2\diff x\diff t}{\color{myred}+\f{\inserted{s}}{\ep} \dst\iint_{\tO} \t\TM \phi^2 \abs{\psi}^2 \diff x\diff t}
 \\ \\
 &{\color{myred}+ \inserted{2} s^2\tT\dst\iint_{\tO} \phi^3 \abs{\psi}^2 \diff x\diff t} {\color{myred}+\f{\inserted{s}}{\ep} \dst\iint_{\tO} \tT\phi^2 \abs{\psi}^2 \diff x\diff t} ,
 \\ \\
 J_2(\psi)&:= {\color{myorange}s\dst\int_0^{\t\TM}\phi (t,0)\abs{\p_x\psi}^2(t,0)\diff t}{\color{myorange}+\f{1}{\ep} \dst\int_0^{\t\TM}  \abs{\p_x\psi}^2(t,0)\diff t},\\ \\
 L(\psi)&:= \inserted{\f{s}{2}} \dst\int_0^{\tT} (\tT \phi^2 +{\ep}\tT^2\phi^3) \abs{\psi(t,0)}^2\diff t  
+ s^3\int_0^{\t\TM}\phi^3 \abs{\psi}^2(t,0)\diff t
\\ \\
&+ s^2 \tT\dst \int_0^{\t\TM}  \phi^3 \abs{\psi}^2(t,0)\diff t
+ s^2\f{2}{\ep}\int_0^{\t\TM} \phi^2\abs{\psi}^2(t,0)\diff t
+  \f{s^2 }{\ep}\int_0^{\t\TM} \phi^2 \abs{\psi}^2(t,0)\diff t
\\ \\
&+\f{\inserted{s}}{\ep}\dst \int_0^{\t\TM} \tT \phi^2 \abs{\psi}^2(t,0)\diff t
 +\f{2s}{\ep^2}\int_0^{\t\TM} \phi \abs{\psi}^2(t,0)\diff t.
 \end{array}
 \end{equation}
The control term is represented by $L$. 

We first gather the terms and simplify them as much as possible:
 $$ 
 \begin{array}{ll}
 I_1(\psi)&={\color{myblue} \dst\iint_{\tO} \abs{\psi}^2 \left(3s^3\phi^3+\f{6s^2}{\ep}  \phi^2 
 +\f{2s}{\ep^2} \phi\right) \diff x\diff t },
 \\ \\
 J_1(\psi)&={\color{myred}\dst\iint_{\tO} \abs{\psi}^2 \left(\inserted{3} s^2 \t\TM \phi^3 + s  \t\TM^2\phi^3 +\f{\inserted{2}s}{\ep} \t\TM \phi^2   \right) \diff x\diff t} ,
 \\ \\
 J_2(\psi)&= {\color{myorange}\dst\int_0^{\t\TM} \abs{\p_x\psi}^2(t,0)\left(s\phi (t,0)+\f{1}{\ep}\right) \diff t},\\ \\
 L(\psi)&=  \int_0^{\tT}  \abs{\psi(t,0)}^2
 \left(\inserted{\f{s}{2}}\tT \phi^2 +s\f{\ep}{2} \tT^2\phi^3
+ s^3\phi^3 + s^2 \tT  \phi^3 
+ s^2\f{\inserted{3}}{\ep}\phi^2
+\f{\inserted{s}}{\ep}  \tT \phi^2 
 +\f{2s}{\ep^2} \phi \right)\diff t.
 \end{array}
 $$
 We can absorb the distributed term $J_1(\psi)$ by the term $I_1(\psi):$ we notice that $\phi\geq \f{4}{\t\TM^2}$ so that $\phi^3\t\TM^2\gtrsim \phi^2,$ \inserted{and we use repeatedly that $\phi^{-1}\lesssim \t T^2.$} It is thus sufficient, to absorb respectively each of the three terms of $J_1,$ that
 $$s\gtrsim \t\TM.$$
 We simplify the control term expression by writing, for $\ep<1,$
  $$\begin{array}{ll}L(\psi)&\lesssim \int_0^{\t\TM} \abs{\psi}^2\phi^3(t,0) \left(\inserted{\f{s}{2}}\tT^3 + {s}\tT^2 + s^3 + s^2\tT + s^2\tT^2\f{\inserted{3}}{\ep} + s\tT^3\f{\inserted{1}}{\ep} + s\tT^4 \f{\inserted{2}}{\ep^2}\right)\diff t
 \\ \\
 &\lesssim s^3 \int_0^{\tT} \abs{\psi}^2 \phi^3(t,0)\diff t
 \end{array}$$
 as soon as we have $$s\gtrsim \tT+\f{\tT^2}{\ep}.$$ 
 Hence the inequality~\eqref{ineq:PJL} is simplified into
\begin{equation}
\label{ineq:PJL2} \norme{P_1\psi}_{L^2(\tO)}^2 +\norme{P_2\psi}_{L^2(\tO)}^2 +{\color{myblue}I_1(\psi)}+{\color{mygreen}I_2(\psi_x)}
+{\color{black}I_3(\psi)}\lesssim {\color{myorange}J_2(\psi_x(\cdot,0)}+s^3\int_0^{\tT} \abs{\psi}^2 \phi^3(t,0)\diff t.
\end{equation}

\

{\bf Fourth step: absorption of the term $J_2$.}

The boundary condition~\eqref{eq:psi:bound} at $x=0$ gives
$$\p_x\psi(t,0)=\f{\ep}{2}\biggl\{\p_t\psi(t,0) +s\left(\p_t \alpha(t,0)\right)\psi(t,0)\large\biggr\}-s(\p_x\alpha(t,0))\psi(t,0)$$
hence, recalling that $\partial_x \alpha=\phi,$ 
$$\begin{array}{ll}\abs{\p_x\psi(t,0)}^2&=\biggl(\f{\ep}{2}\p_t\psi(t,0) +s\bigl(\f{\ep}{2}\p_t \alpha(t,0)-\phi(t,0)\bigr)\psi(t,0)\biggr)^2
\\ \\&=\f{\ep^2}{4}\abs{\p_t\psi}^2 +s\left(\f{\ep}{2}\p_t \alpha-\phi\right) \p_t\left(\inserted{\f{\ep}{2}}\abs{\psi}^2\right)+s^2 \left(\f{\ep}{2}\p_t \alpha-\phi\right)^2\abs{\psi}^2
\end{array}$$
We thus compute $J_2$ as
\begin{equation*}\begin{array}{ll}
 J_2(\psi)&= {\color{myorange}\dst\int_0^{\t\TM} \abs{\p_x\psi}^2(t,0)\left(s\phi (t,0)+\f{1}{\ep}\right) \diff t}\\ \\
&=\dst\int_0^{\t\TM} \left(s\phi +\f{1}{\ep}\right)
\biggl\{ 
\f{\ep^2}{4}\abs{\p_t\psi}^2 +s\left(\f{\ep}{2}\p_t \alpha-\phi\right) \p_t\left(\inserted{\f{\ep}{2}}\abs{\psi}^2\right) 
\\ \\  &\qquad \qquad\qquad\qquad +s^2 \left(\f{\ep}{2}\p_t \alpha-\phi\right)^2\abs{\psi}^2
\biggr\} (t,0)\diff t
\\ \\
&=\dst\int_0^{\t\TM} \left((s\phi\f{\ep^2}{4} +\f{\ep}{4})\abs{\p_t\psi}^2 \right.\\ \\
&\left.
+ \left\{-\p_t\bigl\{\left(\inserted{\f{\ep}{2}}s^2\phi +\f{s}{\inserted{2}}\right) 
 \left(\f{\ep}{2}\p_t \alpha-\phi\right)\bigr\} +s^2 \inserted{ \left(s\phi +\f{1}{\ep}\right) }\left(\f{\ep}{2}\p_t \alpha-\phi\right)^2\right\}\abs{\psi}^2
\right) (t,0)\diff t.
\end{array}\end{equation*}
\inserted{We now use Lemma~\ref{lem:alphax} to obtain}
\begin{equation}\label{ineq:J2}
\begin{array}{ll}
 J_2(\psi) &\lesssim \dst\int_0^{\t\TM}  (s\phi \ep^2 +{\ep})\abs{\p_t\psi}^2(t,0) \diff t 
\\ \\ &+\inserted{\dst\int_0^{\t\TM}} \left(s^2\tT\inserted{^2}\phi^4 {\ep}\inserted{^2} +\inserted{\ep s^2\phi^3 \t T +\ep s \t T^2 \phi^3+s\t T\phi^2}\right)\abs{\psi}^2(t,0)\diff t
\\ \\ & + \dst\int_0^{\t\TM} (\inserted{s^3 \ep^2 \t T^2 \phi^5 + \ep s^3 \phi^4 \t T + s^3\phi^3
+\ep s^2\t T^2\phi^4 
+ s^2 \phi^3 \t T 
+ \f{s^2}{\ep} \phi^2
})  \abs{\psi}^2(t,0)\diff t
\end{array}\end{equation}
We keep the first term for later, and to concatenate the second \inserted{and third} term with $L$ we choose a higher power for the weight function, namely $\inserted{ \ep ^2 s^5\phi^5}$ instead of $s^3\phi^3.$ \inserted{For this we use repeatedly $\phi^{-1}\lesssim \t T^2 \lesssim \ep s,$ and assume $\ep \lesssim s$, in order to check that each term in the second and third lines of~\eqref{ineq:J2} is dominated by $\ep^2 s^5\phi^5.$ We thus obtain:
\[L+J_2 \lesssim \dst\int_0^{\t\TM}  (s\phi \ep^2 +{\ep})\abs{\p_t\psi}^2(t,0) \diff t +\ep^2 s^5 \int_0^{\t T} \abs{\psi}^2 \phi^5(t,0)\diff t .
\]
}

\

{\bf Fifth step: Estimates on $\Vert \p_x \psi\Vert^2$ and $\Vert \p^2_{xx} \psi\Vert^2$.}

$$P_1=\p_t\psi -2s\phi \p_x \psi-\f{2}{\eps} \p_x\psi,
$$
hence 
$$s^{-1}\phi^{-1}\abs{\p_t\psi}^2 \lesssim s\phi \abs{\p_x\psi}^2 + s^{-1} \phi^{-1}\ep^{-2} \abs{\p_x\psi}^2 + s^{-1}\phi^{-1} \abs{P_1}^2,$$
hence \inserted{since $\ep <1$ and $s^{-1}\phi^{-1}\lesssim \ep,$} we have
$$\dst\iint_{\tO} s^{-1}\phi^{-1}\abs{\p_t\psi}^2 \diff x\diff t \lesssim {\color{mygreen}s}\iint_{\tO} \phi \abs{\p_x\psi}^2 \diff x\diff t+ \iint_{\tO} \abs{P_1}^2\diff x\diff t. 
$$
Both terms of the right-hand side appear in the left-hand side of~\eqref{ineq:PJL2} (the first one through $I_2$).

Recalling now the definition~\eqref{def:P123} of $P_2:$
$$P_2=-\p_{xx}^2\psi-s^2\phi^2\psi+s(\p_t\alpha)\psi-\f{2}{\eps}s\phi\psi,$$
we have, since $\abs{\p_t\alpha}\lesssim \tT\phi^2$,
$$\begin{array}{ll}s^{-1}\phi^{-1}\abs{\p^2_{xx}\psi}^2  &\lesssim
 s^3\phi^3\abs{\psi}^2 
+ s \tT^2 \phi^3 {\color{mygreen}\abs{\psi}^2} + \f{\inserted{1}}{\ep^2} {\color{mygreen}s} \phi \abs{\psi}^2+s^{-1}\phi^{-1} \abs{P_2}^2 
\\ \\ &\lesssim  s^3\phi^3\abs{\psi}^2 + \abs{P_2}^2.
\end{array}
$$
The first term on the right-hand side appears with the same order of magnitude in $I_1.$ 
Combined with~\eqref{ineq:PJL} and with~\eqref{ineq:J2} we obtain, for $s\gtrsim \tT+\f{\tT^2}{\ep}:$

\begin{equation}\label{ineq:21CG}
\begin{array}{ll}
\dst\iint_{\tO} \left\{s^{-1}\phi^{-1}\left(\abs{\p_t\psi}^2
+\abs{\p^2_{xx}\psi}^2 \right) +  \left(s^3\phi^3\abs{\psi}^2 +s\phi \abs{\p_x\psi}^2\right)\right\}\diff x\diff t + s\int_0^{\t\TM}\phi\abs{\p_x\psi}^2(t,L)\diff t
\\ \\
\lesssim \dst\int_0^{\t\TM} \inserted{\ep^2 s^5\phi^5} \abs{\psi}^2(t,0) \diff t 
+ s\int_0^{\t\TM} \phi {\ep^2} \abs{\p_t\psi}^2(t,0)\diff t.
\end{array}
\end{equation}
\

{\bf Sixth step: back to the dimensionless solution $\tilde \sol$.}

We recall that by~\eqref{def:psi} we have $\psi=\t\sol e^{-s\alpha},$ so
$\p_x\psi = e^{-s\alpha} (\p_x \t\sol -s\phi\t \sol)$, so that
$$\dst\iint_{\tO}s\phi e^{-2s\alpha} \abs{\p_x\t\sol}^2\diff x\diff t \lesssim \dst\iint_{\tO} s\phi \abs{\p_x\psi}^2 \diff x\diff t + \iint_{\tO} s^3\phi^3 \abs{\psi}^2\diff x\diff t$$
and for the boundary $x=L$ since $\t\sol (t,L)=\psi(t,L)=0$ we have
$$\int_0^{\tT} s\phi e^{-2s\alpha} \abs{\p_x\t\sol}^2(t,L) \diff t \lesssim \int_0^{\tT} s\phi \abs{\p_x\psi}^2 (t,L)\diff t.$$
We also compute $\p_t \psi=e^{-s\alpha}\left(\p_t \t\sol - s(\p_t\alpha )\t\sol\right)$ so that, for $s\geq \tT,$ using Lemma~\ref{lem:alphax}
$$\begin{array}{ll}\dst\iint_{\tO}s^{-1}\phi^{-1} e^{-2s\alpha} \abs{\p_t\t\sol}^2\diff x\diff t
& \lesssim
\dst\iint_{\tO}s^{-1}\phi^{-1} \abs{\p_t\psi}^2 \diff x\diff t
+ \iint_{\tO}s \phi^3 \tT^2 \abs{\psi}^2 \diff x\diff t
\\ \\
&\lesssim \dst\iint_{\tO}s^{-1}\phi^{-1} \abs{\p_t\psi}^2 \diff x\diff t
+ \iint_{\tO}s^3 \phi^3 \abs{\psi}^2 \diff x\diff t,
\end{array} $$
and for the boundary $x=0$ we have for $s\geq \t T^2:$ 
$$ \begin{array}{ll}\dst\int_0^{\tT} s\phi  \abs{\p_t\psi}^2 (t,0)\diff t
& \lesssim
\dst\int_0^{\tT} s \phi e^{-2s\alpha}\abs{\p_t\t\sol}^2 (t,0)\diff t
+ \int_0^{\tT} s^3 \phi^5 \tT^2 e^{-2s\alpha}\abs{\t\sol}^2 (t,0)\diff t
\\ \\
&\lesssim \dst\int_0^{\tT} s \phi  e^{-2s\alpha}\abs{\p_t\t\sol}^2 (t,0)\diff t
+ \int_0^{\tT} s^{\inserted{5}} \phi^5 e^{-2s\alpha}\abs{\inserted{\t\sol}}^2 (t,0)\diff t  ,
\end{array}
$$
Similarly we have 
$\p^2_{xx} \psi=e^{-s\alpha} (\p^2_{xx} \t\sol + s\phi\t\sol - 2 s\phi\p_x\t \sol  +s^2\phi^2\t\sol),$ so that for ${\color{mygreen}s\geq \tT}$
$$\dst\iint_{\tO}s^{-1}\phi^{-1} e^{-2s\alpha} \abs{\p_{xx}^2 \t\sol}^2\diff x\diff t 
\lesssim
\dst\iint_{\tO}\left(s^{-1}\phi^{-1} \abs{\p_{xx}^2 \psi}^2\diff x\diff t
+  s^3 \phi^3\abs{\psi}^2 + {\color{mygreen}s} \phi \abs{\p_x\psi}^2 \right)\diff x\diff t
$$
We combine all these inequalities with~\eqref{ineq:21CG} and obtain
\begin{multline*}
\dst\iint_{\tO} \left\{s^{-1}\phi^{-1}e^{-2s\alpha}\left(\abs{\p_t\t\sol}^2
+\abs{\p^2_{xx}\t\sol}^2 \right) \inserted{+} \left(s^3\phi^3 e^{-2s\alpha} \abs{\t\sol}^2 +s\phi e^{-2s\alpha} \abs{\p_x\t\sol}^2\right)\right\}\diff x\diff t 
\\ 
+ s\dst\int_0^{\t\TM}\phi e^{-2s\alpha}\abs{\p_x\t\sol}^2(t,L)\diff t 
\lesssim 
 s{\ep^2}\int_0^{\t\TM} \phi  e^{-2s\alpha}\abs{\p_t\t\sol}^2(t,0)\diff t
\\ +{\ep^2}\dst\int_0^{\tT} s^{\inserted{5}} \phi^5  e^{-2s\alpha} \abs{\t\sol}^2(t,0)\diff t.
\end{multline*}
This is~\eqref{ineq:14CornilleauGuerrero}, and ends the proof of Proposition~\ref{prop:12fromCornilleauGuerrero}.

\end{proof}

We now go a step further by estimating the term 
$$s\ep^2\int_0^{\t\TM} \phi e^{-2s\alpha} \vert \p_t\t\sol\vert^2(t,0) \diff t$$
on the right-hand side of~\eqref{ineq:14CornilleauGuerrero}. This leads to the following proposition.

\begin{proposition}\label{prop:10fromCornilleauGuerrero}
There exists a constant $C>0$ such that, for every $\ep\in (0,1)$ and $s\geq C( \inserted{\ep +}\tT+ \tT^2/\ep),$ \inserted{with $\tT \lesssim 1,$} we have the following inequality

\begin{multline}\label{ineq:13CornilleauGuerrero}
\dst\iint_{\tO} \left\{s^{-1}\phi^{-1}e^{-2s\alpha}\left(\abs{\p_t\t\sol}^2
+\abs{\p^2_{xx}\t\sol}^2 \right) +  \left(s^3\phi^3 e^{-2s\alpha} \abs{\t\sol}^2 +s\phi e^{-2s\alpha} \abs{\p_x\t\sol}^2\right)\right\}\diff x\diff t 
\\ 
\qquad \qquad + s\int_0^{\t\TM}\phi e^{-2s\alpha}\abs{\p_x\t\sol}^2(t,L)\diff t
\lesssim
\int_0^{\t\TM} \left( \inserted{\ep^2} s^{\inserted{5}} \phi^5 + \inserted{\ep^2 s^6 \phi^7+}\ep^3s^{\inserted{3}}  \phi^\inserted{3} 
\right)e^{-2s\alpha}\abs{\t\sol}^2(t,0) \diff t .
\end{multline}
\end{proposition}

\begin{proof}
Let us first integrate by parts $s\ep^2\int_0^{\t\TM} \phi e^{-2s\alpha} \vert \p_t\t\sol\vert^2(t,0) \diff t$ to get 
\begin{multline}\label{eq:23CG}
s{\ep^2} \int_0^{\t\TM} \phi  e^{-2s\alpha}\abs{\p_t\t\sol}^2(t,0)\diff t
\inserted{=-s{\ep^2}\int_0^{\t\TM} \t\sol \biggl\{ \partial_t \t\sol \p_t \left(\phi e^{-2s\alpha}\right)+\phi e^{-2s\alpha}\p_{tt}^2 \t\sol \biggr\}\diff t}
\\
=s{\ep^2}\int_0^{\t\TM} \left\{\f{1}{2}\p^2_{tt}(\phi e^{-2s\alpha}) \abs{\t\sol}^2 - \phi e^{-2s\alpha} \t\sol \p^2_{tt}(\t\sol) \right\} (t,0)\diff t
\\ 
 \lesssim s{\ep^2} \int_0^{\t\TM} \left\{s^2\tT^2\phi^5 e^{-2s\alpha} \abs{\t\sol}^2 + \phi e^{-2s\alpha} \abs{\t\sol}\abs{ \p^2_{tt}(\t\sol)} \right\} (t,0)\diff t.
\end{multline}
The first term on the right-hand side goes with the former control terms of $L(\psi);$ we now need to estimate $ \p^2_{tt}\t\sol(t,0).$ To do so, we prove the following lemma.

\begin{lemma}\label{lem:remark4}
Let us define the intermediate function 
$$\zeta(t,x):=\theta(t)\p_t\t\sol(t,x):=e^{-s\alpha(t,{\color{black}L})}\phi^{-\f{5}{2}}(t,0) \p_t\t\sol (t,x),$$
we have the estimate
\begin{equation}
\label{ineq:rmk4}
\ep \dst\int_0^{\t\TM} \abs{\p_t \zeta}^2(t,0) \diff t \lesssim {\color{myorange}\f{\t T^2 +1}{\ep^2}  }\iint_{\tO} \abs{\theta' }^2\abs{\p_t \t\sol}^2 \diff x\diff t 
+
{\color{myorange}\ep (1+ \tT)} \int_0^{\t\TM} \abs{\theta' }^2\abs{\p_t \t\sol}^2 (t,0)\diff t
\end{equation}
\end{lemma}

\begin{proof}
\inserted{From~\eqref{eq:model:redim}, w}e compute the system satisfied by $\zeta$ and find
\begin{equation}
\label{eq:zeta}
\left\{ 
\begin{array}{lll}
\p_t \zeta -\f{2}{\ep}\p_x \zeta - \p^2_{xx}\zeta = \theta' \p_t \t\sol &&{\text{in }}\tO,
\\ \\
\left(\p_t\zeta-\f{2}{\ep}\p_x\zeta\right)(t,0) =\theta' \p_t\t\sol(t,0), &\zeta(t,L)=0 & {\text{for }}t\in(0,\tT),
\\ \\
\zeta(0,x)=0&& {\text{for }}x\in(0,\LM).
\end{array}\right.
\end{equation}
 Let us first  multiply~\eqref{eq:zeta} by $\zeta$ and integrate \inserted{in size}:
\begin{multline}\label{ineq:zeta3}
\f{1}{2} \f{d}{\diff t} \Vert \zeta(t,\cdot) \Vert_{L^2(0,L)}^2  + \f{1}{\ep}\zeta(t,0)^2 + \int_0^L \abs{\p_x \zeta}^2\diff x +\f{\ep}{4}\f{d}{\diff t}\abs{\zeta(t,0)}^2
\\ 
=\f{\ep}{2} \zeta(t,0) \theta'(t)\p_t\t\sol(t,0)
+ \int_0^L \zeta \theta'(t)\p_t\t\sol \diff x
\\ 
\leq \f{1}{2\ep} \zeta(t,0)^2  + \f{\ep^3}{8} \abs{ \theta' \p_t\t\sol}^2(t,0) + \int_0^L \zeta\theta'\p_t\t\sol \diff x.
\end{multline}
We integrate~\eqref{ineq:zeta3} in time and apply the Gronwall lemma~\ref{lemma:Gronwall} to the inequality
\[\f{1}{2}  \Vert \zeta(t,\cdot) \Vert_{L^2(0,L)}^2   +\f{\ep}{4}\abs{\zeta(t,0)}^2
\leq \f{\ep^3}{8} \int_0^{\tT} \abs{\theta' \p_t\t\sol}^2(\sigma,0)\diff \sigma + \int_0^t\int_0^L \zeta\theta'\p_t\t\sol \diff x\diff \sigma
\]
by setting, applying Cauchy-Schwarz' inequality to the last term,
\begin{align*}f(t)&:=\sqrt{\int_0^L \abs{\zeta(t,x)}^2\diff x + \f{\ep}{2}\abs{\zeta(t,0)}^2},\\ f_0&:=\sqrt{\f{\ep^3}{4}\int_0^{\tT} \abs{\theta'\p_t\t\sol}^2 (\sigma,0)\diff \sigma},\\ 
g(\sigma)&=\sqrt{\int_0^L\abs{\theta'\p_t\t\sol}^2(\sigma,x)\diff x}.
\end{align*}
We deduce
\begin{multline}
\label{ineq:zeta}\Vert \zeta(t,\cdot) \Vert_{L^2(0,L)}^2 +\f{\ep}{2}\abs{\zeta(t,0)}^2 \leq  2\abs{\inserted{f}_0}^2  +2\left(\int_0^{\tT} g(\sigma)\diff \sigma\right)^2
\\ 
\leq  \f{\ep^3}{2} \dst\int_0^{\tT} \abs{ \theta' \p_t\t\sol}^2(t,0)\diff t  + 2\tT \iint_{\tO} \abs{\theta'\p_t\t\sol}^2 \diff x\diff t.
\end{multline}
We use this inequality to bound the term in $\abs{\p_x\zeta}$ in~\eqref{ineq:zeta3}:
\begin{multline}\label{ineq:dxzeta}
\dst\iint_{\tO} \abs{\p_x\zeta}^2 \diff x \diff t
\leq  \int_0^{\tT} \left\{
 \f{\ep^3}{8} \abs{ \theta' \p_t\t\sol}^2(t,0) + \int_0^L \zeta\theta'\p_t\t\sol \diff x\right\}\diff t
\\ 
\leq \f{\ep^3}{8} \int_0^{\tT} \abs{ \theta' \p_t\t\sol}^2(t,0)\diff t 
+ \f{1}{2}\int_0^{\tT} \int_0^L \left( \zeta^2(t,x) +\abs{\theta'\p_t\t\sol}^2(t,x)\right) \diff x\diff t
\\ 
\leq \f{\ep^3}{8} (1+\inserted{2}\tT) \int_0^{\tT} \abs{ \theta' \p_t\t\sol}^2(t,0)\diff t 
+ (\tT^2 +\inserted{\f{1}{2}})\int_0^{\tT} \int_0^L \abs{\theta'\p_t\t\sol}^2(t,x) \diff x\diff t
\\ 
\lesssim {\ep^3}(1+\tT) \int_0^{\tT} \abs{ \theta' \p_t\t\sol}^2(t,0)\diff t 
+(1+ \tT^2 )\dst\iint_{\tO} \abs{\theta'\p_t\t\sol}^2(t,x) \diff x\diff t.
\end{multline}
Integrating by parts and using the boundary condition, we notice that
\begin{align*}
-\dst\int_0^L \p_t \zeta \p^2_{xx}\zeta (t,x)\diff x&= \int_0^{\LM} (\p_{xt}^2 \zeta )(\p_x \zeta)(t,x)\diff x +
 (\p_t\zeta)(\p_x\zeta)(t,0) \\ 
 &=
 \f{\diff }{\diff t}\f{1}{2}\int_0^{\LM} \abs{\p_{x} \zeta}^2(t,x)\diff x 
+ \f{\ep}{2} \abs{\p_t\zeta}^2(t,0) - \f{\ep}{2} \theta'
\p_t\t\sol \p_t \zeta (t,0),
\end{align*}
hence multiplying~\eqref{eq:zeta} by $\p_t \zeta$ and integrating in $x$ gives 
\begin{multline}\label{ineq:zeta1}\Vert \p_t \zeta (t,\cdot) \Vert^2_{L^2(0,L)}
+ \f{\diff}{\diff t}\f{1}{2}\int_0^{\LM} \abs{\p_{x} \zeta}^2(t,x)\diff x 
+ \f{\ep}{2} \abs{\p_t\zeta}^2(t,0) \\ 
\leq  \int_0^{\LM} \left(\f{2}{\ep}\abs{\p_x \zeta}\abs{\p_t\zeta}(t,x) + \abs{\theta' \p_t\t\sol}\abs{\inserted{\p_t} \zeta}(t,x)\right)\diff x + \f{\ep}{2} \theta'
\p_t\t\sol \p_t \zeta (t,0).
\end{multline}
We integrate in time between $0$ and $t$ and find
\begin{multline*}
\int_0^t \Vert \p_t \zeta (\sigma,\cdot) \Vert^2_{L^2(0,L)}\diff \sigma
+\f{1}{2} \int_0^{\LM} \abs{\p_{x} \zeta}^2(t,x)\diff x 
+ \int_0^t \f{\ep}{2} \abs{\p_t\zeta}^2(\sigma,0)\diff \sigma \\
\leq \int_0^t \left\{\int_0^{\LM} \left(\f{2}{\ep}\abs{\p_x \zeta}\abs{\p_t\zeta}(\sigma,x) + \abs{\theta' \p_t\t\sol}\abs{\inserted{\p_t} \zeta}(\sigma,x)\right)\diff x + \f{\ep}{2} \theta'
\p_t\t\sol \p_t \zeta (\sigma,0)\right\}\diff \sigma 
\\ 
\leq
 \dst\iint_{\tO} \left(\f{\inserted{4}}{\ep^2}\abs{\p_x \zeta}^2 + \f{1}{\inserted{4}}\abs{\p_t\zeta}^2 + \abs{\theta' \p_t\t\sol}^2 +\f{1}{\inserted{4}}\abs{\inserted{\p_t} \zeta}^2\right)\diff x\diff \sigma 
 \\ + \f{\ep}{4}\int_0^t \left(\abs{\theta'
\p_t\t\sol}^2 +\abs{ \p_t \zeta}^2\right)(\sigma,0)\diff \sigma.
\end{multline*}
Using~\eqref{ineq:dxzeta} and~\eqref{ineq:zeta} we deduce
\begin{multline*}
\inserted{\f{1}{2}}\int_0^{\tT} \Vert \p_t \zeta (\sigma,\cdot) \Vert^2_{L^2(0,L)}\diff \sigma
+ \int_0^{\tT} \f{\ep}{\inserted{4}} \abs{\p_t\zeta}^2(\sigma,0)\diff \sigma
 \\
\leq 
 \dst\iint_{\tO} \left(\f{\inserted{4}}{\ep^2}\abs{\p_x \zeta}^2  + \abs{\theta' \p_t\t\sol}^2 \right)\diff x\diff \sigma + \f{\ep}{4}\int_0^t \abs{\theta'
\p_t\t\sol}^2 (\sigma,0)\diff \sigma
\\ 
\lesssim
\ep(1+\tT) \dst\int_0^{\tT} \abs{ \theta' \p_t\t\sol}^2(t,0)\diff t  + \f{\t T^2 +1}{\ep^2} \iint_{\tO} \abs{\theta'\p_t\t\sol}^2 \diff x\diff t.
\end{multline*}
\end{proof}

Since we have $\p_t\zeta=\theta'\p_t\t\sol + \theta \p^2_{tt}\t\sol,$~\eqref{ineq:rmk4} implies by Young's inequality
\begin{multline}\label{ineq:theta1} \ep\dst\int_0^{\tT} \abs{\theta}^2 \abs{\p^2_{tt}\t\sol}^2(t,0) \diff t 
\\ \lesssim 
\left(\f{\color{myorange} \t T^2+1}{\ep^2} \iint_{\tO} \abs{\theta'\p_t\t\sol}^2 \diff x\diff t + {\color{myorange}{\ep} (1+\t T)}\int_0^{\tT} \abs{ \theta' \p_t\t\sol}^2(t,0)\diff t  \right).
\end{multline}
We now integrate by parts the last integral and use again Young's inequality with appropriate weights
\begin{multline*} {\color{myorange}{\ep} (1+\t T)}\int_0^{\tT} \abs{ \theta' \p_t\t\sol}^2(t,0)\diff t 
=-  {\color{myorange}{\ep} (1+\t T)}\int_0^{\tT} \t\sol \p_t\left(\p_t\t\sol \abs{\theta'}^2\right)(t,0)\diff t
\\ 
= -  {\color{myorange}{\ep} (1+\t T)}\int_0^{\tT} \left\{\t\sol \p_t \t\sol \p_t(\abs{\theta'}^2) \inserted{+} \t\sol \p_{tt}^2 \t\sol \abs{\theta'}^2\right\}(t,0)\diff t
\\ 
\leq \int_0^{\tT}  {\color{myorange}{\ep} (1+\t T)} \f{1}{2}\abs{\t\sol}^2  \p^2_{tt}(\abs{\theta'}^2) (t,0)\diff t
+{\color{red}\f{c}{2}}  {\color{myorange}{\ep} (1+\t T)} \int_0^{\tT} \abs{\p_{tt}^2 \t\sol}^2 \theta^2 (t,0)\diff t
\\ + {\color{red}\f{1}{2c}} {\color{myorange}{\ep} (1+\t T)}\int_0^{\tT} \abs{\theta'}^4 \theta^{-2} \abs{\t\sol}^2 (t,0)\diff t.
\end{multline*}
Choosing $c=\f{\inserted{1}}{(1+\t \TM)}$ and injecting this in~\eqref{ineq:theta1} gives 
\begin{multline}\label{ineq:inter1}
\f{\ep}{2}\int_0^{\t\TM} \abs{\theta}^2 \abs{\p^2_{tt}\t\sol}^2(t,0) \diff t 
\lesssim 
\biggl(\f{\color{myorange} \t \TM^2+1}{\ep^2}\dst\iint_{\tO} \abs{\theta'\p_t\t\sol}^2 \diff x\diff t 
 \\+ {\color{myorange}{\ep} (1+\t \TM)} \int_0^{\tT} \abs{\t\sol}^2  \p^2_{tt}(\abs{\theta'}^2) (t,0)\diff t
+{\ep} {\color{myorange} (1+\t \TM)^2} \int_0^{\t\TM} \abs{\theta'}^4 \theta^{-2} \abs{\t\sol}^2 (t,0)\diff t\biggr).
\end{multline}
Taking the definition of $\theta$ and Lemma~\ref{lem:alphax}, \inserted{and using that $\phi(t,L)\leq \phi(t,x)\leq \phi(t,0)$ and $1\lesssim \ep s\phi\leq s \phi,$} we have
$$\vert \p_t \theta \vert \lesssim \left(s \t \TM\phi^{-\f{1}{2}} (t,0) +\t \TM \phi^{-\f{3}{2}} (t,0)\right)e^{-s\alpha (t,{\color{black}L})} \lesssim s\t\TM \phi^{-\f{1}{2}}(t,x)e^{-s\alpha (t,x)},$$
and similarly we compute
$$ \vert \p^2_{tt} \theta \vert \lesssim s^2\t\TM^2\phi^{\f{3}{2}}\inserted{(t,0)} e^{-s\alpha\inserted{(t,L)}},\qquad
 \vert \p^3_{ttt} \theta \vert \lesssim s^3\t\TM^3\phi^{\f{7}{2}}\inserted{(t,0)} e^{-s\alpha\inserted{(t,L)}},
$$
so that
$$\vert \p^2_{tt} (\abs{ \theta'}^2) \lesssim s^4\t\TM^4\phi^3 \inserted{(t,0)}e^{-2s\alpha\inserted{(t,L)}},
$$
and multiplying~\eqref{ineq:inter1} by $s^{-3}\f{\ep^2}{\t\TM^2(1+\t\TM^2)}$ to make appear the first term of the left-hand side of~\eqref{ineq:14CornilleauGuerrero}, 
we obtain~
\begin{multline*}{\ep^3} \f{s^{-3}}{\t\TM^2(1+\t\TM^2)} \int_0^{\tT} \abs{\theta}^2 \abs{\p^2_{tt}\t\sol}^2(t,0) \diff t 
\lesssim 
\biggl(s^{-1} \dst\iint_{\tO} \phi^{-1}e^{-2s\alpha} \abs{\p_t\t\sol}^2 \diff x\diff t 
\\+ {\color{myorange}\f{\ep^3}{\t\TM^2(1+\t \TM)}} s^{-3}\int_0^{\tT} \abs{\t\sol}^2  \p^2_{tt}(\abs{\theta'}^2) (t,0)\diff t
+\f{{\ep}^3}{\t \TM^2} s^{-3}\int_0^{\tT} \abs{\theta'}^4 \theta^{-2} \abs{\t\sol}^2 (t,0)\diff t\biggr),
\end{multline*}
{\it i.e.}
\begin{multline}{\ep^3} \f{s^{-3}}{\t\TM^2(1+\t\TM^2)} \int_0^{\tT} e^{-2s\alpha}\phi^{-5} \abs{\p^2_{tt}\t\sol}^2(t,0) \diff t 
\lesssim 
\biggl(s^{-1} \dst \iint_{\tO} \phi^{-1}e^{-2s\alpha} \abs{\p_t\t\sol}^2 \diff x\diff t 
\\ 
+ {\color{myorange}\f{\ep^3}{(1+\t \TM)}} s \t\TM^2 \int_0^{\tT} \phi^3\inserted{(t,0)} e^{-2s\alpha\inserted{(t,L)}} \abs{\t\sol}^2   (t,0)\diff t
+{\ep}^3s{\t \TM^2}  \int_0^{\tT} \phi^3 \inserted{(t,0)}\abs{\t\sol}^2 (t,0)e^{-2s\alpha\inserted{(t,L)}}\diff t\biggr)
\\
{\inserted {\lesssim 
\biggl(s^{-1} \dst \iint_{\tO} \phi^{-1}e^{-2s\alpha} \abs{\p_t\t\sol}^2 \diff x\diff t 
+{\ep}^3s{\t \TM^2}  \int_0^{\tT} \phi^3 \inserted{(t,0)}\abs{\t\sol}^2 (t,0)e^{-2s\alpha\inserted{(t,L)}}\diff t\biggr)
}}.\label{ineq:interm2}
\end{multline}
We use Young's inequality to estimate the last term of~\eqref{eq:23CG} as follows
\begin{multline}\label{eq:23CG:Y}
s\ep^2\int_0^{\t\TM} \phi e^{-2s\alpha}\vert u\vert \vert \p^2_{tt} u \vert \diff t
\leq
\f{s^5\ep}{2} \t\TM^2(1+\t\TM^2)\int_0^{\t\TM} \phi^7e^{-2s\alpha}\vert u\vert^2 (t,0) \diff t
  \\ +  \f{\eps^{3}s^{-3}}{2\t\TM^2(1+\t\TM^2)}\int_0^{\t\TM}e^{-2s\alpha}\phi^{-5} \vert \p^2_{tt} u \vert^2 (t,0) \diff t
\end{multline}
and we now have all the ingredients to improve~\eqref{ineq:14CornilleauGuerrero} by getting rid of the time-derivative term of the right-hand side and incorporating the previous inequalities in~\eqref{ineq:14CornilleauGuerrero}, first~\eqref{eq:23CG}, then~\eqref{eq:23CG:Y}, and finally~\eqref{ineq:interm2}:
\begin{multline*}
\dst\iint_{\tO} \left\{s^{-1}\phi^{-1}e^{-2s\alpha}\left(\abs{\p_t\t\sol}^2
+\abs{\p^2_{xx}\t\sol}^2 \right) +  \left(s^3\phi^3 e^{-2s\alpha} \abs{\t\sol}^2 +s\phi e^{-2s\alpha} \abs{\p_x\t\sol}^2\right)\right\}\diff x\diff t 
\\ 
+ s\int_0^{\t\TM}\phi e^{-2s\alpha}\abs{\p_x\t\sol}^2(t,L)\diff t
\\ \\
\lesssim \inserted{\ep^2 s^5}\int_0^{\t\TM} \phi^5 e^{-2s\alpha}\abs{\t\sol}^2(t,0) \diff t 
+ s{\ep^2}\int_0^{\t\TM} \phi  e^{-2s\alpha}\abs{\p_t\t\sol}^2(t,0)\diff t
\\
\lesssim \inserted{\ep^2 s^5}\int_0^{\t\TM} \phi^5 e^{-2s\alpha}\abs{\t\sol}^2(t,0) \diff t 
+ s{\ep^2} \int_0^{\t\TM} \left\{s^2\tT^2\phi^5 e^{-2s\alpha} \abs{\t\sol}^2 + \phi e^{-2s\alpha} \abs{\t\sol}\abs{ \p^2_{tt}(\t\sol)} \right\} (t,0)\diff t
\\ 
\lesssim (\inserted{\ep^2 s^5}+s^3\ep^2\t\TM^2)\int_0^{\t\TM} \phi^5 e^{-2s\alpha}\abs{\t\sol}^2(t,0) \diff t 
+ s\ep^2\int_0^{\t\TM} \phi e^{-2s\alpha} \abs{\t\sol}\abs{ \p^2_{tt}(\t\sol)}  (t,0)\diff t
\\ 
\lesssim (\inserted{\ep^2 s^5}+s^3\ep^2\t\TM^2)\int_0^{\t\TM} \phi^5 e^{-2s\alpha}\abs{\t\sol}^2(t,0) \diff t + \f{s^5\ep}{2} \t\TM^2(1+\t\TM^2)\int_0^{\t\TM} \phi^7e^{-2s\alpha}\vert u\vert^2 (t,0) \diff t
\\  + \f{\eps^{3}s^{-3}}{2\t\TM^2(1+\t\TM^2)}\int_0^{\t\TM}e^{-2s\alpha}\phi^{-5} \vert \p^2_{tt} u \vert^2 (t,0) \diff t 
\\ 
\lesssim \int_0^{\t\TM} \left( (\inserted{\ep^2s^5}+s^3\ep^2\t\TM^2)\phi^5 + \f{s^5\ep}{2} \t\TM^2(1+\t\TM^2) \phi^7 \right)e^{-2s\alpha}\abs{\t\sol}^2(t,0) \diff t 
\\ + 
\biggl(s^{-1} \dst\iint_{\tO} \phi^{-1}e^{-2s\alpha} \abs{\p_t\t\sol}^2 \diff x\diff t 
\inserted{+{\ep}^3 s {\t \TM^2}  \int_0^{\tT} \phi^3  \abs{\t\sol}^2 (t,0)e^{-2s\alpha(t,L)}\diff t\biggr)}.
\end{multline*}
Hence finally, \inserted{due to the fact that the constant on the right-hand side may be tuned as to be smaller than on the left-hand side for the term with $\abs{\p_t\t u}^2:$}
\begin{multline*}
\dst\iint_{\tO} \left\{s^{-1}\phi^{-1}e^{-2s\alpha}\left(\abs{\p_t\t\sol}^2
+\abs{\p^2_{xx}\t\sol}^2 \right) +  \left(s^3\phi^3 e^{-2s\alpha} \abs{\t\sol}^2 +s\phi e^{-2s\alpha} \abs{\p_x\t\sol}^2\right)\right\}\diff x\diff t 
\\ 
+ s\int_0^{\t\TM}\phi e^{-2s\alpha}\abs{\p_x\t\sol}^2(t,L)\diff t
\\ 
\lesssim
\int_0^{\t\TM} \left( (\inserted{\ep^2 s^5}+s^3\ep^2\t\TM^2)\phi^5 + {s^5\ep} \t\TM^2(1+\t\TM^2) \phi^7 
+ {\ep}^3 s {\t \TM^2}   \phi^3 
\right)
e^{-2s\alpha}\abs{\t\sol}^2(t,0) \diff t ,
\end{multline*}
which directly implies~\eqref{ineq:13CornilleauGuerrero} and ends the proof of Proposition~\ref{prop:10fromCornilleauGuerrero}.
\end{proof}
We can now go back to the original system and state the following inequality.
\begin{theorem}\label{th:carleman}
Let $\uin \in \Htwo\cap \HoneR$ and  $\sol  \in \Czero[(0,T);\Htwo\cap\HoneR] \cap \Cone[(0,T);\Ltwo]$ 
the unique solution to~\eqref{eq:second-order-L}. There exists $s_0>0$ such that, for weight functions $\alpha$ and $\phi$ defined by~\eqref{def:alpha} and $s\geq s_0 (\ep \TM+\ep \TM^2)$, we have   
\begin{multline}
\label{ineq:thCarleman}
\dst\iint_{\Omega} \left\{s^{-1}\phi^{-1}e^{-2s\alpha}\left(\f{4}{b^2 \ep^2}\abs{\p_t\sol}^2
+\abs{\p^2_{xx}\sol}^2 \right) \right.
\\
\left.+  \left(s^3\phi^3 e^{-2s\alpha} \abs{\t\sol}^2 +s\phi e^{-2s\alpha} \abs{\p_x\t\sol}^2\right)\right\}\diff x\diff t 
 +  s\int_0^{\TM}\phi e^{-2s\alpha}\abs{\p_x\sol}^2(t,L)\diff t
 \\
\lesssim
\int_0^{\TM} \left( \inserted{ \ep^2 s^5 \phi^5 + \ep^2 s^6\phi^7+\ep^3s^3  \phi^3} 
\right)e^{-2s\alpha}\abs{\sol}^2(t,0) \diff t.
\end{multline}
\end{theorem}

\begin{proof}
 We recall~\eqref{ineq:13CornilleauGuerrero}:

\begin{multline*}
\dst\iint_{\tO} \left\{s^{-1}\phi^{-1}e^{-2s\alpha}\left(\abs{\p_t\t\sol}^2
+\abs{\p^2_{xx}\t\sol}^2 \right) +  \left(s^3\phi^3 e^{-2s\alpha} \abs{\t\sol}^2 +s\phi e^{-2s\alpha} \abs{\p_x\t\sol}^2\right)\right\}\diff x\diff t 
\\ 
\qquad \qquad + s\int_0^{\t\TM}\phi e^{-2s\alpha}\abs{\p_x\t\sol}^2(t,L)\diff t
\lesssim
\int_0^{\t\TM} \left( \inserted{ \ep^2 s^5 \phi^5 + \ep^2 s^6\phi^7+\ep^3s^3  \phi^3} 
\right)e^{-2s\alpha}\abs{\t\sol}^2(t,0) \diff t,
\end{multline*}
which implies directly

\begin{multline*}
\dst\iint_{\Omega} \left\{s^{-1}\phi^{-1}e^{-2s\alpha}\left(\f{2}{b \ep}\abs{\p_t\sol}^2 +\f{b\ep}{2}\abs{\p^2_{xx}\sol}^2 \right) 
\right.
\\ \left. 
+  \f{b\ep}{2} \left(s^3\phi^3 e^{-2s\alpha} \abs{\t\sol}^2 +s\phi e^{-2s\alpha} \abs{\p_x\t\sol}^2\right)\right\}\diff x\diff t 
+ \f{b\ep}{2} s\int_0^{\TM}\phi e^{-2s\alpha}\abs{\p_x\sol}^2(t,L)\diff t
\\
\lesssim
\int_0^{\TM} \f{b\ep}{2} \left( \inserted{ \ep^2 s^5 \phi^5 + \ep^2 s^6\phi^7+\ep^3s^3  \phi^3} 
\right)e^{-2s\alpha}\abs{\sol}^2(t,0) \diff t ,
\end{multline*}
which is~\eqref{ineq:thCarleman}.
\end{proof}

\begin{proof}
We now have all the ingredients to conclude the proof of Theorem~\ref{coroll:Carleman}.
We first upper-bound the right-hand side of~\eqref{ineq:thCarleman}.

Denoting $z=\f{\inserted{4}}{\ep^2 \inserted{b^2} t (\TM-t)}\inserted{=\phi(t,0)e^{-2L},}$ we compute the weight  function $\phi^\inserted{\beta}e^{-2s\alpha}$ as
\begin{multline*}f(z)=\phi^\inserted{\beta}e^{-2s\alpha}(t,0)=e^{\inserted{2\beta}L}z^\inserted{\beta} e^{-2s(\lambda-e^{2L})z}
\\
\implies \inserted{\max\limits_z} f(z)=f\left(z_{\min}:=\f{\inserted{\beta}}{2s(\lambda-e^{2L})}\right)=e^{\inserted{2\beta}L} (\f{\inserted{\beta}}{2s(\lambda-e^{2L})})^\inserted{\beta}e^{-\inserted{\beta}}.
\end{multline*}
However, we have $z\in [\f{\inserted{16}}{\ep^2\inserted{b^2}\TM^2},+\infty)$ and we have chosen $s\geq \ep (\TM^2 +\TM)$ and $\ep <1$ so that $z_{\min}=\f{\inserted{\beta}}{2s(\lambda-e^{2L})} \lesssim  \f{\inserted{16}}{\ep^2\inserted{b^2}\TM^2}$. We can thus bound $\phi^\inserted{\beta} e^{-2s\alpha}$ by $f(\f{\inserted{16}}{\ep^2\inserted{b^2}\TM^2})$ and write, \inserted{successively for $\beta=7,\,5,\,3$ in order to bound each term of the right-hand side of~\eqref{ineq:thCarleman},}
$$\int_0^{\TM} {b\ep^\inserted{3}}  s^\inserted{6}\phi^7e^{-2s\alpha}\abs{\sol}^2(t,0) \diff t 
\lesssim b \ep^\inserted{3} s^\inserted{6} e^{14L} (\f{\inserted{4}}{\ep\inserted{b}\TM})^{14}e^{-\f{\inserted{32}s(\lambda-e^{2L})}{\ep^2\inserted{b^2}\TM^2}}\int_0^{\TM} \abs{\sol}^2(t,0) \diff t ,
$$
\inserted{
$$\int_0^{\TM} {b\ep^\inserted{3}}  s^\inserted{5}\phi^5e^{-2s\alpha}\abs{\sol}^2(t,0) \diff t 
\lesssim b \ep^\inserted{3} s^\inserted{5} e^{10L} (\f{\inserted{4}}{\ep\inserted{b}\TM})^{10}e^{-\f{\inserted{32}s(\lambda-e^{2L})}{\ep^2\inserted{b^2}\TM^2}}\int_0^{\TM} \abs{\sol}^2(t,0) \diff t ,
$$
$$\int_0^{\TM} {b\ep^\inserted{4}}  s^\inserted{3}\phi^3e^{-2s\alpha}\abs{\sol}^2(t,0) \diff t 
\lesssim b \ep^\inserted{4} s^\inserted{3} e^{6L} (\f{\inserted{4}}{\ep\inserted{b}\TM})^{6}e^{-\f{\inserted{32}s(\lambda-e^{2L})}{\ep^2\inserted{b^2}\TM^2}}\int_0^{\TM} \abs{\sol}^2(t,0) \diff t .
$$
}
Let us now give a lower bound for the left-hand side of~\eqref{ineq:thCarleman}. On the interval $[\f{\TM}{4},\f{3\TM}{4}]$ we have $\phi \inserted{\approx} \f{1}{\ep^2\TM^2}$ so that
$$s^3\phi^3 e^{-2s\alpha} \gtrsim \f{s^3}{\ep^6\TM^6} e^{-\f{\inserted{8}s(\lambda-e^L)}{\ep^2\inserted{b^2}T^2}}$$
and we then obtain~from~\eqref{ineq:thCarleman}
\[
\f{s^3}{\ep^6\TM^6} e^{-\f{\inserted{8}s(\lambda-e^L)}{\ep^2\inserted{b^2}T^2}} \Vert u\Vert_{L^2([0,T]\times[0,L]}^2
\lesssim
 b \ep^\inserted{3} s^\inserted{6} e^{14L} (\f{\inserted{4}}{\ep\inserted{b}\TM})^{14}e^{-\f{\inserted{32} s(\lambda-e^{2L})}{\ep^2\inserted{b^2}\TM^2}}\int_0^{\TM} \abs{\sol}^2(t,0) \diff t .
\]
which we simplify into
\[ \Vert u\Vert_{L^2([\f{\TM}{4},\f{3\TM}{4}]\times[0,L]}^2
\lesssim
C_{C} \int_0^{\TM} \abs{\sol}^2(t,0) \diff t = \ep^\inserted{3} s^\inserted{3} (\ep\TM)^{-8}e^{-\f{\inserted{8}s(3\lambda+e^L-4e^{2L})}{\ep^2\inserted{b^2}\TM^2}}\int_0^{\TM} \abs{\sol}^2(t,0) \diff t .
\]
Choosing $s\sim \ep(\TM^2+\TM)$ we have $C_C\sim \inserted{\ep^{-2} T^{-5}(1+\TM)^3}e^{\f{c}{\ep}(1+\TM^{-1})}.$ \inserted{We check that the two other terms computed above, coming from $\beta=3$ and $\beta=5,$ give rise to smaller terms, and this ends the proof of Theorem~\ref{coroll:Carleman}.}
\end{proof}
\inserted{This result, together with the exponential stability proved in Proposition~\ref{prop:diffusion-energy-estimate},  provides us with observability in {\it final} time: it suffices to apply~\eqref{ineq:expostab} between $t_0\in[\f{T}{4},\f{3T}{4}]$ and $T$ instead of $0$ and $t,$ and then~\eqref{ineq:Carleman}. We let the details of this proof to the reader. To obtain {\it initial} time stability, we use the log-convexity estimate stated by Proposition~\ref{prop:logconv}. We finally obtain the main result of this article, given by the following observability inequality.}
\begin{theorem}\label{th:initobserv}
Let $\uin \in \Htwo\cap \HoneR$, with $\Vert \uin \Vert_{\Hone} \leq M$, and  $\sol  \in \Czero[(0,T);\Htwo\cap\HoneR] \cap \Cone[(0,T);\Ltwo]$  the unique solution to~\eqref{eq:second-order-L}. Then we have the following observability inequality, for constants $c, C$ depending only on $b$ and $L,$ and defining the increasing function $\rho: x\mapsto x e^x:$
 \begin{equation}\label{estim:observ:final}
 \left\|\uin \right\|_{L^2(0,L)}^{2}+{\ep}\vert \uin(0)\vert^2  \lesssim \f{Ce^{\f{\inserted{2}L}{\ep}}M^2}{
 \rho^{-1} \left(C T^\inserted{7} \inserted{\ep (1+T)^{-3} e^{\f{2L}{\ep}}}{e^{-\f{c}{\ep} (1+T^{-1})}}\f{ M^2}{\int_0^T \vert \sol\vert^2 (t,0) \diff t}\right)
 } \f{T}{\ep}.
 \end{equation}
\end{theorem}
\inserted{\begin{remark}We have kept the constants in order to track a possible improvement in the inequality, however we see that it could be interpreted under an easier form
\[ \left\|\uin \right\|_{L^2(0,L)}^{2} \lesssim \f{M^2}{\rho^{-1} \left(\f{C M^2}{\int_0^T \abs{u}^2(t,0)\diff t}\right)},\] 
which proves that if $\int_0^T \abs{u}^2(t,0)\diff t\to 0$ then the right-hand side goes to zero at a logarithmic rate with respect to $\int_0^T \abs{u}^2(t,0)\diff t$.\end{remark}
Moreover, the a priori on the $\Hone$ norm for $\uin$ guided our choice for a Tikhonov regularisation strategy, as explained in Section~\ref{subsec:tikhonov}.}
\begin{proof}
We depart from~\eqref{eq:backward-estimate}, that we apply for $t\in [\f{T}{4},\f{3T}{4}]$, take its square and integrate in time: this gives
\[\left\|\t\state(0)\right\|_{\stateSpace}^2 \f{T}{2}\leq \exp \left(\frac{\left(\modelOpShift \t\state(0),\t\state(0)\right)_\stateSpace}{\left\|\t\state(0)\right\|_{\stateSpace}^{2}} \f{3 T}{2}\right)\int_{T/4}^{3T/4}\left\|\t\state(t)\right\|_{\stateSpace}^2\diff t.\]
We recall the links between the norms: 
\begin{align*}\left\|\t\state (t)\right\|_{\stateSpace}^{2}&=\left\|\sol(t,\cdot) e^{\f{\cdot}{\ep}}\right\|_{\inserted{\Ltwo}}^{2}+\f{\ep}{2}\vert \sol (t,0)\vert^2,
\\
\left(\modelOpShift \t\state(t),\t\state(t)\right)_\stateSpace &= \f{b\ep}{2} \left\|\p_x(\sol (t,\cdot)e^{\f{\cdot}{\ep}})\right\|^2_{\Ltwo} +\f{b}{2\ep}
\left\| \sol(t,\cdot)e^{\f{\cdot}{\ep}} \right\|_{\Ltwo}^2 +\f{b}{\ep}\vert \sol(0)\vert^2
\\
&\leq C b \f{e^{\f{\inserted{2}L}{\ep}}}{\ep} \Vert \sol (t,\cdot)\Vert^2_{H^1}.
\end{align*}
We now apply Theorem~\ref{coroll:Carleman} to estimate
\begin{multline*} \int_{T/4}^{3T/4}\left\|\t\state(t,\cdot)\right\|_{\stateSpace}^2\diff t=
\int_{T/4}^{3T/4}\left (\left\|\sol(t,\cdot) e^{\f{\cdot}{\ep}}\right\|_{\inserted{\Ltwo}}^{2}+\f{\ep}{2}\vert \sol (t,0)\vert^2 \right)\diff t
\\  \lesssim C e^{\f{c}{\ep} (1+T^{-1})} \inserted{\ep^{-2}T^{-5} (1+T)^3} \int_0^T \vert \sol \vert^2 (t,0) \diff t ,
\end{multline*}
hence 
\[ \left\|\t\state(0)\right\|_{\stateSpace}^2 T \lesssim 
C e^{\f{c}{\ep} (1+T^{-1})} \inserted{\ep^{-2}T^{-5} (1+T)^3}  \exp \left(\frac{Ce^{\f{\inserted{2}L}{\ep}}\Vert \sol(0,\cdot) \Vert^2_{H^1}}{\left\|\t\state(0)\right\|_{\stateSpace}^{2}} \f{T}{\ep}\right) \int_0^T \vert \sol \vert^2 (t,0) \diff t.
\]
We now use the increasing function $\rho(x)=xe^x$ applied to $X=\f{Ce^{\f{\inserted{2}L}{\ep}}\Vert \sol(0,\cdot) \Vert^2_{H^1}}{\left\|\t\state(0)\right\|_{\stateSpace}^{2}} \f{T}{\ep},$ so that the previous inequality reads
\[ T \lesssim 
C e^{\f{c}{\ep} (1+T^{-1})} \inserted{e^{-\f{2L}{\ep}}\ep^{-1}T^{-6} (1+T)^3}\f{1}{ \Vert \sol  (0,\cdot)\Vert_{H^1}^2}X \exp (X) \int_0^T \vert \sol \vert^2 (t,0) \diff t,
\]
hence by applying $\rho^{-1}$
\[ \rho^{-1} \left(C T^\inserted{7} \inserted{\ep (1+T)^{-3} e^{\f{2L}{\ep}}}{e^{-\f{c}{\ep} (1+T^{-1})}}\f{ \Vert \sol  (0,\cdot)\Vert_{H^1}^2}{\int_0^T \vert \sol\vert^2 (t,0) \diff t}\right)
 \lesssim \f{Ce^{\f{\inserted{2}L}{\ep}}\Vert \sol  (0,\cdot)\Vert^2_{H^1}}{\left\|\t\state(0)\right\|_{\stateSpace}^{2}} \f{T}{\ep},
\]
so that
\[\left\|\sol(0,\cdot) \right\|_{\inserted{\Ltwo}}^{2}+\f{\ep}{2}\vert \sol (0,0)\vert^2  \lesssim \f{Ce^{\f{\inserted{2}L}{\ep}}\Vert \sol  (0,\cdot)\Vert^2_{H^1}}{
 \rho^{-1} \left(C T^\inserted{7} \inserted{\ep (1+T)^{-3} e^{\f{2L}{\ep}}}{e^{-\f{c}{\ep} (1+T^{-1})}}\f{ \Vert \sol  (0,\cdot)\Vert_{H^1}^2}{\int_0^T \vert \sol\vert^2 (t,0) \diff t}\right)
 } \f{T}{\ep},
\]
and we conclude thanks to the fact that the function $x\mapsto \f{x}{\rho^{-1}(x)}$ is increasing on $(0,\infty)$: we easily compute that $\rho'(x)=e^x(1+x)$, hence $\f{d}{dx}(\rho^{-1})\left(\rho(x)\right)=\inserted{\f{1}{1+x}}>0$ for $x>0$.

\end{proof}

\subsection{An estimate for Tikhonov regularisation}\label{subsec:tikhonov}
We now use the observability result to propose a regularisation strategy. 
The fact that $\Vert\uin\Vert_{H^1}$ appears in~\eqref{estim:observ:final} drives us to choose the following penalized functional:
We minimize 	
	\begin{equation}\label{eq:least-squares functional}
	\criter_{|T} (\uin) := \frac{\alpha}{2M^2} \Vert \uin\Vert_{H^1}^2
	+   \frac{1}{2\delta^2} \int_0^T  \abs{\observ^\delta(t) - \sol_{|\uin}(t,0)}^2  \,  \diff t ,		
	\end{equation}
	where $\delta>0$ represents the level of noise in the observation, $\observ^\delta$ is the noisy observation and $\sol_{|{\uin}}$ the solution to~\eqref{eq:second-order-L} departing from $\uin$ as an initial data. Noticing that $\criter_{|T}$ is a quadratic functional, we look for an estimate $\bar{\sol}^0_{|T}$ of $\uin$ as the unique minimiser of $\criter_{|T}:$
\begin{equation}\label{def:uinestim}
\bar{\sol}^0_{|T}=\argmin_{{\uin}\in \HoneR}  \criter_{|T} (\uin).
\end{equation}
We have the following error estimate. 
\begin{theorem}\label{th:tikhonov}
Let $\uin \in \HoneR$ with $\Vert \uin\Vert_{H^1}\leq  M,$ and $\sol_{|{\uin}}$ the solution to~\eqref{eq:second-order-L} departing from $\uin$ as an initial data. We denote $\observ: t\mapsto \sol_{|{\uin}}(t,0)$. 

Let $\delta>0$. We assume that we observe $\observ^\delta$ such that
\[\Vert \observ - \observ^\delta\Vert_{L^2(0,T)} \leq \delta.
\]
Let $\inserted{\bar{\sol}^0_{|T}}$ defined by~\eqref{def:uinestim}. Then there exist constants $C_1$ and $C_2,$ depending only on the parameters $L,$ $b$, $T$ and $\ep,$ such that
\begin{equation}
\Vert \uin - \bar{\sol}^0_{|T}\Vert_{L^2(0,L)}^2 \leq \f{C_1 M^2}{{\rho^{-1}(C_2 \f{M^2}{\delta^2})}}.
\end{equation}
\end{theorem}
\begin{proof}
Let us denote $\tilde{\sol}^0=\uin-\bar{\sol}^0_{|T}$ and  $\t\sol=\sol_{|\tilde{\sol}^0}$ the solution to~\eqref{eq:second-order-L} with $\tilde{\sol}^0$ as an initial condition. We apply~\eqref{estim:observ:final} to $\t\sol$ (and replace $M$ by $\Vert\tilde{\sol}^0\Vert_{H^1}$), and obtain
\[\left\|\t\sol^0 \right\|_{L^2(0,L)}^{2}+{\ep}\vert \t\sol^0(0)\vert^2  \lesssim \f{C'_1\Vert \tilde{\sol}^0\Vert_{H^1}^2}{
 \rho^{-1} \left(C_2' \f{ \Vert\tilde{\sol}^0\Vert_{H^1}^2}{\int_0^T \vert \t\sol\vert^2 (t,0) \diff t}\right)
 } .
\]
By the triangular inequality, we first have 
\[\Vert\tilde{\sol}^0\Vert_{H^1}^2\leq 2\Vert\uin\Vert_{H^1}^2+2 \Vert\bar{\sol}^0_{|T}\Vert_{H^1}^2\leq 2 M^2+ 2\Vert\bar{\sol}^0_{|T}\Vert_{H^1}^2.
\]
To estimate $\Vert \bar{\sol}^0_{|T}\Vert_{H^1}^2,$ we use the fact that it minimises $\criter_{|T}:$ by definition, we have
\begin{multline*}
\Vert\bar{\sol}^0_{|T}\Vert_{H^1}^2 
\leq \f{2M^2}{\alpha}\criter_{|T}(\bar{\sol}^0_{|T})\leq \f{2M^2}{\alpha}\criter_{|T}(\uin)
\\
\leq \f{2M^2}{\alpha}\left(\f{\alpha}{2M^2} \Vert \uin\Vert^2_{H^1}+\f{1}{2}\right)
\leq {M^2 \left(1+\f{1}{\alpha}\right)},
\end{multline*}
hence
\[\Vert\tilde{\sol}^0\Vert_{H^1}^2\leq 2M^2(2+\f{1}{\alpha}),
\]
and \inserted{since we have seen that $x\mapsto \f{x}{\rho^{-1}(x)}$ is increasing, this implies, for adapted constants $C_1$ and $C_2'',$
\[ \Vert \uin - \bar{\sol}^0_{|T}\Vert_{L^2(0,L)}^2 \lesssim \f{C_1M^2}{
 \rho^{-1} \left(C_2'' \f{M^2}{\int_0^T \vert \t\sol\vert^2 (t,0) \diff t}\right)
 } .
\]
To conclude, it only remains to prove that $\int_0^T \vert \t\sol\vert^2 (t,0) \diff t \leq C^{st} \delta^2.$ This comes once again from the minimisation of $\criter_{|T} $ and the triangular inequality: we first have
\begin{multline*}
\int_0^T \vert \t\sol\vert^2 (t,0) \diff t \leq 2 \int_0^T \vert \observ-\observ^\delta\vert^2 (t,0) \diff t  + 2 \int_0^T \vert \observ^\delta - \sol_{|{\bar{\sol}^0_{|T}}}\vert^2 (t,0) \diff t 
\\
\leq 2\delta^2 + 4 \delta^2 \criter_{|T}(\bar{\sol}^0_{|T})
\leq 2\delta^2 + 4 \delta^2 \criter_{|T}(\uin) \leq 2\delta^2 + 4 \delta^2 (\f{\alpha}{2}+\f{1}{2})  \leq C^{st} \delta^2 
\end{multline*}
We conclude by taking $C_2=\f{C_2''}{C^{st}},$ due to the fact that $x\mapsto \f{1}{\rho^{-1}(\f{1}{x})}$ is once again an increasing function.
}

\end{proof}
Let us note that Theorem~\ref{th:tikhonov} holds true for a regularisation term taken in any norm equivalent to the $H^1-$ norm.

\section{An estimation strategy based on Kalman filtering}

We now propose a practical approach for solving the estimation problem using the functional $\criter_T$ based on Kalman filtering. The resulting sequential approach allows to compute $\bar{\adjoint}_{\T}$ as time $T$ increases, illustrating how the observability is modified through time.

Let us define the following observation operator 
\[
	\obsOp : \left|
	\begin{aligned}
		\stateSpace \to \R, \\
		\state = \begin{pmatrix}
			u \\ v
		\end{pmatrix} &\mapsto \sqrt{\frac{2}{\varepsilon}}v 
	\end{aligned} \right.
\]
which is obviously bounded.

The criterion $\criter_T$ can be rewritten for $\initNoise \in \stateSpace$ instead of $\uin \in H^1:$ 
\begin{equation}\label{eq:least-squares-functional-state-space}
	\criter_{|T} (\initNoise) := \frac{\alpha}{2M^2} a_0( \initNoise,\initNoise)
	+   \frac{1}{2\delta^2} \int_0^T  \abs{\observ^\delta(t) - \obsOp \state_{|\initNoise}(t,0)}^2  \,  \diff t ,		
\end{equation}
where $\state_{|\initNoise}$ is the solution of \eqref{eq:semigroup-L} 
\[
\begin{cases}
		\dfrac{\diff}{\diff t} \state + \modelOp \state  = 0, \quad t \in [0,T]\\
	\state(0) = \hat{\state}_0 + \initNoise
\end{cases}
\]
and the symmetric bilinear form
\[
	\forall z = (u,v) \in \mathcal{V},q = (w,k)\in \mathcal{V},\quad a_{0}(z,q) = \int_0^L u'w' \diff x +  vk
\]
 is coercive on $\mathcal{V}$ from the Poincaré inequality and defines a norm for $\initNoise$ which is equivalent to the $H^1$ norm taken for $\uin$ in~\eqref{eq:least-squares functional}. Therefore, we introduce the operator $(N_0,\mathcal{D}(N_0))$ such that 
\[
\begin{cases}
		\mathcal{D}(N_0) = \Bigr\{ z = (u,v) \in \mathcal{V} \text{ such that } \exists \beta = (f,g) \in \stateSpace \text{ such that } \\\hspace{6cm }\forall q = (w,k)\in \mathcal{V},\quad a_{0}(z,q) = (f,q) \Bigl\},\\
		\forall z \in \mathcal{D}(N_0), \quad N_0 \state = \beta \end{cases}
\]
The operator $N_0$ reads
\[
	\forall z = (u,v) \in \mathcal{D}(N_0), \quad N_0 \state = \begin{pmatrix}
		u'' \\
		\sqrt{\frac{\varepsilon}2} u'(0) \\
	\end{pmatrix}.
\]
Since $a_{0}$ is coercive, and the injection from $\mathcal{V}$ into $\stateSpace$ is compact from the Rellich–Kondrachov theorem, we have that $N_0$ is invertible and the symmetric operator $\cov_0 = \frac{M^2}{\delta^2} (\alpha N_0)^{-1} \in \mathcal{S}^+(\stateSpace)$ the space of bounded positive self-adjoint operators.  

Note finally that in the case $\varepsilon = 0$, we can use a classical $\mathrm{L}^2$ regularization instead of $a_0$.

\subsection{The two-ends optimality system}

Let us start by characterizing the unique minimizer $\bar{\initState}_{\T}$ of $\criter_T$ using the adjoint variable associated with the dynamics constraint in the criterion definition.

\begin{theorem}\label{thm:minimum-criterion-generalized}
Defining $\cov_0 =  \frac{M^2}{\delta^2} (\alpha N_0)^{-1}$, the minimizer of $\criter_T$ is characterized by
\begin{equation}\label{eq:optimality-heat}
	\bar{\initState}_{\T} = \cov_0 \bar{\adjoint}_{\T}(0),\end{equation}
where $\bar{\adjoint}_{\T} \in W_T$ and $\bar{\state}_{\T} = \state_{\bar{\initState}_{\T}} \in W_T$  satisfy  
\begin{equation}\label{eq:2end-estimator-heat-generalized}
	\begin{cases}
		\dot{\bar{\state}}_{\T} + \modelOp \bar{\state}_{\T} = 0, &\quad \text{in } (0,T)\\
		\dot{\bar{\adjoint}}_{\T} - \modelOp^\ast \bar{\adjoint}_{\T} = - \obsOp^* (\observ^\delta - \obsOp \bar{\state}_{\T}(t)),&\quad \text{in } (0,T) \\
		\bar{\state}_{\T}(0) = \hat{\state}_0 + \cov_0 \bar{\adjoint}_{\T}(0) \\
		\bar{\adjoint}_{\T}(T) = 0
	\end{cases}	
\end{equation}	
\end{theorem}
The coupled dynamics \eqref{eq:2end-estimator-heat-generalized} classically encountered in such a minimization strategy with a dynamics constraint is often referred to as a ''two-ends'' problem, following \cite{bensoussan-filtrage-book}, since the optimal trajectory is defined forward in time whereas the adjoint variable is defined backward in time from a null final condition.

Moreover, since $\bar{\adjoint}_{\T} \in \mathcal{W}_T$ and the space $\mathcal{W}_T$ has a continuous injection into $\Czero[{[0,T];\stateSpace}]$, we have that $\bar{\adjoint}_{\T}(0) \in \stateSpace$, so $\bar{\state}_{\T} \in \Czero[{[0,T];\stateSpace}]$ from semigroup theory. Hence, we get
$\bar{\adjoint}_{\T} \in \Czero[{[0,T];\stateSpace}]$. 

In practice, our estimation problem can be solved by solving the two-ends problem \eqref{eq:2end-estimator-heat-generalized} as a space-time boundary problem, or by defining an iterative strategy where we solve a sequence of Cauchy problems. Then at each iteration, a forward solution is computed and followed by an adjoint dynamics computation until convergence to the initial condition that minimizes the criterion.

\subsection{The equivalent Kalman observer}

As an alternative to solving \eqref{eq:2end-estimator-heat-generalized}, we propose a sequential strategy based only on the Cauchy formulation of the so-called Kalman estimator.  

To this end, let us first introduce the following Riccati dynamics
\begin{equation}\label{eq:riccati}
\begin{cases}
\dot{\cov} + \modelOp \cov +  \cov \modelOp^\ast +  \cov \obsOp^\ast \obsOp \cov = 0, \quad  t > 0 \\
\cov(0)  = \cov_0
\end{cases}	
\end{equation}
to be solved in the space $\mathcal{S}^+(\stateSpace)$ of bounded positive self-adjoint operator. 

\begin{theorem}\label{thm:existence-riccati-strict}
	Considering $\cov_0 = \frac{M^2}{\delta^2} (\alpha N_0)^{-1}\in \mathcal{S}^+(\stateSpace)$, there exists one and only one strict solution $\cov \in \Cone[{[0,T];\mathcal{S}^+(\stateSpace)}]$ of the Riccati dynamics~\eqref{eq:riccati}.
\end{theorem}

\begin{proof}
	Let us define the bilinear form
\[
	b_{\cov_0} : \left|
	\begin{aligned}
		\mathcal{D}(\modelOp^\ast) \times \mathcal{D}(\modelOp^\ast) &\to \R, \\
		(\state,v) &\mapsto (\cov_0\state,\modelOp^\ast v)_{\stateSpace} + (\cov_0 v,\modelOp^\ast \state)_{\stateSpace},
	\end{aligned} \right.
\]
We can show that $b_{\cov_0}$ can be continuously extended as a bilinear form of $\stateSpace \times \stateSpace$. Therefore, by \cite[IV-1 Proposition 3.2]{bensoussan-daprato-book}, there exists one and only one strict solution $\cov \in \Cone[{[0,T];\mathcal{S}^+(\stateSpace)}]$ of ~\eqref{eq:riccati}. 
\end{proof}

Since, $\cov \in \Cone[{[0,T];\mathcal{S}^+(\stateSpace)}]$, we deduce as a direct application of the bounded perturbation result \cite[Proposition II-1 3.4]{bensoussan-daprato-book} the following existence result.

\begin{theorem}\label{thm:kalman-existence}
	Given $\observ^\delta \in L^2([0,T];\stateSpace)$, there exists one and only one \emph{mild} solution in $C^0([0,T];\stateSpace)$ of the dynamics
	\begin{equation}
	\label{eq:kalman}
	\begin{cases}
		\dot{\hat{\state}} + \modelOp \hat{\state} =  \cov \obsOp^* (\observ^\delta - \obsOp \hat{\state}), \quad \text{ in } (0,T)\\
		\hat{\state}(0) = \hat{\state}_0
	\end{cases}	
	\end{equation}
\end{theorem}

The two Cauchy problems \eqref{eq:kalman} and \eqref{eq:riccati} defining $\hat{\state}$ and $\cov$ are fundamental to decouple the two-ends problems \eqref{eq:optimality-heat} of solution $(\bar{\state}_{\T},\bar{\adjoint}_{\T})$ as stated by the next theorem, see a proof in \cite[Theorem 1.14]{aussal:hal-03911500}.

\begin{theorem}\label{thm:kalman}
	Let $\cov$ be a strong solution of \eqref{eq:riccati} and $\hat{\state}$ the solution of the dynamics \eqref{eq:kalman} defined from $\cov$. We have the fundamental identity
	\begin{equation}
		\label{eq:link-kalman-optimal}
			\forall t \in [0,T], \quad \bar{\state}_{\T}(t) = \hat{\state}(t) + \cov(t)\bar{\adjoint}_{\T}(t).
	\end{equation}
\end{theorem}

In particular we have at time $T$, $\bar{\state}_{\T}(T) = \hat{\state}(T)$ . Moreover, if we consider in the criterion $\criter_{|t}$ the available measurements until $t$, we find that  
\[
	\forall t \in [0,T], \quad \bar{\state}_{t}(t) = \hat{\state}(t)
\]
The solution $\hat{\state}$ is called the Kalman estimator which therefore gives an equivalent alternative to the minimization of the criterion $\criter_{|T}$. 

Our initial objective is to estimate the initial condition, hence to compute $\bar{\initState}_T$ which is not computed by the Kalman estimator $\hat{\state}$. However, injecting \eqref{eq:link-kalman-optimal} in the adjoint dynamics of \eqref{eq:2end-estimator-heat-generalized}, we get
\begin{equation}\label{eq:adjoint-dynamics-from-kalman}
	\begin{cases}
		\dot{\bar{\adjoint}}_{\T} - \big(\modelOp^\ast +  \obsOp^\ast \obsOp \cov(t)\big) \bar{\adjoint}_{\T} = -  \obsOp^* \big(\observ^\delta - \obsOp \hat{\state}\big),\quad t \in (0,T)\\
		\bar{\adjoint}_{\T}(T) = 0
	\end{cases}
\end{equation}
with still $\bar{\initState}_T = \cov_0 \bar{\adjoint}_{\T}(0)$. Note that \eqref{eq:adjoint-dynamics-from-kalman} is well-posed for the same reason that \eqref{thm:kalman} is well-posed. Therefore, a backward dynamics allows to reconstruct the initial condition in only one iteration. Unfortunately, the dynamics \eqref{eq:adjoint-dynamics-from-kalman} is not of practical use as it should be computed backward but the Riccati solution $\cov$ remains to be computed forward in time.

As an alternative, we introduce the evolution operator $\varLambda \in C^0([0,T];\mathcal{L}(\stateSpace))$ such that $r : t \mapsto \varLambda(t) r_0$ is solution of  
\begin{equation}\label{eq:riccati-cond-init}
\begin{cases}
\dot{r} + (\modelOp  +  \cov \obsOp^* \obsOp) r  = 0\quad \text{ in } (0,T)\\
r(0) = \cov_0 r_0
\end{cases}
\end{equation}
For all $r_0 \in \stateSpace$, we remark that $r(t) = \varLambda(t)r_0$ is solution of \eqref{eq:kalman} with $r(0) \in \mathcal{D}(\modelOp)$ and $\observ = 0$. Hence $r \in \Cone[{[0,T];\stateSpace}] \cap \Czero[{[0,T];\mathcal{D}(\modelOp)}]$. Then, we introduce the estimator 
\begin{equation}
\begin{cases}
\dot{\hat{\initState}} =  \varLambda^\ast \obsOp^\ast (\observ^\delta - \obsOp \hat{\state}), \quad \text{ in } (0,T)\\
\hat{\initState}(0) = 0
\end{cases}
\end{equation}
which admits a weak solution in $\Ltwo[{(0,T);\stateSpace)}]$  as $t \mapsto \varLambda^\ast(t) \obsOp^* (\observ^\delta(t) - \obsOp \hat{\state}(t)) \in \Ltwo[{(0,T);\stateSpace)}]$. This new operator and dynamics allow to reconstruct the initial condition as stated in the next proposition. In other words, by solving one additional forward Riccati equation and one additional forward dynamics, we will be able to directly estimate $\bar{\initState}_{\T}$.
\begin{proposition}\label{eq:}
	For all $T\geq 0$, we have the following identity 
\begin{equation}\label{eq:identity-kalman-init}
	\forall t \in [0,T],\quad \bar{\initState}_{\T} =  \hat{\initState}(t) 	+ \varLambda^\ast(t) \bar{\adjoint}_{\T}(t).
\end{equation}
Hence $\hat{\initState}$ is a Kalman estimator of the initial condition in the following sense: 
\begin{equation}
	\forall t \geq 0,\quad  \hat{\initState}(t) = \bar{\initState}_{t}.	
\end{equation}
\end{proposition}

\begin{proof}
	We denote $\eta\, :\,  t \mapsto \hat{\initState}(t) 	+ \varLambda^\ast(t) \bar{\adjoint}_{\T}(t)$, and consider $v \in \stateSpace$. Both $\varLambda^\ast v$ and $\bar{\adjoint}_{\T}$ belong to $\Cone[{([0,T];\stateSpace)}]$, hence we can compute	
	\[
	\begin{split}
		\frac{\diff}{\diff t} (\eta,v)_{\stateSpace} &= \frac{\diff}{\diff t} \bigl(  \hat{\initState}(t), v\bigr)_\stateSpace + \frac{\diff}{\diff t} \bigl( \bar{\adjoint}_{\T}(s), \varLambda^\ast(t) v\bigr)_{\stateSpace}\Big|_{s = t} + \frac{\diff}{\diff t} \bigl(\bar{\adjoint}_{\T}(t), \varLambda^\ast(s) v\bigr)_{\stateSpace}\Big|_{s = t}	\\
		&= \bigl(   (\observ^\delta - \obsOp \hat{\state}), \obsOp \varLambda^\ast(t) v\bigr)_{\stateSpace} - \bigl(\bar{\adjoint}_{\T}(t), (\modelOp - \cov \obsOp^* \obsOp) \varLambda^\ast(t) v \bigr)_{\stateSpace} \\
		&\hspace{1cm} - \bigl(   (\observ^\delta - \obsOp \hat{\state}), \obsOp \varLambda^\ast(t) v\bigr)_{\stateSpace} + \bigl(\bar{\adjoint}_{\T}(t), (\modelOp - \cov \obsOp^* \obsOp) \varLambda^\ast(t) v \bigr)_{\stateSpace}\\
		&= 0.
	\end{split}
	\]
	So $\eta \equiv \eta(0) =  \cov_0 \bar{\adjoint}_{\T}(t) = \bar{\initState}_{\T}$, which justifies \eqref{eq:identity-kalman-init}. Finally $\adjoint_{\T}(T) = 0$  in \eqref{eq:identity-kalman-init} gives $\hat{\initState}(T) = \bar{\initState}_{\T}$. 
\end{proof}

\section{Numerical solution}\label{sec:num}
\subsection{Discretization aspects}

\subsubsection{Model discretization} 

 To illustrate the proposed asymptotic models, we perform a space-time discretization of \eqref{eq:second-order} using finite differences and a forward Euler time-scheme. \inserted{We discretize an interval $[0,L]$ and add a Dirichlet boundary condition on $x=L.$} 
\begin{equation}\label{eq:time-scheme-direct}
			\begin{cases}
			\dfrac{u^\varepsilon_{n+1,i} - u^\varepsilon_{n,i}}{\Delta t} - b \dfrac{u^\varepsilon_{n,i+1} - u^\varepsilon_{n,i}}{h} \\[0.2cm] 
			\hspace{3cm}- \dfrac{b\varepsilon}2 \dfrac{u^\varepsilon_{n,i+1} -  2u^\varepsilon_{n,i} + u^\varepsilon_{n,i-1}}{h^2} = 0 , \quad 1\leq i \leq n \\[0.2cm]
						\dfrac{u^\varepsilon_{n+1,0} - u^\varepsilon_{n,0}}{\Delta t} - b \dfrac{u^\varepsilon_{n,1} - u^\varepsilon_{n,0}}{h} = 0 \\
			u^\varepsilon_{n,N} = 0
		\end{cases}
\end{equation}
Note that if $\Delta t b h^{-1} + b \varepsilon h^{-2} \leq 1$, for all $n \in \N$ and all $1\leq i\leq N$, $u^\varepsilon_{n+1,i}$ is a convex combination
\[
	u^\varepsilon_{n+1,i} = \sum_{j=0}^N \lambda_j u^\varepsilon_{n,j}, \text{ with }  0 \leq \lambda_j \leq 1 \text{ and } \sum_{j=0}^N \lambda_j =1.
\]
In this case the scheme is stable under the CFL condition $\Delta t  \leq b^{-1}(1 + \varepsilon h^{-1})^{-1} h$.  Furthermore, this scheme satisfies a maximum principle, namely
\[
	\forall n \in \N, \forall 1\leq i \leq N,\quad \min_{0\leq i \leq N} (u_0,i) \leq u_{n,i} \leq \max_{0\leq i \leq N} (u_0,i).
\]
Additionally, the transport model $\varepsilon=0$ can be simulated without additional numerical dissipation by simply choosing $b \Delta t = h$.

\inserted{
Similarly, the Becker-Döring model is discretized using a forward-Euler time-discretization, so that we have
\begin{equation}\label{eq:discrete-bd}
	\frac{C_i^{n+1} - C_i^n}{\Delta \tau} = b(C_{i+1}^{n} - C_i^{n}),
\end{equation}
which ensure, for the same reasons, that $C_i^n \leq 0$ for all $i \geq i_0$ and all $n \geq 0$.

}

Note finally that the resulting discretization scheme can all be rewritten in the following abstract form
\begin{equation}
\label{eq:dynsys-linear-ode-discrete}
	\begin{cases}
	\mathrm{\state}_{n+1}(t) =  \Upphi_n \mathrm{\state}_{n},\quad t > 0,\\ 
	\state(0) = \hat{\mathrm{\state}}_{0} + \upzeta
\end{cases}
\end{equation}
where typically for the asymptotics, $\mathrm{\state} = (u_0 \ldots u_N)^\intercal$ and 
\[
	\Upphi_n = \mathbbm{I}\mathrm{d}_{N+1} - \frac{b\Delta t}h \begin{pmatrix}
		1 & -1 &  & \\
		& \ddots & \ddots & \\
		& & \ddots & -1 \\
		& &  & 1
	\end{pmatrix} - \frac{b\varepsilon\Delta t}{2 h^2} 
	\begin{pmatrix}
		0 & 0 &  & &\\
		-1 & 2 & -1 & &\\
		& \ddots & \ddots & \ddots &\\
		& & \ddots & \ddots & -1 \\
		& & & -1 & 2
	\end{pmatrix},
\]
We assumed that the observations are interpolated on the model time-grid leading to the available sequence $(\mathrm{\observ}_n^\delta)_{0 \leq n \leq\N_{\T} }$. The observation operator associated with a boundary observation is given by $\obsOpDof = (1 \, 0 \cdots 0)$.

\subsubsection{Discretizing the Kalman estimators}
We here recall that the discretization of the time-continuous Kalman estimator \eqref{eq:kalman} can be performed using the discrete-time Kalman estimator: 
\begin{equation}\label{eq:kalman-time-discrete}
\begin{cases}
\text{Initialization:}\\
		\hat{\stateDof}_{0}^- = \hat{\stateDof}_{0} \text{ and } \hat{\upzeta}_{0}^- = 0,\\[0.1cm]
		\covDof_{0}^{-} = \covDof_{0} \text{ and } \Lambda_{0}^{-} = \covDof_{0}, \\[0.2cm]
\text{Correction ($n \geq 0$):}\\
		\mathrm{K}_n = \Delta t \covDof_{n}^{-} \obsOpDof^\intercal (\Delta t \obsOpDof \covDof_{n}^{-} \obsOpDof^\intercal +   \Id)^{-1}, \\[0.1cm] 
		\mathrm{L}_n = \Delta t (\Lambda_{n}^{-})^\intercal \obsOpDof^\intercal ( \Delta t \obsOpDof \covDof_{n}^{-} \obsOpDof^\intercal +   \Id)^{-1}, \\[0.1cm]
		\hat{\stateDof}_{n}^{+} = \hat{\stateDof}_{n}^{-} + \mathrm{K}_n (\mathrm{\observ}_n^\delta - \obsOpDof \stateDof_n ), \\[0.1cm]
		\hat{\upzeta}_{n}^+ = \hat{\upzeta}_{n}^{-} + \mathrm{L}_n (\mathrm{\observ}_n^\delta - \obsOpDof \stateDof_n ),  \\[0.1cm]
		\covDof_{n}^{+} = \covDof_{n}^{-} - \mathrm{K}_n \obsOpDof \covDof_{n}^{-}, \\[0.1cm]
		\Lambda_{n}^+ = \Lambda_{n}^{-} - \mathrm{K}_n \obsOpDof \Lambda_{n}^{-}, \\[0.2cm]
\text{Prediction ($n \geq 0$):}\\
		\hat{\state}_{n+1}^{-} = \Upphi_n \hat{\state}_{n}^+, \\[0.1cm]
		\hat{\upzeta}_{n+1}^{-} = \hat{\upzeta}_{n}^{+}, \\[0.1cm]
		\covDof_{n+1}^{-} = \Upphi_n \covDof_{n}^{+} \Upphi_n^\intercal, \\[0.1cm]
		\Lambda_{n+1}^- = \Upphi_n \Lambda_{n}^+.
\end{cases}
\end{equation}
In fact as in the time-continuous infinite dimensional setting, the estimator \eqref{eq:kalman-time-discrete} can be associated with the minimization of a discretization of the criterion \eqref{eq:least-squares-functional-state-space}, namely 
\[
	\criter_{n}^{\Delta t-}(\upzeta) 
		 = \frac{1}2 \bigr(\upzeta,\covDof_0^{-1} \initNoise^h\bigl)_{\stateSpace}  
		 + \frac{1}2 \sum_{k=0}^{n-1} \Delta t \delta^{-2} |\observ_k - \obsOpDof \stateDof_{k|\upzeta}|^2,  
\]
and we refer to \cite{moireau:hal-03921465} for further comment on the time-discretization aspects of the Kalman filter. 

\subsection{Numerical illustrations}

\subsubsection{Experiments with direct models} As a first illustration, we consider the following configuration. We simulate \eqref{eq:depol-ode} with $b =1$, during the time-period $\mathcal{T} = 100$, and with polymers of size $i\in[2,100]$ (namely $i_0 = 2$).  We follow the discretization \eqref{eq:discrete-bd} with an \emph{overkill} time-step $\Delta \tau = 10^{-3}$ to limit the impact of the chosen discretization. By comparison, we simulate \eqref{eq:first-order} and \eqref{eq:second-order} on the space domain $[0,L]$ with $L = 1$ with \eqref{eq:time-scheme-direct} using two values $\vareps = \{0,10^{-2}$\}. The spatial discretization is set to $h=0.1$ for the case $\vareps = 0$ and  $h=0.2$ for the case $\vareps = 10^{-2}$. The final time is $T = 1$ and the time-step is chosen as $\Delta t = 10^{-2}$ which provides for $\epsilon = 0$ an exactly conservative time-scheme, and ensures that the CFL condition is fulfilled for $\vareps = 10^{-2}$. All these parameters are also reported Table~\ref{tab:param-direct}. The initial condition is a truncated Gaussian function centered in $i=50$ of variance $50$ and translated and scaled so that $u_0(L) = 0$ and $\sup_{x\in[0,L]} u_0(x) = 1$.

\begin{table}[h]
	\begin{tabular}{ |p{2.5cm}||p{2.5cm}|p{2cm}|p{3cm}|  }
	\hline
	\multicolumn{4}{|c|}{Physics} \\
	\hline
	Parameters & Becker-Döring (BD) & Transport (Tr) & Advection-diffusion (AD)\\
	\hline
	$N_{\text{poly}}$ & 100 & - & - \\
	$i_0$ & 1 & - & -\\
	$L$ & - & 1 & 1 \\
	$T$ & 100 & 1 & 1 \\
	$b$ & 1 & 1 & 1 \\
	$\varepsilon$ & - & 0 & $10^{-2}$ \\
	\hline
	\multicolumn{4}{|c|}{Discretization} \\
	\hline
	$\Delta t$ & $10^{-3}$ & $10^{-2}$ & $5.10^{-3}$ \\
	$h$ & - & $10^{-2}$ & $10^{-2}$ \\
	\hline
	\end{tabular}
	\caption{Direct models and associated parameters}
	\label{tab:param-direct}
\end{table}

The resulting simulations are compared in Figure~\ref{fig:direct-problem-frames}, where we rescale the asymptotic simulations according to \eqref{def:ueps} by plotting $([2,100],[0,T]) \ni (i,t) \mapsto u^\varepsilon(\varepsilon i,\varepsilon t)$. We also plot in Figure~\ref{fig:moment-rescaled} the resulting moment $\mu_k , k=\{0,1,2\}$ calculated with each model. This provides information about how accurate the moments reconstructed from the asymptotic models are, and thus also about the type of minimum bias that should be expected when using an asymptotic model with an observation generated by the Becker-Döring model.

\begin{figure}[h]
	\includegraphics[width=13cm]{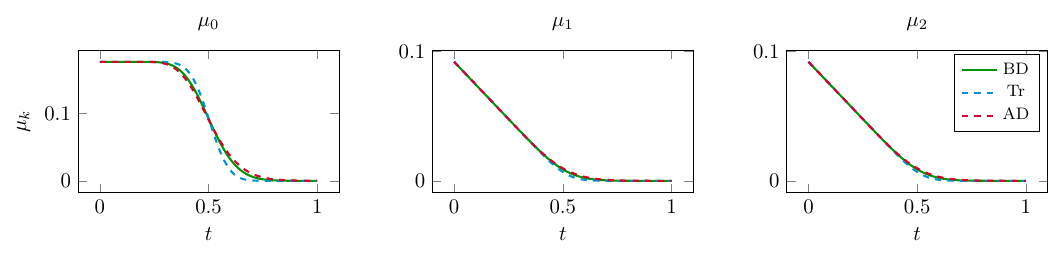}
	\caption{Moments computed from the initial Becker-Döring model (BD, green), the Transport model (Tr, dashed-blue), or the Advection-Diffusion model (AD, dashed-red)}
	\label{fig:moment-rescaled}
\end{figure}

\begin{figure}[p]
	\includegraphics{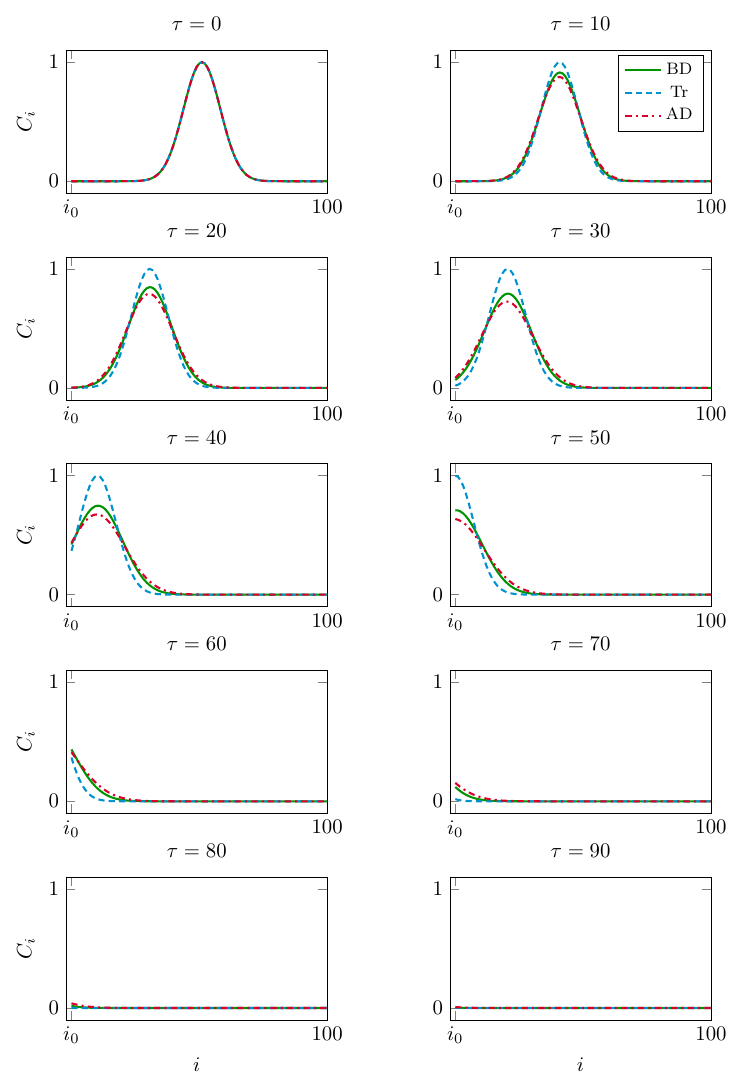}
	\caption{Snapshots of the resulting depolymerization obtained using the initial Becker-Döring model (BD, green), the Transport model (Tr, dashed-blue), or the Advection-Diffusion model (AD, dashed-red)}
	\label{fig:direct-problem-frames}
\end{figure}

\begin{figure}[p]
	\includegraphics{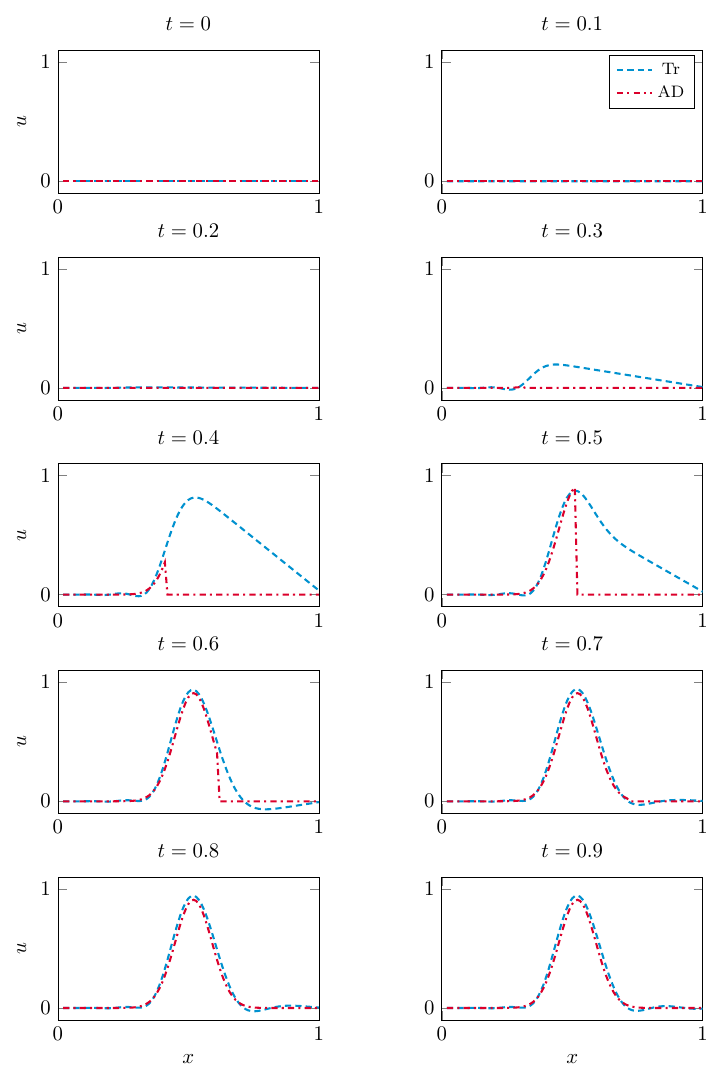}
	\caption{Snapshots of the identity approximation \eqref{eq:involution} for the Transport model (Tr, dashed-blue), or the Advection-Diffusion model (AD, dashed-red)}
	\label{fig:estimation-from-boundary-frames}
\end{figure}

\subsubsection{Boundary observability assessment} A second series of numerical results illustrates the evolution with respect to time of the observability associated with our Carleman estimates. Here, we propose for the two asymptotic models, and for our previously defined Gaussian initial condition $u_0$, to compute
\begin{equation}\label{eq:involution}
	t \mapsto \left[\Pi_0^{-1} + (\Psi_{t;\varepsilon,h,\Delta t}^{u_0|Tr})^\ast\Psi_{t;\varepsilon,h,\Delta t}^{u_0|Tr}\right]^{-1} (\Psi_t^{u_0|Tr})^\ast \Psi_{t;\varepsilon,h,\Delta t}^{u_0|Tr} u_0.	
\end{equation}
Moreover, here $M=\delta=1$, while $\alpha = 10^-4$, so that $\Pi_0^{-1}$ has a ``vanishing regularization effect''. Therefore, the operator defined by \eqref{eq:involution} should be an increasingly accurate approximation of the identity as time increases. Moreover, the reconstruction indicates which part of the initial condition is observable. It is striking that for the advection-diffusion model, the solution can be reconstructed for early times, which illustrates the observability at small times. However, the quality of the reconstruction is clearly affected by the fact that the observability is weak for early time. In contrast, for the transport problem, the part of the solution that can be reconstructed at each time corresponds to the part that would have been transported to the boundary. But then, the reconstruction for this part is very good, since the stability estimate for the transport equation indicates a well-posed inversion for the observable part.

\subsubsection{Robustness of boundary measurement inversion with respect to measurement errors} We then illustrate the reconstruction of the initial condition reconstruction from boundary observations generated by the Becker-Döring model. Figure~\ref{fig:observation}-(Left) shows the rescaled boundary condition generated by the Becker-Döring model compared to the boundary observations generated from a corresponding initial condition with the two asymptotic models. This again illustrates the degree of divergence between the three models from the point of view of boundary observations. The reconstruction is presented in Figure~\ref{fig:estimation-from-BD-frames} with $M=1$, $\delta = 0.1$ and $\alpha = 1$. We additionally contaminate the Becker-Döring boundary observation with an additive Gaussian error corresponding to $10\%$ to the signal maximum and display the reconstruction in Figure~\ref{fig:estimation-from-noisy-BD-frames}. Here, we can clearly see that the transport model in Figure~\ref{fig:estimation-from-BD-frames} provides a less accurate reconstruction, as there is no diffusion in the transport model. However, the reconstruction for the transport model appears to be more robust to noisy observations in Figure~\ref{fig:estimation-from-noisy-BD-frames}.

\begin{figure}[h]
	\includegraphics[width=6cm]{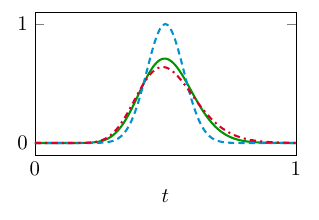}
	\includegraphics[width=5cm]{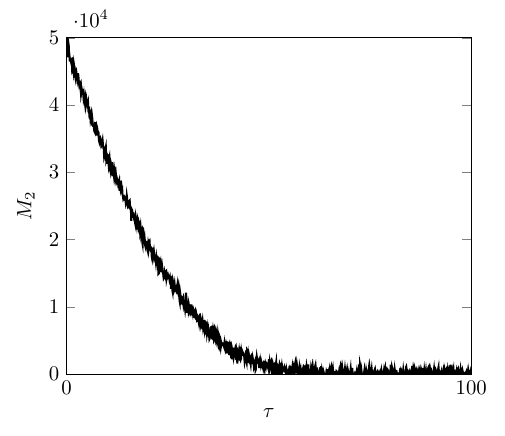}
	\caption{(Left) Boundary observation generated by the Becker-Döring model (BD, green), the Transport model (Tr, dashed-blue), or the Advection-Diffusion model (AD, dashed-red). (Right) Second moment observation from Becker-Döring model with additive noise.}
	\label{fig:observation}
\end{figure}

\begin{figure}[htbp]
	\includegraphics{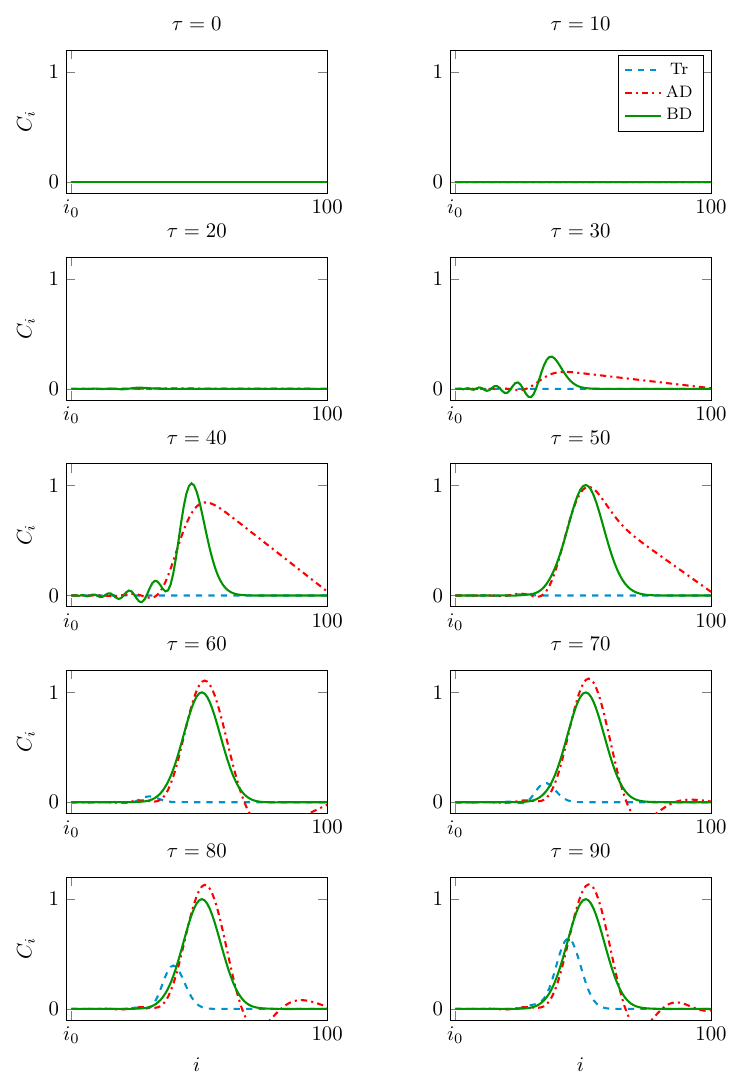}
	\caption{Snapshots of the reconstruction from a boundary observation generated from the Becker-Döring model using the Becker-Döring model (BD, green), the Transport model (Tr, dashed-blue), or the Advection-Diffusion model (AD, dashed-red)}
	\label{fig:estimation-from-BD-frames}
\end{figure}

\begin{figure}[htbp]
	\includegraphics{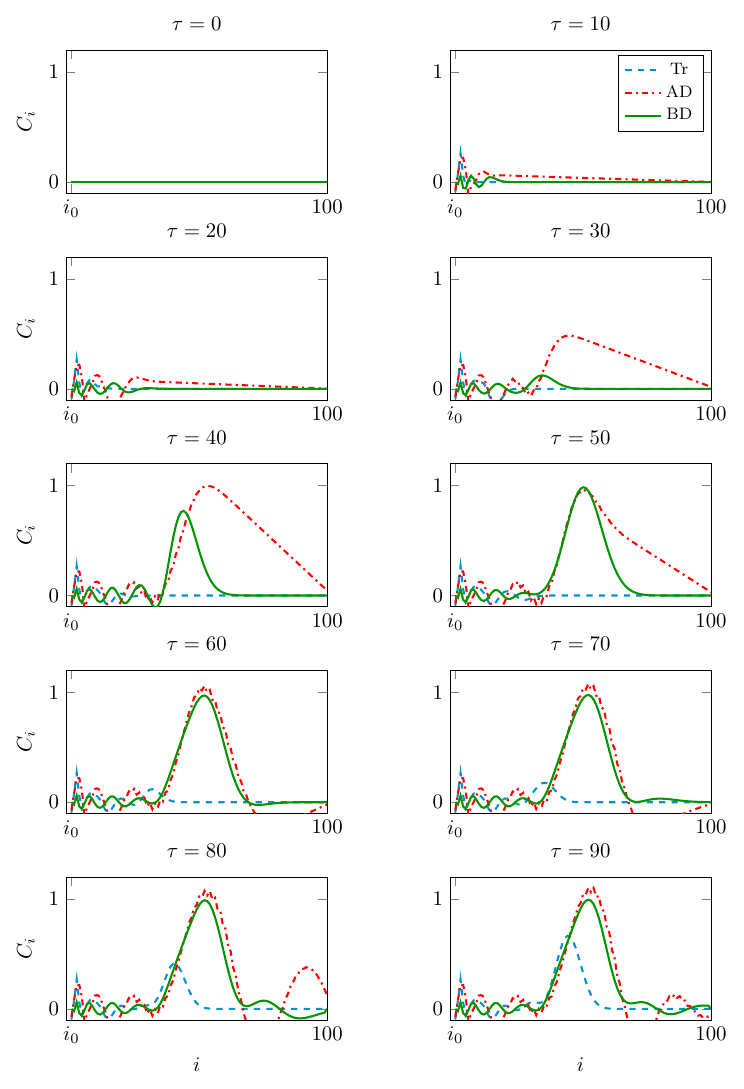}
	\caption{Snapshots of the reconstruction from a noisy boundary observation generated from the Becker-Döring model using the Becker-Döring model (BD, green), the Transport model (Tr, dashed-blue), or the Advection-Diffusion model (AD, dashed-red)}
	\label{fig:estimation-from-noisy-BD-frames}
\end{figure}

\subsubsection{Inversion of a second moment in a synthetic but realistic configuration} Finally, we present a typical reconstruction in a synthetic but realistic configuration, where we have a noisy observation of the second moment at our disposal, as it was considered in biological experimenters in \cite{armiento:2016}. We corrupt the second moment computed from the Becker-Döring model by a white noise of amplitude $1\%$ of the maximum signal value. In addition, the negative value of the resulting signal is deleted to obtain a more realistic noise for such an observation, see Figure~\ref{fig:observation}. The resulting reconstruction is shown in Figure~\ref{fig:estimation-from-noisy-moment-frames} with an initial covariance $M=1$, $\delta = 0.01$ and $\alpha = 10^-1$ for the transport case and $1$ for the advection diffusion case. The initial condition prior is changed here to a Gaussian variable centered in $x=1/3$. This was inspired by the realistic cases in \cite{armiento:2016}, where biologists defined a plausible initial condition, albeit not well centered. Here, we see the impact of manipulating the second moment instead of a boundary condition on the overall reconstruction. Even using the exact Becker-Döring model that generated the observation leads to an oscillatory reconstruction. Again, the advection-diffusion model captures the exact peak value better, but at the cost of a more fluctuating reconstruction.

\begin{figure}[htbp]
	\includegraphics{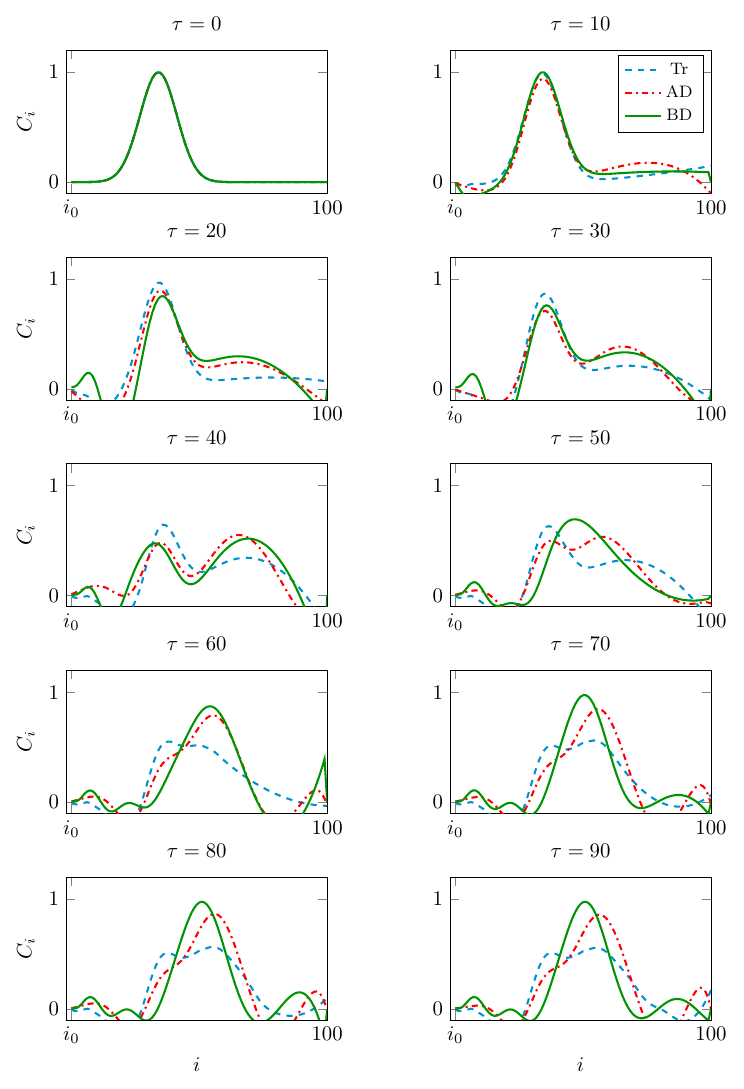}
	\caption{Snapshots of the reconstruction from a noisy second moment observation generated from the Becker-Döring model using the Becker-Döring model (BD, green), the Transport model (Tr, dashed-blue), or the Advection-Diffusion model (AD, dashed-red)}
	\label{fig:estimation-from-noisy-moment-frames}
\end{figure}

\section{Conclusion}

In this article, we addressed a discrete inverse problem, namely the state estimation of  the initial size distribution of polymers from  the observation of the time dynamics of a moment of the size distribution, by proposing an asymptotic approach. We have approximated the discrete inverse problem by two successive size-continuous models, a pure transport equation at first order, a drift-diffusion equation equipped with an absorbing boundary condition at second order. For each of these approximate inverse problems, we derived  model error and estimation error estimates, which highlighted a balance between accuracy of the approximate model and degree of ill-posedness of the inverse problem. Moreover, and this is also illustrated numerically,  the diffusive correction allows one to estimate the initial condition at any positive time, contrarily to the pure transport approximation where the time-window of observation needs to be large enough for the distribution of large sizes to be estimated. Ideally, this fact should appear explicitly in the estimation constants for the second order problem: we expect them to decay for large enough times, namely for $T\geq \f{L}{b},$ uniformly with respect to the diffusion correction factor $\eps.$ This question is linked to the question of the minimal time for uniform controllability of the drift-diffusion equation with vanishing diffusion, first formulated in~\cite{Coron:2005vd} and then further improved by several authors~\cite{amirat2019asymptotic,glass2010complex} and very recently in~\cite{laurent2023uniform}  for the case of a Dirichlet boundary condition.

\medskip

\bibliographystyle{plain}

\end{document}